\documentclass[12pt,twoside]{preprint}
\usepackage{amssymb}
\usepackage{amsmath}
\usepackage{color}
\usepackage{hyperref}
\usepackage{mhenvs}
\usepackage{mhequ}
\usepackage{mhsymb}
\usepackage{booktabs}
\usepackage{tikz}
\usepackage{mathrsfs}
\usepackage{times}
\usepackage{microtype}
\usepackage{comment}

\makeatletter
\pgfdeclareshape{crosscircle}
{
  \inheritsavedanchors[from=circle] 
  \inheritanchorborder[from=circle]
  \inheritanchor[from=circle]{north}
  \inheritanchor[from=circle]{north west}
  \inheritanchor[from=circle]{north east}
  \inheritanchor[from=circle]{center}
  \inheritanchor[from=circle]{west}
  \inheritanchor[from=circle]{east}
  \inheritanchor[from=circle]{mid}
  \inheritanchor[from=circle]{mid west}
  \inheritanchor[from=circle]{mid east}
  \inheritanchor[from=circle]{base}
  \inheritanchor[from=circle]{base west}
  \inheritanchor[from=circle]{base east}
  \inheritanchor[from=circle]{south}
  \inheritanchor[from=circle]{south west}
  \inheritanchor[from=circle]{south east}
  \inheritbackgroundpath[from=circle]
  \foregroundpath{
    \centerpoint%
    \pgf@xc=\pgf@x%
    \pgf@yc=\pgf@y%
    \pgfutil@tempdima=\radius%
    \pgfmathsetlength{\pgf@xb}{\pgfkeysvalueof{/pgf/outer xsep}}%
    \pgfmathsetlength{\pgf@yb}{\pgfkeysvalueof{/pgf/outer ysep}}%
    \ifdim\pgf@xb<\pgf@yb%
      \advance\pgfutil@tempdima by-\pgf@yb%
    \else%
      \advance\pgfutil@tempdima by-\pgf@xb%
    \fi%
    \pgfpathmoveto{\pgfpointadd{\pgfqpoint{\pgf@xc}{\pgf@yc}}{\pgfqpoint{-0.707107\pgfutil@tempdima}{-0.707107\pgfutil@tempdima}}}
    \pgfpathlineto{\pgfpointadd{\pgfqpoint{\pgf@xc}{\pgf@yc}}{\pgfqpoint{0.707107\pgfutil@tempdima}{0.707107\pgfutil@tempdima}}}
    \pgfpathmoveto{\pgfpointadd{\pgfqpoint{\pgf@xc}{\pgf@yc}}{\pgfqpoint{-0.707107\pgfutil@tempdima}{0.707107\pgfutil@tempdima}}}
    \pgfpathlineto{\pgfpointadd{\pgfqpoint{\pgf@xc}{\pgf@yc}}{\pgfqpoint{0.707107\pgfutil@tempdima}{-0.707107\pgfutil@tempdima}}}
  }
}
\makeatother

\colorlet{symbols}{blue!90!black}
\colorlet{testcolor}{green!60!black}

\def\symbol#1{\textcolor{symbols}{#1}}
\def\1{\mathbf{\symbol{1}}}
\def\X{\symbol{X}}

\usetikzlibrary{shapes.misc}
\usetikzlibrary{shapes.symbols}
\usetikzlibrary{snakes}
\usetikzlibrary{decorations}
\usetikzlibrary{decorations.markings}

\def\drawx{\draw[-,solid] (-3pt,-3pt) -- (3pt,3pt);\draw[-,solid] (-3pt,3pt) -- (3pt,-3pt);}
\tikzset{
	root/.style={circle,fill=testcolor,inner sep=0pt, minimum size=2mm},
	dot/.style={circle,fill=black,inner sep=0pt, minimum size=1mm},
	var/.style={circle,fill=black!10,draw=black,inner sep=0pt, minimum size=2mm},
	dotred/.style={circle,fill=black!50,inner sep=0pt, minimum size=2mm},
	generic/.style={semithick,shorten >=1pt,shorten <=1pt},
	dist/.style={ultra thick,draw=testcolor,shorten >=1pt,shorten <=1pt},
	testfcn/.style={ultra thick,testcolor,shorten >=1pt,shorten <=1pt,<-},
	testfcnx/.style={ultra thick,testcolor,shorten >=1pt,shorten <=1pt,<-,
		postaction={decorate,decoration={markings,mark=at position 0.6 with {\drawx}}}},
	kprime/.style={semithick,shorten >=1pt,shorten <=1pt,densely dashed,->},
	kprimex/.style={semithick,shorten >=1pt,shorten <=1pt,densely dashed,->,
		postaction={decorate,decoration={markings,mark=at position 0.4 with {\drawx}}}},
	kernel/.style={semithick,shorten >=1pt,shorten <=1pt,->},
	multx/.style={shorten >=1pt,shorten <=1pt,
		postaction={decorate,decoration={markings,mark=at position 0.5 with {\drawx}}}},
	kernelx/.style={semithick,shorten >=1pt,shorten <=1pt,->,
		postaction={decorate,decoration={markings,mark=at position 0.4 with {\drawx}}}},
	kernel1/.style={->,semithick,shorten >=1pt,shorten <=1pt,postaction={decorate,decoration={markings,mark=at position 0.45 with {\draw[-] (0,-0.1) -- (0,0.1);}}}},
	kernel2/.style={->,semithick,shorten >=1pt,shorten <=1pt,postaction={decorate,decoration={markings,mark=at position 0.45 with {\draw[-] (0.05,-0.1) -- (0.05,0.1);\draw[-] (-0.05,-0.1) -- (-0.05,0.1);}}}},
	kernelBig/.style={semithick,shorten >=1pt,shorten <=1pt,decorate, decoration={zigzag,amplitude=1.5pt,segment length = 3pt,pre length=2pt,post length=2pt}},
	rho/.style={dotted,semithick,shorten >=1pt,shorten <=1pt},
	renorm/.style={shape=circle,fill=white,inner sep=1pt},
	labl/.style={shape=rectangle,fill=white,inner sep=1pt},
	xi/.style={circle,fill=symbols!10,draw=symbols,inner sep=0pt,minimum size=1.2mm},
	xix/.style={crosscircle,fill=symbols!10,draw=symbols,inner sep=0pt,minimum size=1.2mm},
	xib/.style={circle,fill=symbols!10,draw=symbols,inner sep=0pt,minimum size=1.6mm},
	xibx/.style={crosscircle,fill=symbols!10,draw=symbols,inner sep=0pt,minimum size=1.6mm},
	not/.style={circle,fill=symbols,draw=symbols,inner sep=0pt,minimum size=0.5mm},
	>=stealth,
	}
\makeatletter
\def\DeclareSymbol#1#2#3{\expandafter\gdef\csname MH@symb@#1\endcsname{\tikz[baseline=#2,scale=0.15,draw=symbols]{#3}}\expandafter\gdef\csname MH@symb@#1s\endcsname{\scalebox{0.7}{\tikz[baseline=#2,scale=0.15,draw=symbols]{#3}}}}
\def\<#1>{\csname MH@symb@#1\endcsname}
\makeatother

\DeclareSymbol{Xi22}{0.5}{\draw (0,0) node[xi] {} -- (-1,1) node[not] {} -- (0,2) node[xi] {};}

\DeclareSymbol{Xi2}{-2}{\draw (0,-0.25) node[xi] {} -- (-1,1) node[xi] {};}
\DeclareSymbol{Xi3}{0}{\draw (0,0) node[xi] {} -- (-1,1) node[xi] {} -- (0,2) node[xi] {};}
\DeclareSymbol{Xi4}{2}{\draw (0,0) node[xi] {} -- (-1,1) node[xi] {} -- (0,2) node[xi] {} -- (-1,3) node[xi] {};}
\DeclareSymbol{Xi2X}{-2}{\draw (0,-0.25) node[xi] {} -- (-1,1) node[xix] {};}
\DeclareSymbol{XXi2}{-2}{\draw (0,-0.25) node[xix] {} -- (-1,1) node[xi] {};}

\DeclareSymbol{IXi2}{0}{\draw (0,-0.25) node[not] {} -- (-1,1) node[xi] {} -- (0,2) node[xi] {};}
\DeclareSymbol{IXi^2}{-1}{\draw (-1,1) node[xi] {} -- (0,0) node[not] {} -- (1,1) node[xi] {};}

\DeclareSymbol{XiX}{-2.8}{\node[xibx] {};}
\DeclareSymbol{Xi}{-2.8}{\node[xib] {};}
\DeclareSymbol{IXiX}{-1}{\draw (0,-0.25) node[not] {} -- (-1,1) node[xix] {};}

\DeclareSymbol{Xi3b}{-1}{\draw (-1,1) node[xi] {} -- (0,0) node[xi] {} -- (1,1) node[xi] {};}

\DeclareSymbol{IXi3}{2}{\draw (0,-0.25) node[not] {} -- (-1,1) node[xi] {} -- (0,2) node[xi] {} -- (-1,3) node[xi] {};}
\DeclareSymbol{IXi}{-2}{\draw (0,-0.25) node[not] {} -- (-1,1) node[xi] {};}
\DeclareSymbol{XiI}{-2}{\draw (0,-0.25) node[xi] {} -- (-1,1) node[not] {};}

\DeclareSymbol{Xi4b}{-1}{\draw(0,1.5) node[xi] {} -- (0,0); \draw (-1,1) node[xi] {} -- (0,0) node[xi] {} -- (1,1) node[xi] {};}
\DeclareSymbol{Xi4b'}{-1}{\draw(0,1.5) node[xi] {} -- (0,-0.2); \draw (-1,1) node[xi] {} -- (0,-0.2) node[not] {} -- (1,1) node[xi] {};}
\DeclareSymbol{Xi4c}{0}{\draw (0,1) -- (0.8,2.2) node[xi] {};\draw (0,-0.25) node[xi] {} -- (0,1) node[xi] {} -- (-0.8,2.2) node[xi] {};}
\DeclareSymbol{Xi4d}{-4.5}{\draw (0,-1.5) node[not] {} -- (0,0); \draw (-1,1) node[xi] {} -- (0,0) node[xi] {} -- (1,1) node[xi] {};}
\DeclareSymbol{Xi4e}{0}{\draw (0,2) node[xi] {} -- (-1,1) node[xi] {} -- (0,0) node[xi] {} -- (1,1) node[xi] {};}
\DeclareSymbol{Xi4e'}{0}{\draw (0,2) node[xi] {} -- (-1,1) node[xi] {} -- (0,-0.2) node[not] {} -- (1,1) node[xi] {};}

\def\sXi{\symbol{\Xi}}

\newtheorem{assumption}[lemma]{Assumption}

\let\I\CI
\let\J\CJ
\let\M\CM
\let\D\CD
\def\s{\mathfrak{s}}

\def\K{\mathfrak{K}}
\def\DeltaM{\Delta^{\!M}}

\def\Ren{\mathscr{R}}
\def\MM{\mathscr{M}}
\def\TT{\mathscr{T}}
\def\RR{\mathfrak{R}}
\def\GGamma#1{\Gamma_{\!#1}}
\def\bGGamma#1{\bar \Gamma_{\!#1}}
\def\bbar#1{\bar{\bar #1}}
\def\${|\!|\!|}

\begin{document}

\title{A Wong-Zakai theorem for stochastic PDEs}
\author{Martin Hairer$^1$ and \'Etienne Pardoux$^2$}
\institute{University of Warwick, UK, \email{M.Hairer@Warwick.ac.uk} 
\and  Aix-Marseille Universit\'e, CNRS, Centrale Marseille, I2M, UMR 7373 13453 Marseille, France, \email{etienne.pardoux@univ-amu.fr}}

\maketitle

\begin{abstract}
We prove a version of the Wong-Zakai theorem for
one-dimensional parabolic nonlinear stochastic PDEs
driven by space-time white noise. As a corollary, we obtain
 a detailed local description of solutions.\\[.5em]
 
\noindent\textit{Dedicated to the memory of Kiyosi It\^o on the
occasion of the 100th anniversary of his birth.}
\end{abstract}

\setcounter{tocdepth}{2}
\tableofcontents

\section{Introduction}

A series of classical results pioneered 
by Wong and Zakai \cite{WongZakai2,WongZakai} states that 
if $B_\eps$ denotes some ``natural'' smooth $\eps$-approximation to a $d$-dimensional
Brownian motion $B$ (for example piecewise linear approximation or convolution with a mollifier), $g$ and $h$ are smooth functions, 
and $x_\eps$ denotes the solution to the ODE
\begin{equ}[e:SDE]
\dot x_\eps = h(x_\eps) + g(x_\eps)\dot B_\eps\;,
\end{equ}
then $x_\eps$ converges in probability, as $\eps \to 0$, to the solution
to the SDE 
\begin{equ}
dx = h(x)\,dt + g(x)\circ dB\;,
\end{equ}
where $\circ dB$ denotes Stratonovich integration against $B$.

It has been an open problem for some time to obtain an analogous result
in the case of stochastic PDEs driven by space-time white noise of the type
\begin{equ}[e:SPDE]
du = \d_x^2 u\,dt + H(u)\,dt + G(u)\,dW(t)\;.
\end{equ}
The problem is that there is no Stratonovich formulation for such an
equation since the It\^o-Stratonovich correction term would be infinite.
It is formally given by ${1\over 2} G'(u)G(u) \tr Q$ where $tQ$ is the covariance
operator of $W(t)$. In the case of space-time white noise, $Q$ is
the identity operator on $L^2$, which is of course not trace class.
On the other hand, we know that if one subtracts a suitable correction term
from the random ODE \eref{e:SDE}, then it is possible to ensure
that solutions converge to
the It\^o solution. More precisely, if one considers the sequence of
equations given by
\begin{equ}[e:SDEIto]
\dot x_\eps = h(x_\eps) - {1\over 2} Dg(x_\eps) \,g(x_\eps) + g(x_\eps)\dot B_\eps\;,
\end{equ}
then $x_\eps$ converges, as $\eps \to 0$, to the It\^o solution to
$dx = h(x)\,dt + g(x)\,dB$. Since the It\^o solution is the only
``natural'' notion of solution available for \eref{e:SPDE}, this suggests that
if one considers approximations of the type
\begin{equ}[e:SPDEapprox]
\d_t u_\eps = \d_x^2 u_\eps + H(u_\eps) - C_\eps G'(u_\eps)G(u_\eps) + 
G(u_\eps)\,\xi_\eps\;,
\end{equ}
where $\xi_\eps$ denotes an $\eps$-approximation to space-time white
noise and $C_\eps$ is a suitable constant which diverges as $\eps \to 0$,
then one might expect $u_\eps$ to converge to the solution to \eref{e:SPDE},
interpreted in the It\^o sense.

The main result of this article demonstrates that this is \textit{almost} the case,
at least for a large class of Gaussian approximations to space-time white noise,
and on bounded domains. While it is true that the solutions $u_\eps$ to \eref{e:SPDEapprox}
converge to a limit if $C_\eps$ is suitably chosen, it is \textit{not} true in general 
that this limit is given by the It\^o solution to \eref{e:SPDE}. Instead, the limit solves
the same equation, but with a modified drift term $H$. Alternatively, this can be formulated 
as stating that, in order to obtain the limiting equation \eref{e:SPDE}, the ``correct'' approximations
are of the form \eref{e:SPDEapprox}, but with $H$ replaced by some different nonlinearity $\bar H$.
It furthermore turns out that the constants describing $\bar H$ are not universal, but depend
on the details of the regularisation $\xi_\eps$.

In order to formulate our results more precisely, we fix an even (in the sense that $\rho(t,x) = \rho(t,-x)$), smooth, 
compactly supported function $\rho\colon \R^2 \to \R$ with $\int \rho = 1$ and we set 
\begin{equ}[e:crhofirst]
\rho_\eps(t,x) = \eps^{-3} \rho(\eps^{-2}t,\eps^{-1}x)\;,\qquad
c_\rho = \int P(z) (\rho \star \rho)(z)\,dz\;,
\end{equ}
where $P$ denotes the heat kernel on $\R$ and $\star$ denotes space-time convolution.
Here and everywhere in the paper, we use the {\it parabolic scaling}, which means that $t$ is rescaled as if it were 
a 2--dimensional variable. The distance associated to this scaling is defined by 
$|(t,x)|=\sqrt{|t|}+|x|$.
We also define our regularised noise $\xi_\eps$ by
\begin{equ}[e:defxieps]
\xi_\eps(t,x) = \int_{-\infty}^\infty \scal{\rho_\eps(t-s, x-\cdot),dW(s)}\;,
\end{equ}
where $W$ is the cylindrical Wiener process driving \eref{e:SPDE}.
Finally, we will assume from now on that the spatial variable $x$ takes values in the 
one-dimensional torus $S^1$, so that $\scal{\cdot,\cdot}$ denotes the scalar
product in $L^2(S^1)$. 
With these notations at hand, the main result of this article is as follows.

\begin{theorem}\label{theo:main}
Let $H$ and $G$ be of classes $\CC^2$ and $\CC^5$ respectively, both with bounded first derivatives, 
let $u$ denote the solution to \eref{e:SPDE}, and let $u_\eps$ denote the classical solution 
to the random PDE \eref{e:SPDEapprox} with $C_\eps = \eps^{-1} c_\rho$ and $H$ replaced by
\begin{equ}[e:defbarH]
\bar H(u) = H(u) - c_\rho^{(1)} G'(u)^3G(u) - c_\rho^{(2)}G''(u)G'(u)G^2(u)\;,
\end{equ}
for some constants $c_\rho^{(i)}$ possibly depending on $\rho$ but not on $\eps$.
Both solutions are started with the same initial condition $u_\eps(0,\cdot) = u(0,\cdot) \in \CC(S^1)$.

Then, there exists a choice of $c_\rho^{(i)}$ such that, for any $T > 0$, one has
\begin{equ}
\lim_{\eps \to 0} \sup_{(t,x) \in [0,T] \times S^1} |u(t,x) - u_\eps(t,x)| = 0\;,
\end{equ}
in probability. Moreover, for any $\alpha \in (0,{1\over 2})$ and $t > 0$, the 
restriction of $u_\eps$ to
$[t,T] \times S^1$ converges to $u$ in probability for the topology
of $\CC^{\alpha/2,\alpha}$.
\end{theorem}

Before we proceed, let us remark on several possible straightforward (and not so straightforward)
extensions to Theorem~\ref{theo:main}, as well as a few reality-checks.

\begin{remark}
There are explicit formulae for the constants $c_\rho^{(i)}$, see Section~\ref{sec:constants} below.
One important remark is that if we set $\rho(t,x) = \delta^{-1} \tilde \rho(\delta^{-1}t, x)$ for some fixed $\tilde \rho$ and let $\delta \to 0$, then we have $c_\rho^{(i)} \to 0$. 
This is not too
surprising since it intuitively corresponds to first removing the temporal regularisation at fixed
spatial regularisation
(there, traditional generalisations of the Wong-Zakai theorem apply and only an It\^o-Stratonovich correction term appears) and then 
only removing the spatial regularisation, under which we know that \eref{e:SPDE} is 
stable \cite{DPZ}.
\end{remark}

\begin{remark}
As will be clear, our proof automatically yields a situation analogous to that
arising when building a solution theory to finite-dimensional SDEs using the theory
of rough paths \cite{MR2314753} (see also \cite{MR2599193} for an earlier attempt in this direction in the context of evolution equations): one obtains a 
natural notion of solution to \eqref{e:SPDE}
which is pathwise, i.e.\ defined for some set of measure $1$ independent of the
initial condition and even of the choice of nonlinearities $G$ and $H$.
Note however that even if $G$ and $H$ are bounded with bounded derivatives
of all orders, our construction does \textit{not} yield the existence of a stochastic
flow, but only of local solutions. The reason in a nutshell is that 
in order to obtain our convergence result we reformulate
\eqref{e:SPDEapprox} as a fixed point problem in a space of functions of high ``regularity''
(where regularity is to be understood in an unconventional sense though). In this space,
just like in the classical spaces $\CC^\alpha$ for $\alpha > 1$,
the map $u \mapsto G(u)$ is not globally Lipschitz continuous, even if all 
derivatives of $G$ are bounded. As a consequence, we cannot rule out that exceptional realisations
of the driving noise and exceptional initial conditions lead to a finite-time blow-up.
\end{remark}

\begin{remark}
Using the same methodology, one can also treat systems of equations of the type
\begin{equ}
d u_i = \d_x^2 u_i\,dt + H_i(u)\,dt + G_{ij}(u)\,dW_j(t)\;,
\end{equ}
with summation over repeated indices implied and $W_j$ a finite collection of i.i.d.\ cylindrical
Wiener processes. The only difficulty that arises is notational, so we will stick to the case
of one single equation.

In the multi-dimensional  case, it is however not obvious \textit{a priori} how indices are paired in the
correction terms. The ``It\^o-Stratonovich'' correction multiplying $C_\eps$
has the same pairing of indices as in the case of classical SDEs, namely 
$\d_k G_{ij}(u)\, G_{kj}(u)$, with summation over $k$ and $j$ implied. Keeping track of the
indices in \eref{e:RHS} below shows that the additional order $1$
correction term is given by
\begin{equs}[e:corrMultiple]
\bar H_i(u) &= H_i(u) - c_\rho^{(1,1)}\, \d_k G_{ij}(u)\,\d_\ell G_{km}(u)\, \d_n G_{\ell j}(u)\, G_{nm}(u) \\
&\qquad - c_\rho^{(1,2)}\, \d_k G_{ij}(u)\,\d_\ell G_{km}(u)\, \d_n G_{\ell m}(u)\, G_{nj}(u) \\
&\qquad - c_\rho^{(2,1)}\, \d_{k\ell}^2 G_{ij}(u)\,\d_m G_{kn}(u)\,G_{mj}(u)\,G_{\ell n}(u)\\
&\qquad - c_\rho^{(2,2)}\, \d_{k\ell}^2 G_{ij}(u)\,\d_m G_{kj}(u)\,G_{mn}(u)\,G_{\ell n}(u)\;,
\end{equs}
with implied summation over repeated indices. The constants $c_\rho^{(i,j)}$ appearing in this
expression will be given in Section~\ref{sec:constants} and are such that
$c_\rho^{(1)} = c_\rho^{(1,1)} + c_\rho^{(1,2)}$ and similarly for $c_\rho^{(2)}$.
\end{remark}

\begin{remark}
The appearance of a correction term involving second derivatives of $G$ explains why we need $G$
to be at least of class $\CC^3$. The reason why we actually need it to be of class
$\CC^5$ is twofold. First, at this level of ``irregularity'', one could in principle also have 
the appearance of correction terms involving the third derivative of $G$. (See in particular
the right hand side of \eref{e:RHS} below.) In our case however, the corresponding term vanishes
thanks to the specific properties of Gaussian random variables. Second, in order to guarantee that the
expression \eqref{e:RHS} below belongs to $\CD^{\gamma,\eta}$ for some $\gamma > 0$, we need some
control on the next order terms, so that some additional regularity beyond $\CC^4$ is required.
\end{remark}

\begin{remark}
If, in the first expression of \eref{e:defcrho}, we assume that the integration 
variable $z$ is real-valued and 
that $P$ is the Heaviside function (which is nothing but the Green's function for the 
differential operator $\d_t$ on $\R$), then we obtain $c_\rho = {1\over 2}$, independently
of $\rho$, as a consequence of the fact that $\rho^{(2)}$ is symmetric and integrates to $1$. 
This shows that the dominant correction term is indeed compatible with its interpretation
as an It\^o-Stratonovich correction.
\end{remark}

\begin{remark}
In law, the noise $\xi_\eps(t,x)$ is the same as $\eps^{-3/2} \xi_1(\eps^{-2} t, \eps^{-1} x)$.
It is then natural to ask whether a similar Wong-Zakai theorem still holds if 
one sets $\xi_\eps(t,x) = \eps^{-3/2} \xi_1(\eps^{-2} t, \eps^{-1} x)$, but with $\xi_1$ an
arbitrary (non-Gaussian) space-time stationary process with sufficiently rapid mixing properties.
It is at this stage still unclear whether \eref{e:SPDEapprox}
is the correct approximation in this case, as there may be additional correction terms needed.
In particular, it seems plausible that correction terms of order $\eps^{-1/2}$ appear in general.
\end{remark}

\begin{remark}
If one sets $\xi_\eps(t,x) = \eps^{-1} \xi_1(\eps^{-2} t, \eps^{-1} x)$,
then the sequence of solutions to \eref{e:SPDEapprox} with $C_\eps = 0$ converges to a 
deterministic limit, given by the solution to the corrected PDE
\begin{equ}
\d_t u = \d_x^2 u + H(u) + c_\rho G'(u) G(u)\;.
\end{equ}
See \cite{PP12} for a treatment of the linear case.
\end{remark}

\begin{remark}
The results presented in this article
are somewhat orthogonal to the results of \cite{Hendrik} where the
authors considered a more general class of (also possibly vector-valued) equations
of the type
\begin{equ}
du = \d_x^2 u\,dt + H_1(u)\,dt + H_2(u)\d_x u\,dt + G(u)\,dW\;,
\end{equ}
for $W$ an $L^2$-cylindrical Wiener process. Indeed, the
problems tackled in \cite{Hendrik} include a larger class of equations
and allows to consider approximations where the linear operator $\d_x^2$ 
is replaced by a quite general $\eps$-approximation to the Laplacian.
However, the allowable approximations to $W$ included only \textit{spatial}
regularisations (with the integral already treated as an It\^o integral for
$\eps > 0$), while the whole point of the current article is to allow
for \textit{space-time} regularisations. In particular, no additional
correction term appears in \cite{Hendrik}.
\end{remark}

Theorem~\ref{theo:main} is a consequence of a much stronger convergence 
result given in Theorem~\ref{theo:mainFull} below. The statement of this stronger result 
does however first require the introduction of the algebraic / analytical machinery
of regularity structures developed in \cite{Regularity}, so we refrain from stating it here.
As a consequence of this stronger result, one immediately obtains a number
of results that have traditionally been considered out of the scope
of current techniques of proof. For example, it immediately follows that 
solution to \eref{e:SPDE} are locally continuous with respect to their
initial condition in a pathwise sense, rather than just in probability.

\begin{corollary}\label{cor:flow}
Let $H$ be $\CC^2$ and $G$ be $\CC^5$, both with bounded first derivatives. 
Then, for every $u_0 \in \CC(S^1)$, every $T>0$, and every 
$\delta > 0$, there exists
a neighbourhood $\CU$ of $u_0$, an event $\Omega_0 \subset \Omega$ with $\P(\Omega_0) > 1-\delta$
and a map $\Phi\colon \CU \times \Omega \to \CC([0,T] \times S^1)$
such that $\Phi(\cdot,\omega)$ is continuous for every $\omega \in \Omega_0$ and such that 
$\Phi(\bar u_0, \cdot)$ solves \eref{e:SPDE} with initial condition $\bar u_0$ for every
$\bar u_0 \in \CU$.
\end{corollary}

One also obtains the following very sharp regularity result (see for example
the recent article \cite{DavarWeak} for a similar but weaker result):

\begin{corollary}\label{cor:regular}
Let $H$ and $G$ be of class $\CC^2$ and $\CC^5$ respectively, both with bounded first derivatives,
and let $u$ denote the solution to \eref{e:SPDE}. Let furthermore $v$ denote the solution
to the linear equation
\begin{equ}
dv = \d_x^2 v\, dt + dW\;,\qquad  v(0,x) = 0\;.
\end{equ}
Then, for every random space-time point $(t,x) \in \R_+ \times S^1$ and every $\kappa > 0$,
there exists a random constant $C$ such that the bound
\begin{equ}
\bigl|u(s,y) - u(t,x) - G(u(t,x))\bigl(v(s,y) - v(t,x)\bigr)\bigr|
\le C (\sqrt{|t-s|} + |x-y|)^{1-\kappa}\;,
\end{equ}
holds uniformly over $y\in S^1$ and $|t-s| < 1 \wedge |t|/2$.
\end{corollary}

\begin{remark}
Note here that the time $t$ does not need to be a stopping time! This is a
consequence of the pathwise nature of our analysis.
\end{remark}

As a final application of our results, we can show that if the solutions to two stochastic PDEs driven by the 
same realisation of the noise and with the same diffusion coefficient $G$
cross, then they are tangent to even higher order at their intersection.

\begin{corollary}\label{cor:coincide}
Let $H$, $\bar H$ and $G$ be $\CC^2$, $\CC^2$ 
and $\CC^5$ respectively with bounded first derivatives and let $u$ and $\bar u$ denote
the solutions to \eref{e:SPDE}, as well as the same equation with $H$
replaced by $\bar H$. Let $(t,x) \in \R_+\times S^1$ be a random space-time point
such that $u(t,x) = \bar u(t,x)$ and let $\kappa > 0$.

Then, there exist random constants $C$ and $D$ such that
\begin{equ}
\bigl|u(s,y) - \bar u(s,y) - D(y-x)\bigr|
\le C (|t-s| + |x-y|^2)^{{3\over 4} - \kappa}\;,
\end{equ}
holds uniformly for $(s,y)$ in a compact subset of $\R_+ \times S^1$. 
(With $C$ depending on the set.) If one has $H = \bar H$ and one
knows furthermore that $u(s,y) \ge \bar u(s,y)$ in 
a region of the type $\{(s,y)\,:\, |x-y| \le \delta \;\&\; s \in (t-\delta,t]\}$ for some
$\delta > 0$, then this bound improves to 
\begin{equ}
\bigl|u(s,y) - \bar u(s,y) - D_1(t-s) - D_2(y-x)^2\bigr|
\le C (|t-s| + |x-y|^2)^{{5\over 4} - \kappa}\;,
\end{equ}
and furthermore $D_1, D_2 \ge 0$.
\end{corollary}

\begin{remark}
The assumption that $H$ and $G$ have bounded first derivatives ensures the existence of 
global solutions to \eref{e:SPDE}, which makes the statement easier to formulate. It is straightforward to localise the statements
to cover situations with possible blow-ups.
\end{remark}

\begin{remark}
The main result of this paper, in the particular case $H\equiv0$ and $G(u)=u$, is equivalent to 
the convergence to its Hopf--Cole solution of the regularised KPZ equation. Indeed, consider equation \eqref{e:SPDEapprox} with $H(u)=-c^{(1)}_\varrho \,u$ and $G(u)=u$, and with an initial condition which satisfies
$u_\varepsilon(0,x)>0$, for all $x\in S^1$, which implies that the same is true for $u_\varepsilon(t,x)$ for all $t>0, x\in S^1$. Let now 
$Z_\varepsilon(t,x)=\log[u_\varepsilon(t,x)]$. Dividing  \eqref{e:SPDEapprox} by $u_\epsilon$, we obtain that 
$$\partial_tZ_\varepsilon=\partial^2_xZ_\varepsilon+(\partial_x Z_\varepsilon)^2-c^{(1)}_\varrho-C_\varepsilon+\xi_\varepsilon.$$
A consequence of Theorem \ref{theo:main} is that $Z_\epsilon(t,x)\to Z(t,x):=\log u(t,x)$,
where $u$ solves the SPDE 
$$du=\partial^2_x u\, dt+u\, dW(t)$$
in the It\^o sense, i.e. $Z$ is the Hopf--Cole solution of the KPZ equation. 
This result was implicit in \cite{KPZ,Regularity} but not formulated there in detail.
It is also used in \cite{KPZJeremy} to identify the limit for a class of weakly asymmetric interface
fluctuation models as being the Hopf--Cole solution of the KPZ equation.
\end{remark}

\subsection*{Acknowledgements}

{\small
MH gratefully acknowledges financial support from the Philip Leverhulme Trust, the Royal Society, and the ERC.
}

\section{Values of constants}
\label{sec:constants}

In this section, we give explicit expressions for $c_\rho^{(1)}$ and
$c_\rho^{(2)}$ and we discuss the limiting behaviours of these constants under scalings of the 
type $\rho = \delta^{-1}\tilde \rho(\delta^{-1}t,x)$ and $\rho = \delta^{-1}\tilde \rho(t,\delta^{-1}x)$.
This allows us to also consider different scaling behaviours for the noise.

Denoting by $P$ the heat kernel (with the convention that $P(t,x) = 0$ for $t \le 0$), 
by $\delta$ the Dirac at the origin, and
 using the shorthand $\rho^{(2)} = \rho \star \rho$, the constants $c_\rho^{(i)}$ are given by the expression
\begin{equ}
c_\rho^{(i)} = c_\rho^{(i,1)} + c_\rho^{(i,2)}\;,
\end{equ}
with
\begin{equs}
c_\rho^{(1,1)} &= \int P(z_1)P(z_2)P(z_3) \,\rho^{(2)}(z_1 + z_2)  \rho^{(2)}(z_2 + z_3)\,\prod_{i=1}^3 dz_i \;,\\
c_\rho^{(1,2)} &= \int P(z_1)P(z_2) \bigl(P(z_3)\,\rho^{(2)}(z_3) - c_\rho \delta(z_3)\bigr) \rho^{(2)}(z_1 + z_2 + z_3)\,\prod_{i=1}^3 dz_i\;,\\
c_\rho^{(2,1)} &= \int P(z_1)P(z_2) P(z_3)\,\rho^{(2)}(z_1 + z_2) \rho^{(2)}(z_2 - z_3)\,\prod_{i=1}^3 dz_i \;, \label{e:defcrho}\\
c_\rho^{(2,2)} &= \int P(z_1)P(z_2) \bigl(P(z_3)\,\rho^{(2)}(z_3) - c_\rho \delta(z_3)\bigr) \rho^{(2)}(z_1 - z_2 + z_3)\,\prod_{i=1}^3 dz_i\;.
\end{equs}
Here, the integration variables
$z_i$ are space-time variables and are integrated over all of $\R^2$. More concisely,
if we draw a node \tikz[baseline=-3] \node [dot] {}; for each integration variable, a special node \tikz[baseline=-3] \node [root] {}; for the origin, 
an arrow for the heat kernel (evaluated at the difference between the two 
variables that it connects), and a dotted line for $\rho^{(2)}$ (also 
evaluated at the difference between the two 
variables that it connects), then \eref{e:defcrho} can be
rewritten much more concisely and suggestively as
\begin{equs}[e:defrhopicture]
c_\rho = 
\begin{tikzpicture}[baseline=10,scale=0.5]
\node at (0,2) [dot] (1) {};
\node at (0,0) [root] (0) {};

\draw[rho] (1) to[bend left=60] (0);
\draw[kernel] (1) to[bend right=60] (0);
\end{tikzpicture}\;,\quad
c_\rho^{(1,1)} &= 
\begin{tikzpicture}[baseline=10,scale=0.5]
\node at (1,2) [dot] (1) {};
\node at (-1,2) [dot] (2) {};
\node at (1,0) [dot] (3) {};
\node at (-1,0) [root] (4) {};

\draw[kernel] (1) to (2);
\draw[kernel] (2) to (3);
\draw[kernel] (3) to (4);

\draw[rho] (1) to (3);
\draw[rho] (2) to (4);
\end{tikzpicture}\;, \quad
c_\rho^{(1,2)} = 
\begin{tikzpicture}[baseline=10,scale=0.5]
\node at (-1,2) [dot] (1) {};
\node at (1,2) [dot] (2) {};
\node at (-1,0) [root] (3) {};
\node at (1,0) [dot] (4) {};

\draw[kernel] (1) to (2);
\draw[kernel] (4) to (3);
\draw[kernel,bend right=60] (2) to (4);
\draw[rho,bend left=60] (2) to (4);

\node at (1,1) {$\Ren$};

\draw[rho] (1) to (3);
\end{tikzpicture}\;,\\
c_\rho^{(2,1)} &= 
\begin{tikzpicture}[baseline=10,scale=0.5]
\node at (1,2) [dot] (1) {};
\node at (-1,2) [dot] (2) {};
\node at (1,0) [dot] (3) {};
\node at (-1,0) [root] (4) {};

\draw[kernel] (1) to (2);
\draw[kernel] (2) to (3);
\draw[kernel] (4) to (3);

\draw[rho] (1) to (3);
\draw[rho] (2) to (4);
\end{tikzpicture}\;,\quad
c_\rho^{(2,2)} = 
\begin{tikzpicture}[baseline=10,scale=0.5]
\node at (-1,2) [dot] (1) {};
\node at (1,2) [dot] (2) {};
\node at (-1,0) [root] (3) {};
\node at (1,0) [dot] (4) {};

\draw[kernel] (1) to (2);
\draw[kernel] (3) to (4);
\draw[kernel,bend right=60] (2) to (4);
\draw[rho,bend left=60] (2) to (4);

\node at (1,1) {$\Ren$};

\draw[rho] (1) to (3);
\end{tikzpicture}\;,
\end{equs}
where the symbol $\Ren$ inside a loop indicates that it was ``renormalised'' by subtracting a 
delta-function with weight identical to the integral of the kernel represented by the loop.
In this case, the kernel in question is $P \rho^{(2)}$ and its integral is precisely $c_\rho$.
This kind of notation will be used extensively in Section~\ref{sec:convergence} below, so we urge the reader
to familiarise themselves with the link between \eref{e:defrhopicture} and \eref{e:defcrho}.
Note also that in this particular case, the location on the graph of the distinguished 
variable \tikz[baseline=-3] \node [root] {}; representing the origin does not matter.

The fact that the integrals depicted by \eref{e:defrhopicture} 
converge (and the precise meaning of the ones involving ``renormalisation'') is not 
completely obvious, so we will discuss this now.

\subsection{Convergence of the renormalisation constants}
\label{sec:convConst}

The convergence of the integral defining $c_\rho$ is very easy to verify, so we focus
on the remaining four terms.
There are three issues with the definitions \eref{e:defrhopicture}: integrability of the
expressions at short scales, integrability at large scales, and the meaning of the ``renormalised''
integrals. Regarding integrability at small scales, it is very easy to convince oneself that 
this is not a problem, thanks to the relatively mild singularity of the heat kernel and the
boundedness of the function $\rho^{(2)}$. 

Regarding the integrability at large scales, we first consider the terms $c_\rho^{(i,1)}$.
For these terms, we see that thanks to the fact that $P(t,x) = 0$ for $t < 0$ and $\rho^{(2)}$
is compactly supported, the time variables take values in a compact domain. Pictorially, if
two nodes are connected by $\rho^{(2)}$ like this \tikz[baseline=-3] \draw[rho] (0,0) node [dot] {} -- (1,0) node [dot] {};, then the corresponding time variables can be separated by at most a fixed finite distance. Furthermore, for any configuration of the type \begin{tikzpicture}[baseline=-3] \node at (0,0) [dot] (1) {};\node at (1,0) [dot] (2) {};\node at (2,0) [dot] (3) {}; \draw[kernel] (1) -- (2);\draw[kernel] (2) -- (3);\end{tikzpicture} (here, the direction of the arrows does matter), the time 
coordinate of the middle variables has to lie between the time coordinates of the other two.
Since, for any fixed time, the heat kernel decays exponentially fast in the space variable, this 
shows that the expressions for the constants $c_\rho^{(i,1)}$ are well-defined.

Regarding $c_\rho^{(1,2)}$, we can simply rewrite it as
\begin{equ}
c_\rho^{(1,2)} = \begin{tikzpicture}[baseline=10,scale=0.5]
\node at (-1,2) [dot] (1) {};
\node at (1,2) [dot] (2) {};
\node at (-1,0) [root] (3) {};
\node at (1,0) [dot] (4) {};

\draw[kernel] (1) to (2);
\draw[kernel] (4) to (3);
\draw[kernel,bend right=60] (2) to (4);
\draw[rho,bend left=60] (2) to (4);

\draw[rho] (1) to (3);
\end{tikzpicture}\;-\;
c_\rho \;
\begin{tikzpicture}[baseline=10,scale=0.5]
\node at (-1,2) [dot] (1) {};
\node at (1,1) [dot] (2) {};
\node at (-1,0) [root] (3) {};

\draw[kernel] (1) to (2);
\draw[kernel] (2) to (3);

\draw[rho] (1) to (3);
\end{tikzpicture}
\end{equ}
Using the same considerations as those used to bound $c_\rho^{(i,1)}$, we conclude that
each of these two terms converges separately.

The constant $c_\rho^{(2,2)}$ is a little bit more delicate to bound. Indeed, if we decompose
it as we did for $c_\rho^{(1,2)}$, then we see that each of these terms taken individually 
diverges. This can be easily seen for the second term, since
this is given by $\int \rho^{(2)} (z_1 - z_2) P(z_1)P(z_2)\,dz_1\,dz_2$ and $|P(z)|^2$ decays like 
$|z|^{-2}$ which is far from being integrable at large scales. (Recall that
we endow $\R^2$ with the parabolic distance which, for $z = (t,x)$, is given by
$|z| = |x| + \sqrt{|t|}$. When endowed with the parabolic scaling, the scaling dimension
of $\R^2$ is equal to $3$, so we are one whole power short of being integrable, and not just 
borderline.) Denote now by $Q$ the distribution
formally given by $Q(z) = P(z)\rho^{(2)}(z) - c_\rho\,\delta_0(z)$, then $c_\rho^{(2,2)}$
should be interpreted as
\begin{equ}
c_\rho^{(2,2)} = \int \rho^{(2)} (z_1 - z_2) (P \star Q)(z_1)\,P(z_2)\,dz_1\,dz_2\;.
\end{equ}
A finite bound on this expression then follows immediately from the following very 
simple result.

\begin{lemma}
There exists a constant $C$ such that the  function $P\star Q$ is bounded
by $C \bigl(|z|^{-1} \wedge |z|^{-3}\bigr)$, uniformly in $z$.
\end{lemma}

\begin{proof}
We use the the fact that $Q$ has compact support 
 (since $\varrho$, hence also $\varrho^{(2)}$, has this property), integrates to $0$
and is even in the spatial variable. As a consequence, it annihilates every function
of the form $(t,x) \mapsto a+bx$ with $a,b \in \R$.
The claim then follows immediately from \cite[Prop.~A.1]{Regularity} (applying 
it with $A = \{(0,0), (0,1)\}$), combined with the
scaling properties of $P$.
\end{proof}

\subsection{Non-parabolic scalings: spatial regularisation}

The case of purely spatial regularisation formally corresponds to
$\rho(t,x) = \delta_0(t) \bar \rho(x)$ (with $\delta_0$ the Dirac distribution
centred at the origin). In general, one could consider a regularisation
of the noise of the type
$\xi_\eps = \eps^{-(1+\beta)/2} \xi_1(\eps^{-\beta} t, \eps^{-1} x)$ for some
$\beta > 2$, which corresponds to the case where we consider approximations that 
are ``more regular'' in space than in time. We can recast this in the previous 
framework, but now $\rho$ (or rather $\rho^{(2)}$ since this is all that ever matters
for the calculation of our constants) is replaced by
\begin{equ}[e:spatialreg]
\rho_\delta^{(2)}(t,x) = \delta^{-1} \rho^{(2)}(\delta^{-1} t,x)\;,\qquad \delta = \eps^{\beta-2}\;.
\end{equ}
With such a choice, we then obtain
\begin{equ}
c_\rho = c_\rho(\delta) = \int P(z)\, \rho_\delta^{(2)}(z) \,dz \to c_\rho(0) = \int_0^\infty \rho^{(2)}(t,0)\,dt\;,
\end{equ}
which is consistent with the results obtained in \cite{PP12,HPP13}. 

We claim that the constants $c_\rho^{(i,j)}$ converge to $0$ as $\delta \to 0$ if the
regularisation $\rho^{(2)}$ is of the form \eref{e:spatialreg}.
Regarding the constants $c_\rho^{(i,2)}$, it is straightforward to see that, weakly, one
has $P \rho^{(2)}_\delta \to c_\rho(0) \,\delta_0$, so that these two constants do indeed converge to $0$
as $\delta \to 0$. Regarding the constants $c_\rho^{(i,1)}$, we obtain an upper bound
on their absolute values by replacing $\rho^{(2)}_\delta(t,x)$ by the function 
$\delta^{-1} \one_{|t| < \delta}(t,x)$.
For any fixed values of the three time variables, the integration over the spatial variables is
then bounded by $\delta^{-2}$ as a simple consequence of the fact that $\int P(t,x)\,dx = 1$ for every $t$.
(The factor $\delta^{-2}$ comes from the two factors $\rho^{(2)}_\delta$.)
Since the time variables are bounded by some multiple of $\delta$ and there are three time integrals,
we conclude that $|c_\rho^{(i,1})| \le C \delta$ for some $C$, so that they also converge to $0$.
This is also consistent with \cite[Thm~1.11]{BMS}, where the authors obtain
a convergence result similar in spirit to our Theorem~\ref{theo:main}, but with $\xi_\eps$ given by
a regularisation of space-time white noise at scale $\eps$ in space and scale $\exp(-c/\eps)$ 
(which is of course \textit{much} smaller than $\eps^2$) in time.

\subsection{Non-parabolic scalings: temporal regularisation}

The case of temporal regularisation is much less obvious, and this was already the ``difficult case'' in
\cite{PP12,HPP13}. 
This time, we consider a regularisation
of the noise of the type
$\xi_\eps = \eps^{-1+{\beta\over 2}} \xi_1(\eps^{-2} t, \eps^{-\beta} x)$ for some
$\beta > 1$.
Again, we can recast this into our 
framework, with $\rho^{(2)}$ replaced by
\begin{equ}
\rho_\delta^{(2)}(t,x) = \delta^{-1} \rho^{(2)}(t,\delta^{-1}x)\;,\qquad \delta = \eps^{\beta-1}\;.
\end{equ}
With this choice it is then straightforward to verify that one has
\begin{equ}
c_\rho = \int P(z)\, \rho_\delta^{(2)}(z) \,dz \to \int_0^\infty {\bar \rho(t) \over \sqrt{ 4\pi t}}\,dt\;,
\end{equ}
where we used the shorthand notation $\bar \rho(t) = \int_\R \rho^{(2)}(t,x)\,dx$.
This is due to the fact that $\rho_\delta^{(2)}$ degenerates to a Dirac delta-function in the
spatial variable and $P(t,0) = 1/\sqrt{4\pi t}$. 
Inspection of \eref{e:defcrho}, combined with the fact that
the heat kernel $P$ satisfies the identities
\begin{equs}
\int_\R P_s(x)P_t(x)\,dx &= {1\over \sqrt{4\pi(s+t)}}\;,\\
\int_\R P_s(x)P_t(x)P_u(x)\,dx &= {1\over 4\pi\sqrt{st+tu+su}}\;,
\end{equs}
yields furthermore the expressions
\begin{equs}
c_\rho^{(1)} &= {1\over 4\pi}\int_{\R_+^3} {\bar \rho(t+s)\bar \rho(t+u)\over \sqrt{st+tu+su}}\,ds\,dt\,du\\
&\qquad + {1\over 4\pi}\int_{\R_+^3} {\bar \rho(t+u+s)\over \sqrt{t+s}} \Bigl({\bar \rho(u) \over \sqrt u} - c_\rho\,\delta(u)\Bigr)\,du\,ds\,dt\;,\\
c_\rho^{(2)} &= {1\over 4\pi}\int_{\R_+^3} {\bar \rho(t-s)\bar \rho(t+u)\over \sqrt{st+tu+su}}\,ds\,dt\,du\\
&\qquad + {1\over 4\pi}\int_{\R_+^3} {\bar \rho(t+u-s)\over \sqrt{t+s}} \Bigl({\bar \rho(u) \over \sqrt u} - c_\rho\,\delta(u)\Bigr)\,du\,ds\,dt\;,
\end{equs}
for the constants governing the remaining two correction terms.

\section{Regularity structures}

In order to prove the type of convergence result mentioned in the 
introduction, we make use of the theory of regularity structures 
developed in \cite{Regularity}. A complete self-contained exposition of
the theory is of course beyond the scope of this article, so we will
content ourselves with a short summary of the theory's main concepts 
and results, when specialised to the specific example of the class \eref{e:SPDE}.
For a more concise exposition of the general theory, see also the lecture 
notes \cite{Notes}, as well as \cite[Chapter~15]{Book}.

The main ingredient of the theory is that of a \textit{regularity structure}.
This consists first of a graded vector space $\CT = \bigoplus_{\alpha \in A} \CT_\alpha$
where $A$ denotes a set of real-valued indices (called homogeneities) that is 
locally finite and bounded from below. In our specific situation, each 
of the spaces $\CT_\alpha$ is finite-dimensional
and comes with a distinguished canonical basis. The space $\CT$ also comes
endowed with a group $\CG$ of continuous linear transformations of $\CT$ with the property that,
for every $\Gamma \in \CG$, every $\alpha \in A$, and every $\tau \in \CT_\alpha$ one has
\begin{equ}[e:basicRel]
\Gamma \tau - \tau \in \bigoplus_{\beta < \alpha} \CT_\beta\;.
\end{equ}
The canonical example to keep in mind is the space $\bar \CT = \bigoplus_{n \in \N} \bar \CT_n$ 
of abstract polynomials in
finitely many indeterminates, with $A = \N$ and $\bar \CT_n$ denoting the space of monomials
that are homogeneous of degree $n$. In this case, a natural group of transformations $\CG$ acting
on $\bar \CT$ is given by the group of translations, which does indeed satisfy \eref{e:basicRel}.

\subsection{Specific regularity structure}
\label{sec:regStruc}

The regularity structure that is relevant for the analysis of \eref{e:SPDE} is built in the
following way. First, we start with the regularity structure $\bar \CT$ given by all polynomials in
two indeterminates, let us call them $X_1$ and $X_0$, which denote the space and time
directions respectively. We do however endow these with the usual parabolic space-time 
scaling so that each factor of the ``time'' variable $X_0$ increases the homogeneity by $2$.
In particular, one has $\one \in \bar \CT_0$,  $X_1 \in \bar \CT_1$, $X_0 \in \bar \CT_2$, $X_1^2 \in \bar \CT_2$, etc.

We then introduce two additional symbols, $\Xi$ and $\CI$, which will be interpreted 
as an abstract representation of the driving noise $\xi$ and of the operation of convolution
with the heat kernel respectively. Fixing some (sufficiently small in the sequel) exponent $\kappa > 0$,
we then define $\CT_{-{3\over 2} - \kappa}$ as the copy of $\R$ with unit vector denoted by $\Xi$
and we postulate that if $\tau$ is some formal expression with homogeneity $|\tau| = \alpha$, then
$\CI(\tau)$ is a new formal expression with homogeneity $|\CI(\tau)| = \alpha + 2$.
Furthermore, we also postulate that $\CI(X^k) = 0$ for every multiindex $k$.
(Here, for $k = (k_0,k_1)$, we have used the shorthand $X^k = X_0^{k_0}X_1^{k_1}$.)
See Section~\ref{sec:models} below, and in particular \eref{e:killPoly}, for an interpretation
of this fact.
Furthermore, if $\tau, \bar \tau$ are formal expressions with respective homogeneities
$\alpha, \bar \alpha$, then $\tau \bar \tau = \bar \tau \tau$ is postulated to be a new formal expression with
homogeneity $\alpha + \bar \alpha$. 

A few examples of formal expression with their respective homogeneities
that can in principle be built in this way are given by
\begin{equ}[e:examples]
|X_0\Xi| = {1\over 2}-\kappa\;,\quad 
|\Xi\CI(\Xi)| = -1-2\kappa\;,\quad 
|\Xi^2\CI(\Xi)^2| = -2-4\kappa\;. 
\end{equ}
In order to define our regularity structure $\CT$, we do not keep all of these formal expressions,
but only those that are actually useful for the abstract reformulation of \eref{e:SPDE}.
More precisely, we consider a collection $\CU$ of formal expressions which is the
smallest collection containing $\one$, $X_0$, and $X_1$, and such that 
\begin{equs}[2]
\tau \in \CU \quad&\Rightarrow\quad& \CI(\tau) &\in \CU\;,\\
\tau \in \CU \quad&\Rightarrow\quad& \CI(\Xi\tau) &\in \CU\;,\label{e:propU}\\
\tau, \bar \tau \in \CU \quad&\Rightarrow\quad& \tau \bar \tau &\in \CU\;.
\end{equs}
We then set 
\begin{equ}
\CW = \CU \cup \{\Xi \tau\,:\, \tau \in \CU\}\;,
\end{equ}
and we define $\CT$ as the set of all linear combinations of elements in $\CW$. 
Naturally, $\CT_\alpha$ consists of those linear combinations that only involve
elements in $\CW$ that are of homogeneity $\alpha$. Furthermore, $\CW$ is the
previously announced set of canonical basis elements of $\CT$.
In particular, $\CT$ contains the first
two expressions of \eref{e:examples}, but not the last one. It follows furthermore
from \cite[Lemma~8.10]{Regularity} that, for every $\alpha \in \R$, $\CW$ contains only
finitely many elements of homogeneity less than $\alpha$, so that each $\CT_\alpha$
is finite-dimensional and $A \cap (-\infty,\alpha]$ is finite.

We also decompose $\CT$ into a direct sum as
\begin{equ}[e:Tdirectsum]
\CT = \CT_\Xi \oplus \CT_\CU\;,
\end{equ}
where $\CT_\CU$ is the linear span of $\CU$ and $\CT_\Xi$ is the linear span of $\Xi \CU$, which 
are all those symbols in $\CW$ containing a factor $\Xi$.

\subsection{Structure group}

Let us now describe the structure group $\CG$ associated to the space $\CT$. For this, we 
first introduce $\CT_+$, the free commutative algebra generated by the formal expressions
\begin{equ}[e:genT+]
\CW_+ \eqdef \bigl\{X_0, X_1\bigr\}\cup \bigl\{\CJ_k(\tau)\,:\, \tau \in \CW \setminus \bar \CT \;, \; |k| < |\tau| + 2\}\;,
\end{equ}
where $k$ is an arbitrary $2$-dimensional multiindex and $|k|$ denotes its
``parabolic length'', i.e.
\begin{equ}
|k| = 2k_0 + k_1\;.
\end{equ}
In other words, $\CT_+$ consists of all linear combinations of products of formal expressions 
in $\CW_+$. We will view $\CJ_k$ as a map from $\CT$ into $\CT_+$
by postulating that it acts linearly on $\CT$ and that $\CJ_k(\tau) = 0$ for those formal expressions
$\tau$ for which $|\tau| + 2 \le |k|$ or $\tau \in \bar \CT$. 
Note that for the moment, elements of $\CT_+$ are formal objects. They will be used later on to index
matrix elements for some useful linear transformations on $\CT$. We will give explicit formulae for
the relations between these formal expressions and the numbers they represent in Section~\ref{sec:models}.

With this definition at hand, we construct a linear map $\Delta \colon \CT \to \CT\otimes \CT_+$ in a recursive way. 
In order to streamline notations, we shall write 
$\tau^{(1)}\otimes\tau^{(2)}$ as a shorthand for $\Delta\tau$.
(This is an abuse of notation, following Sweedler, since in general
$\Delta\tau$ is a linear combination of such terms.) We then define $\Delta$ via the identities
\begin{equ}[e:Delta1]
\Delta\one=\one\otimes\one\;,\qquad
\Delta\Xi= \Xi\otimes\one\;,\qquad
\Delta X_i= X_i\otimes\one+\one\otimes X_i\;,
\end{equ}
and then recursively by the following relations:
\begin{equs}[e:Delta2]
\Delta \tau\overline{\tau}&= \tau^{(1)}\overline{\tau}^{(1)}\otimes \tau^{(2)}\overline{\tau}^{(2)}\;,\\
\Delta\I(\tau)&=\I(\tau^{(1)})\otimes \tau^{(2)}+\sum_{\ell,k}\frac{X^\ell}{\ell!}\otimes\frac{X^k}{k!}
\J_{\ell+k}(\tau)\;.
\end{equs}
For any linear functional $f \colon \CT_+ \to \R$, we can now define in a natural way
a map $\Gamma_{\!f} \colon \CT \to \CT$ by
\begin{equ}
\Gamma_{\!f} \tau = (I \otimes f)\Delta \tau\;.
\end{equ}
Let now $\CG_+$ denote the set of all such linear functionals $f$ which are multiplicative in the sense that 
$f(\tau \bar \tau) = f(\tau)f(\bar \tau)$ for any two elements $\tau, \bar \tau \in \CT_+$. With this definition
at hand, we set
\begin{equ}
\CG = \{\Gamma_{\! f}\,:\, f \in \CG_+\}\;.
\end{equ}
It is not difficult to see that these operators satisfy the property \eref{e:basicRel}, but it is
a highly non-trivial fact that the set $\CG$ of these linear operators does indeed form a group
under composition, see \cite[Sec.~8.1]{Regularity}.

\begin{remark}\label{rem:cutoff}
As a matter of fact, we will never need to consider the full space $\CT$ as defined above,
but it will be sufficient to consider the subspace generated by all elements of homogeneity less
than some  large enough number $\zeta$. In practice, it actually turns out to be 
sufficient to choose any $\zeta > {3\over 2} + \kappa$, except when proving Corollary~\ref{cor:coincide} for which we
require $\zeta = {5/2}$.
\end{remark}

\begin{remark}\label{rem:prod}
As a consequence of \eqref{e:Delta1}, \eqref{e:Delta2}, and the fact that linear functionals in $\CG_+$ are multiplicative,
one has $\Gamma \Xi = \Xi$ and $\Gamma(\tau \bar \tau) = (\Gamma \tau)(\Gamma \bar \tau)$ for every $\Gamma \in \CG$
and every $\tau, \bar \tau \in \CT$ such that the product $\tau \bar \tau$ also belongs to $\CT$.
\end{remark}

\subsection{Models}\label{sec:models}

Now that we have fixed our algebraic regularity structure $(\CT,\CG)$, we introduce a family of analytical 
objects associated to it that will play the role of Taylor polynomials in our theory in order to
allow us to describe solutions to \eref{e:SPDE} locally, up to arbitrarily high order, despite the fact
that they are not smooth in the conventional sense.

From now on, we also fix a value $\zeta \ge 2$ as in Remark~\ref{rem:cutoff} and we
set $\CT = \bigoplus_{\alpha \in A\,:\, \alpha \le \zeta} \CT_\alpha$. This also has the advantage that
$\CT$ itself is finite-dimensional so we do not need to worry about topologies.
In order to describe our ``polynomial-like'' objects, 
we first fix a kernel $K \colon \R^2 \to \R$ with the following properties:
\begin{enumerate}
\item The kernel $K$ is compactly supported in $\{x^2 + |t| \le 1\}$.
\item One has $K(t,x) = 0$ for $t \le 0$ and $K(t,-x) = K(t,x)$.
\item For $(t,x)$ with $x^2 + t < 1/2$ and $t > 0$, one has
\begin{equ}
K(t,x) = {1\over \sqrt{4\pi t}} e^{-{x^2 \over 4t}}\;,
\end{equ}
and $K$ is smooth on $\{x^2 + |t| \ge 1/4\}$.
\item For every polynomial $Q \colon \R^2 \to \R$ of parabolic degree less than $\zeta$, one has
\begin{equ}[e:killPoly]
\int_{\R^2} K(t,x) Q(t,x)\,dx\,dt = 0\;.
\end{equ}
\end{enumerate}
in other words, $K$ has essentially all the properties of the heat kernel $P$, except that it is furthermore
compactly supported and satisfies \eref{e:killPoly}. The constants $1/2$ and $1/4$ appearing in
the third point are of course completely arbitrary as long as they are strictly between $0$ and $1$.
The existence of a kernel $K$ satisfying these properties is very easy to show.

We now denote by $\CS'$ the space of Schwartz distributions on $\R^2$ and by $\CL(\CT,\CS')$ the space
of (necessarily continuous) linear maps from $\CT$ to $\CS'$. Furthermore, given
a continuous test function $\phi\colon \R^2 \to \R$ and a point $z = (t,x) \in \R^2$, we set
\begin{equ}
\phi_z^\lambda(\bar z) = \lambda^{-3} \phi\bigl((\lambda^{-2}(\bar t - t), \lambda^{-1}(\bar x - x)\bigr)\;,
\end{equ}
where we also used the shorthand $\bar z = (\bar t, \bar x)$. Finally, we write $\CB$ for the set
of functions $\phi \colon \R^2 \to \R$ that are smooth, compactly supported in the ball of radius one,
and with their values and both first and second derivatives bounded by $1$.

Given a kernel $K$ as above, we then introduce a set $\MM$ of \textit{admissible models} which
are analytical objects built upon our regularity structure $(\CT,\CG)$ that will play a role for our solutions that is
similar to that of the usual Taylor polynomials for smooth functions.
An admissible model consists of a pair $(\Pi,F)$ of functions
\begin{equs}[2]
\Pi \colon \R^2 &\to \CL(\CT,\CS') \quad & \quad F \colon \R^2  &\to \CG \\
z &\mapsto \Pi_z & z &\mapsto F_z 
\end{equs}
with the following properties. First, writing $\gamma_{z\bar z} \in \CG_+$ for the element
such that $F_z^{-1} \circ F_{\bar z}  = \GGamma{\gamma_{z\bar z}}$, we impose that
they satisfy the analytical bounds
\begin{equ}[e:bounds]
\bigl|\bigl(\Pi_z \tau\bigr)(\phi_z^\lambda)\bigr| \lesssim \lambda^{|\tau|}\;,\qquad
\bigl|\gamma_{z\bar z}(\bar \tau)\bigr| \lesssim |z-\bar z|^{|\bar \tau|}\;,
\end{equ}
uniformly over $\phi \in \CB$, $\lambda \in (0,1]$, $\tau \in \CW$, and $\bar \tau \in \CW_+$. 
Here, with the same shorthand as before,
we set
\begin{equ}
|z-\bar z| = |x-\bar x| + \sqrt{|t-\bar t|}\;.
\end{equ}
Also, the proportionality constants implicit in the notation $\lesssim$ are assumed to be bounded
uniformly for $z$ and $\bar z$ taking values in any compact set. We furthermore assume
that one has the algebraic identity
\begin{equ}[e:algebraic]
\Pi_{z} F_z^{-1} = \Pi_{\bar z} F_{\bar z}^{-1}\;,
\end{equ}
valid for every $z, \bar z$ in $\R^2$. 

\begin{remark}
We will write $f_z$ for the element in $\CG_+$ such that $F_z = \Gamma_{f_z}$ and we will provide
explicit expressions for $f_z$. We will also
use interchangeably the notations $(\Pi,F)$ and $(\Pi, f)$ for the model. Note also that 
there is an explicit formula for the bilinear map giving $\gamma_{z\bar z}$ in terms of
$f_z$ and $f_{\bar z}$, but its expression is somewhat complicated and of no particular use
for this article. See \cite[Section~8.1]{Regularity} for more details.
\end{remark}

Finally, and this is why our models are called
\textit{admissible}, we assume that, for every multiindex $k$,
\minilab{e:admissible}
\begin{equ}\label{e:admissible1}
\bigl(\Pi_z X^k\bigr)(\bar z) = (\bar z- z)^k\;,\qquad f_z(X^k) = (-z)^k\;,
\end{equ}
and that, for every $\tau \in \CW$ with $\CI(\tau) \in \CT$
(since we truncated $\CT$, this is not the case for \textit{all} $\tau$), one has the identities
\minilab{e:admissible}
\begin{equs} 
f_z(\CJ_k \tau) &=  -\int_{\R^2} D^k K(z - \bar z)\bigl(\Pi_{z} \tau\bigr)(d\bar z) \;, \qquad |k| < |\tau|+2\;,\label{e:admissible2}\\
\bigl(\Pi_z \CI \tau\bigr)(\bar z) &=  \int_{\R^2} K(\bar z - \bbar z)\bigl(\Pi_{z} \tau\bigr)(d\bbar z) + \sum_{k} {(\bar z - z)^k \over k!} f_z(\CJ_k \tau) \;.\label{e:admissible3}
\end{equs}
Recall that we have set $\CJ_k \tau = 0$ if $|k| \ge |\tau|+2$, so that the sum on the second line is finite.
It is not clear in principle that these integrals converge, but it turns out that the analytical conditions
\eref{e:bounds} guarantee that this is always the case, see \cite[Sec.~5]{Regularity}.

\begin{remark}\label{rem:redundant1}
Note that since $f_z \in \CG_+$, so that it is multiplicative, \eref{e:admissible1} and \eref{e:admissible2}
do specify $f_z$ (and therefore $F_z$) completely once we know $\Pi_z$. There is therefore
quite a lot of rigidity in these definitions, which makes the mere existence of admissible models a 
highly non-trivial fact.
\end{remark}

\begin{remark}\label{rem:redundant2}
Building further on Remark~\ref{rem:redundant1}, it actually turns out that if $\Pi \colon \R^2 \to \CL(\CT,\CS')$ 
satisfies the {\em first} analytical bound in \eref{e:bounds} and
is such that, for $F$ defined from $\Pi$ via \eref{e:admissible}, one has the 
identities \eref{e:admissible} and \eref{e:algebraic}, then the second analytical bound in \eref{e:bounds}
is {\em automatically} satisfied. This is a consequence of \cite[Thm.~5.14]{Regularity}.
\end{remark}

Given any smooth space-time function $\xi_\eps$, there is a canonical way of building an admissible model
$\Psi(\xi_\eps) = (\Pi^\eps, f^\eps)$ as follows. First, we set $\Pi^\eps_z \Xi = \xi_\eps$, independently
of $z$, and we define it on $X^k$ as in \eref{e:admissible1}. 
Then, we define $\Pi^\eps_z$ recursively by \eref{e:admissible3}, as well as the identity
\begin{equ}[e:canonical]
\bigl(\Pi_z^\eps \tau \bar \tau\bigr)(\bar z) = \bigl(\Pi_z^\eps \tau\bigr)(\bar z)
\bigl(\Pi_z^\eps \bar \tau\bigr)(\bar z)\;.
\end{equ}
Note that this is only guaranteed to makes sense if $\xi_\eps$ is a function! It was shown in
\cite[Prop.~8.27]{Regularity} that if we furthermore define $f^\eps$ via \eref{e:admissible}, then
this does indeed define an admissible model for every continuous function $\xi_\eps$.
It is however very important to keep in mind that not every admissible model is obtained in this way, 
or even as a limit of such models! This will be very important in Section~\ref{sec:renorm} 
below when we discuss the renormalisation
procedure that relates \eref{e:SPDE} to \eref{e:SPDEapprox}.

In our case, we would like to define a limiting random model $(\hat \Pi,\hat f)$, naturally called
the ``It\^o model'', based on $\xi$ being space-time
white noise. One could imagine doing this in a very natural way as follows. As before, 
one sets $\hat\Pi_z \Xi = \xi$ and $\hat\Pi_z X^k$ as in \eref{e:admissible1}. In other words, we have
the identity
\begin{equ}
\bigl(\hat\Pi_z \Xi\bigr)(\psi) = \int_{-\infty}^\infty \scal{\psi(t,\cdot), dW(t)}\;,
\end{equ}
where $W$ is our $L^2$-cylindrical Wiener process.

We will only consider admissible
models, so it is again the case that once we know $\hat\Pi_z \tau$, 
$\hat\Pi_z \CI(\tau)$ is determined by
\eref{e:admissible}, so it remains to define $\hat\Pi_z \tau \bar \tau$. To do this, 
recall the decomposition
$\CT = \CT_\Xi \oplus \CT_\CU$ as in \eqref{e:Tdirectsum}.
It is straightforward to verify that the 
structure group $\CG_+$ leaves both of these subspaces invariant and that $|\tau| > 0$ for every
$\tau \in \CU \setminus \{\one\}$. It then 
follows from \cite[Prop.~3.28]{Regularity} 
that if $(\Pi, f)$ is a model, then $\Pi_z \tau$ is
a continuous function for every $\tau \in \CT_\CU$. As a consequence, if $\tau, \bar \tau \in \CU$,
we can again define $\hat\Pi_z \tau \bar \tau$ by 
\begin{equ}
\bigl(\hat\Pi_z \tau \bar \tau\bigr)(\bar z) = \bigl(\hat\Pi_z \tau\bigr)(\bar z)\,\bigl(\hat\Pi_z \bar \tau\bigr)(\bar z)\;.
\end{equ}
It remains to define $\hat\Pi_z$ on elements of the form $\Xi \tau$ with $\tau \in \CT_\CU$. In this case,
it would seem natural to postulate that $\hat\Pi_z \Xi \tau$ is a random distribution which acts on test 
functions $\psi$ by
\begin{equ}[e:wantedProp]
\bigl(\hat\Pi_z \Xi \tau\bigr)(\psi) = \int_{-\infty}^\infty \scal{\psi(s,\cdot)\, \bigl(\hat\Pi_{(t,x)} \tau\bigr)(s,\cdot),dW(s)}\;.
\end{equ}
Unfortunately, things are not quite that easy. Indeed, for this to make sense as an It\^o integral, we need the
integrand to be adapted. This is unfortunately the case only if the support of $\psi$ is 
included in $(t,\infty) \times S^1$.
Indeed, it can easily be verified recursively that, as a consequence of the non-anticipativity of the kernel $K$
and the definition \eref{e:admissible} for an admissible model, the random variables 
$\bigl(\hat\Pi_{(t,x)} \tau\bigr)(s,y)$
are $\CF_{s\vee t}$-measurable, where $\CF$ is the filtration generated by the increments of $W$,
but they are in general \textit{not} $\CF_s$-measurable for $s < t$.
As a consequence, \eref{e:wantedProp} does not make sense as an It\^o integral in general.
Since we know that as far as the limiting equation \eref{e:SPDE} is concerned It\^o integration coincides
with Skorokhod integration, one may think that it suffices to interpret \eref{e:wantedProp} in the 
Skorokhod sense, which is always meaningful if the integrand is sufficiently ``nice''.
Unfortunately, this is not the case either. As a matter of fact, replacing It\^o integration by
Skorokhod integration would not even allow us to satisfy the consistency equations for the ``model''
defined in this way.

Our main result will be a consequence of the convergence of a suitable sequence of renormalised models toward an 
``It\^o model'' $(\hat\Pi,\hat f)$ which does indeed satisfy the property \eref{e:wantedProp} for test functions
that are supported ``in the future''. Here, renormalisation is crucial: 
if $\xi_\eps$ denotes an $\eps$-mollification of
our space-time white noise, then one does \textit{not} expect the sequence of models
$\Psi(\xi_\eps)$ to converge to $(\hat \Pi,\hat f)$. Indeed, even in the analogous case of finite-dimensional
SDEs, the sequence of solutions to random ODEs obtained from natural regularisations of the noise
converges to the Stratonovich solution and not the It\^o solution.

In the case of SPDEs of the type \eref{e:SPDE}, we expect to have to subtract an 
asymptotically infinite correction term
in order to obtain a finite limit. This subtraction will be made at the level of the \textit{model}
rather than at the level of the equation and the ability to do this is one of the main strengths of
the theory used in this article. We will describe below in Section~\ref{sec:renorm} the precise
renormalisation procedure required to achieve this convergence. 

An admissible model $(\Pi,F)$ really defines an extension of the usual Taylor polynomials, which are
given by \eref{e:admissible1}. It is then natural to define spaces $\CD^{\gamma,\eta}$ which mimic a weighted
version of the H\"older spaces $\CC^\gamma$ in the following way. In order to state our
definition, given a compact space-time domain
$D$,  we denote by $D^{(2)}$ the
set of pairs of points $(z, \bar z) \in D^2$ such that furthermore $|z-\bar z| \le 1 \wedge {1\over 2}\sqrt{|t| \wedge |\bar t|}$, where we used $t$ and $\bar t$ as before as a shorthand for the time components of $z$ and $\bar z$.

\begin{definition}
A function $U \colon \R^2 \to \bigoplus_{\alpha<\gamma}\CT_\alpha$ belongs to $\CD^{\gamma,\eta}$ if, for every
compact domain $D$, one has
\begin{equ}[e:normU]
\|U\|_{\gamma,\eta} \eqdef \sup_{z \in D}\sup_{\alpha < \gamma} {\|U(z)\|_\alpha \over |t|^{({\eta - \alpha\over 2})\wedge 0}}
+ \sup_{(z,\bar z) \in D^{(2)}}\sup_{\alpha < \gamma} {\|U(z) - \GGamma{z\bar z} U(\bar z)\|_\alpha \over \bigl(|t|\wedge |\bar t|\bigr)^{\eta-\gamma \over 2} |z-\bar z|^{\gamma - \alpha}} < \infty\;.
\end{equ}
Here, we wrote $\|\tau\|_\alpha$ for the norm of the component of $\tau$ in $\CT_\alpha$ and we used the notation $\GGamma{z\bar z}$ as a shorthand for $\GGamma{\gamma_{z\bar z}} = F_z^{-1} \circ F_{\bar z}$
as above.
\end{definition}

\begin{remark}
The powers of $t$ appearing in this definition allow elements of $\CD^{\gamma,\eta}$ to exhibit a
singularity on the line $\{(t,x)\,:\, t = 0\}$. This is essential in order to be able to deal with solutions to
\eref{e:SPDE} with ``rough'' initial conditions.
\end{remark}

\begin{remark}
In order to streamline notations, we suppressed the dependence on the domain $D$ in this norm.
This is because in practice, we will only ever use this on some fixed space-time domain.
\end{remark}

Note that the space $\CD^{\gamma,\eta}$ does depend in a crucial way on the underlying model
$(\Pi,F)$. Therefore, it is not obvious \textit{a priori} how to compare elements belonging to $\CD^{\gamma,\eta}$,
but based on two different models. This is however crucial when investigating the convergence of solutions
to \eref{e:SPDEapprox} as $\eps \to 0$ since these will 
be obtained from a fixed point problem in $\CD^{\gamma,\eta}$, but where the underlying model depends on $\eps$.
Given two admissible models $(\Pi,F)$ and $(\bar \Pi, \bar F)$, the bounds \eref{e:bounds} yield a natural
notion of a semi-distance between the two models by considering, for a given compact domain $D \subset \R^2$,
the quantity
\begin{equ}[e:distPi]
\|\Pi;\bar \Pi\| = \sup_{z \in D} \sup_{\phi\in \CB \atop \lambda\in(0,1]} \sup_{\tau \in \CW} {\bigl|\bigl(\Pi_z \tau - \bar \Pi_z \tau\bigr)(\phi_z^\lambda)\bigr| \over \lambda^{|\tau|}}
+
\sup_{z,\bar z \in D} \sup_{\tau \in \CW_+} {\bigl|\gamma_{z\bar z}\tau - \bar \gamma_{z\bar z}\tau\bigr| \over |z-\bar z|^{|\tau|}}\;.
\end{equ}
We write this as a distance between $\Pi$ and $\bar \Pi$ only since, as already remarked, $\gamma$ and $\bar \gamma$ are
determined uniquely by $\Pi$ and $\bar \Pi$ via \eref{e:admissible}.
Again, we intentionally do not make the domain $D$ explicit in our notation. This is because our
spatial domain is bounded and we will only consider some finite time horizon $T$. As a 
consequence, we only ever care about our models in some sufficiently large bounded domain
anyway.

A natural distance between elements $U \in \CD^{\gamma,\eta}$ and $\bar U \in \bar \CD^{\gamma,\eta}$
(denoting by $\bar \CD^{\gamma,\eta}$ the space built over the model $(\bar \Pi, \bar F)$), is given by \eref{e:normU},
with $U(z)$ replaced by $U(z) - \bar U(z)$ in the first term and $U(z) - \GGamma{z\bar z} U(\bar z)$ replaced by
\begin{equ}[e:distUUbar]
U(z) - \bar U(z) - \GGamma{z\bar z} U(\bar z) +  \bGGamma{z\bar z} \bar U(\bar z)
\end{equ}
in the second term. We call this quantity $\|U; \bar U\|_{\gamma,\eta}$.
Note that this distance is \textit{not} a function of $U - \bar U$!
It does however define a distance function on the ``fibred space'' $\MM \ltimes \CD^{\gamma,\eta}$
which consists of pairs $((\Pi,F),U)$ of models and modelled distribution such that $U$ belongs 
to the space $\CD^{\gamma,\eta}$ associated to the model $(\Pi,F)$.

The idea now is to reformulate \eref{e:SPDEapprox}, but with $C_\eps = 0$ for the moment, 
as a fixed point problem in $\CD^{\gamma,\eta}$ 
(based on the canonical model $\Psi(\xi_\eps)$ built above) for suitable values of the exponents 
$\gamma$ and $\eta$. As a matter of fact, we will view it as a fixed point problem in
the subspace $\CD_\CU^{\gamma,\eta} \subset \CD^{\gamma,\eta}$ consisting of those functions
that take values in $\CT_\CU$. Any element $U \in \CD_\CU^{\gamma,\eta}$
can be written uniquely as
\begin{equ}[e:decompU]
U(z) = u(z) \one + \tilde U(z)\;,
\end{equ}
where $\tilde U(z)$ takes values in $\bigoplus_{\alpha > 0}\CT_\alpha$. Furthermore, it is a
consequence of \cite[Prop.~3.28]{Regularity} 
that if $\gamma > {1\over 2}-\kappa$, then the function $u$ is necessarily H\"older continuous, in the parabolic sense,
with H\"older exponent ${1\over 2}-\kappa$ on $\R_+ \times \R$. (Its modulus of H\"older continuity might become singular near $t = 0$.)
We will denote by $\CR \colon \CD_\CU^{\gamma,\eta} \to \CC^{{1\over 2}-\kappa}(\R_+\times\R)$
the map $U \mapsto u$. Note that with $U\in \CD_\CU^{\gamma,\eta}$ based on an admissible model $(\Pi,F)$, 
an equivalent way of defining $\CR$ is given by
\begin{equ}[e:defR]
\bigl(\CR U\bigr)(z) = \bigl(\Pi_z U(z)\bigr)(z)
\end{equ}
since, as a consequence of \eref{e:bounds}, $\bigl(\Pi_z \tilde U(z)\bigr)(z) = 0$.

This definition makes sense since it turns out that $\Pi_z \tau$ is necessarily a function (and not just a distribution)
for every $\tau \in \CU$. If it so happens that this is so for every $\tau \in \CW$ (as it is for example for the
models $(\Pi^\eps, F^\eps)$ mentioned above), then \eref{e:defR} actually 
makes sense for every $U \in \CD^{\gamma,\eta}$ (and not just for $U \in \CD_\CU^{\gamma,\eta}$).
A remarkable fact, and this is the content of \cite[Thm~3.10]{Regularity}, is that provided that $\gamma > 0$, the map
\begin{equ}[e:mapR]
(\Pi,U)\mapsto \CR U
\end{equ}
given by \eref{e:defR} is jointly (locally)
Lipschitz continuous with respect to the metric defined in \eref{e:normU} and \eref{e:distPi}, so that 
the map \eref{e:mapR} makes sense even in situations where the definition \eref{e:defR} is nonsensical!
This of course relies very heavily on the fact that we \textit{only} consider admissible models in \eref{e:mapR}
and not arbitrary functions $\Pi\colon \R^2 \to \CL(\CT,\CS')$. The map $\CR$ is called the \textit{reconstruction operator}
since it reconstructs the (global) distribution $\CR U$ from the (local) data $U$ and $\Pi$.

\begin{remark}\label{rem:reconstr}
It is possible to verify that if $U \in \CD^{\gamma,\eta}$ and one defines $\bar U$ by $\bar U(z) = \CQ_{\bar \gamma} U(z)$,
where $\CQ_{\bar \gamma}$ is the projection onto the subspace $\bigoplus_{\alpha < \bar \gamma} \CT_\alpha$,
then one has $\bar U \in \CD^{\bar \gamma,\eta}$, provided of course that $\bar \gamma \le \gamma$. If one still has
$\bar \gamma > 0$, then it is the case that $\CR \bar U = \CR U$. In view of this, one might think that only 
symbols with negative homogeneity ``matter''. This is not the case however, since one can easily loose regularity.
In particular, if $U \in \CD^{\gamma,\eta}$, then, combining the definition of $\CD^{\gamma,\eta}$
with the facts that $|\Xi| = -{3\over 2}-\kappa$, homogeneities are additive,
and by Remark~\ref{rem:prod} $\Gamma \Xi \tau = \Xi \Gamma \tau$ for every $\Gamma \in \CG$,
one easily verifies that $\Xi U \in \CD^{\gamma -{3\over 2}-\kappa,\eta -{3\over 2}-\kappa}$. 
As a consequence, one needs $\gamma > {3\over 2}+\kappa$ if one wishes
the reconstruction operator $\CR$ to be uniquely defined on $\Xi U$.
\end{remark}

\subsection{Abstract fixed point problem}

We now reformulate \eref{e:SPDEapprox} as a fixed point problem in
$\CD_\CU^{\gamma,\eta}$ for suitable values of $\gamma$ and $\eta$. Note first that by Duhamel's formula,
the unrenormalised version of \eref{e:SPDEapprox} 
(i.e.\ the equation with $C_\eps = 0$) is equivalent to the integral
equation
\begin{equ}[e:SPDEint]
u = P \star \bigl((H(u) + G(u)\xi_\eps)\one_{t > 0}\bigr) + Pu_0\;.
\end{equ}
Here, $P$ denotes the heat kernel, $\star$ denotes space-time convolution, 
and $Pu_
0$ denotes the solution to the heat equation with initial condition $u_0$.
A local solution is a pair $(u,T)$ with $T > 0$ and such that \eref{e:SPDEint}
holds on $[0,T] \times \R$. (Here we formulated the problem as if it were on $\R$, which 
one can easily reduce oneself to by considering the periodic extension of the solution.)
For this, we need to reformulate the operations of composition with $H$
and $G$, multiplication by $\xi_\eps$, and convolution against $P$.

Given $U \in \CD_\CU^{\gamma,\eta}$ as in \eref{e:decompU} (which defines $\tilde U$)
and a smooth function $G\colon \R \to \R$,
we write
\begin{equ}[e:FHat]
\bigl(\hat G(U)\bigr)(z) = G(u(z))\one + \sum_{k\ge 1} {D^k G(u(z))\over k!} \tilde U(z)^{k}\;,
\end{equ}
with the understanding that the product between any number of terms such that their
homogeneity adds up to $\gamma$ or more vanishes. 
It was then shown in \cite[Prop.~6.13]{Regularity} that the map $U \mapsto \hat G(U)$ is
locally Lipschitz continuous from $\CD_\CU^{\gamma,\eta}$ to itself, provided that
$\gamma > 0$ and $\eta \in [0,\gamma]$.

Furthermore, for every $\delta > 0$ such that furthermore $\delta < {1\over 2} - \kappa - \eta$, 
and for every $\gamma > 2-\delta$, it is possible to construct a 
linear operator  $\CP \colon \CD^{\gamma-2+\delta,\eta-2+\delta}\to \CD_\CU^{\gamma,\eta}$ 
with the following properties:
\begin{enumerate}
\item One has the identity $\CR \CP U = P \star \CR U$, so $\CP$ represents space-time 
convolution by the heat kernel.
\item One has $\CP U = \CI U + \tilde \CP U$\;, where $\tilde \CP U$ only takes values in $\bar \CT$,
the linear span of the Taylor polynomials $\{X^k\}$.
\item There exists $\theta > 0$ such that
\begin{equ}
\|\CP \one_{t > 0}U\|_{\gamma,\eta} \lesssim T^\theta \|U\|_{\gamma-2+\delta,\eta-2+\delta}\;,
\end{equ}
where the norms are taken over the domain $[0,T] \times \R$.
\end{enumerate}
For a proof of these properties, see Equ.~5.15, Thm~5.12, 
Prop.~6.16, and Thm~7.1 in \cite{Regularity}.
Finally, given a $\CC^\gamma$ function $u$, we write $\TT_\gamma u$ for its Taylor expansion
of (parabolic) order $\gamma$, namely 
\begin{equ}
\bigl(\TT_\gamma u\bigr)(z) = \sum_{|k| < \gamma} {X^k \over k!} \bigl(D^k u\bigr)(z) \in \bar \CT \subset \CT\;. 
\end{equ}

With all of these notations at hand, we can lift \eref{e:SPDE} in a very natural way
to a fixed point problem in $\CD_\CU^{\gamma,\eta}$, by looking for solutions $U$ to
\begin{equ}[e:FP]
U = \CP \bigl((\hat H(U) + \hat G(U)\Xi)\one_{t > 0}\bigr) + \TT_\gamma P u_0\;.
\end{equ}
The main results of \cite[Sec.~7]{Regularity} then allow us to obtain the following result.

\begin{theorem}\label{theo:gen}
Fix $\gamma \in ({3\over 2} + \kappa,\zeta)$ and let $F$, $G$ be smooth. 
Then, for every initial condition $u_0 \in \CC(S^1)$
and every admissible model $(\Pi,F)$, there exists a time $T$ such that the fixed point map
\eref{e:FP} has a unique solution in $\CD^{\gamma,0}([0,T] \times S^1)$. Furthermore, the solution 
is locally Lipschitz continuous as a function from $\CC(S^1) \times \MM$ into $\MM \ltimes \CD^{\gamma,0}$.
\end{theorem}

In principle, one would think that the proof of Theorem~\ref{theo:gen} follows
immediately from Thm~7.8 and Prop.~7.11 in \cite{Regularity}. The only caveat is that
we do not know whether the map $U \mapsto \hat G(U)$ (and similarly for $\hat F$) is locally
Lipschitz continuous in the strong sense on $\CD^{\gamma,0}$. Indeed, while
\cite[Prop.~6.13]{Regularity} yields local Lipschitz continuity when considering arguments
built on the same model, it does \textit{not} make any claim regarding the comparison
of arguments based on different models. There is however a ``trick'' that allows us to obtain
such a result as a corollary of \cite[Prop.~6.13]{Regularity}, thus yielding the following statement.

\begin{proposition}\label{prop:diffNonlin}
Let $\gamma > 0$ and let $(\CT,\CG)$ be a regularity 
structure with no elements of negative homogeneity, 
endowed with a $\gamma$-regular\footnote{See \cite[Def.~4.6]{Regularity} for this terminology. In particular, the product used in this article is $\gamma$-regular for every $\gamma$ by Remark~\ref{rem:prod}.}
product and denote by $\chi > 0$ the lowest non-zero homogeneity appearing in $\CT$.
For $G\colon \R \to \R$ a function of class $\CC^{({\gamma \over \chi} \vee 1)+1}$, let $\hat G$ be
defined as in \eref{e:FHat}. 

Then, provided that $\eta \in [0,\gamma]$,
one has the bound
\begin{equ}[e:boundDiffNonlinear]
\|\hat G(f); \hat G(\bar f)\|_{\gamma,\eta} \lesssim \| f ; \bar f\|_{\gamma,\eta} + \$\Gamma - \bar \Gamma\$_\gamma \bigl(\|f\|_{\gamma,\eta} + \|\bar f\|_{\gamma,\eta}\bigr)\;,
\end{equ}
where we used the notation
\begin{equ}[e:defNormGamma]
\$\Gamma\$_\gamma = \sup_{\alpha < \gamma} \sup_{\tau \in \CT_\alpha\atop \|\tau\| = 1} \sup_{\beta < \alpha} \sup_{z \neq \bar z} {\bigl\|\Gamma_{z\bar z}\tau \bigr\|_\beta \over |z-\bar z|^{\alpha - \beta}} \;,
\end{equ}
with the innermost supremum running over the same domain as the one on which the
norms in \eref{e:boundDiffNonlinear} are taken. Here, $f \in \CD^{\gamma,\eta}$
and $\bar f \in \bar \CD^{\gamma,\eta}$, where $\bar \CD^{\gamma,\eta}$ is the space based on
$\bar \Gamma$.
\end{proposition}

\begin{remark}
The proportionality constant implicit in the above statement is uniform over the set
of all $f$, $\bar f$, $\Gamma$ and $\bar \Gamma$ with
\begin{equ}
\|f\|_{\gamma,\eta} + \|f\|_{\gamma,\eta} + \$\Gamma\$_\gamma + \$\bar \Gamma\$_\gamma \le R\;,
\end{equ}
for arbitrary $R > 0$. Since $G$ is a smooth function with arbitrary growth, one
cannot expect better in general.
\end{remark}

\begin{proof}
The proof relies on the following construction. Given our regularity structure $\CT$ with
structure group $\CG$, we can double the structure in the following way. First, we set
\begin{equ}
\hat \CT = \CT \oplus \CT / \sim\;,
\end{equ}
where $\sim$ is the equivalence relation such that $(\tau_1, \tau_2) \sim (\bar \tau_1, \bar \tau_2)$ if and only if there exists $c \in \R$ such that $\bar \tau_1 = \tau_1 + c \one$ and
$\bar \tau_2 = \tau_2 - c \one$. Since $\one$ is always invariant under the structure
group $\CG$, the group $\hat \CG = \CG \oplus \CG$ acting on $\CT \oplus \CT$ induces 
a natural action on $\hat \CT$.

We now introduce a parameter $\delta >0$ and, for $\alpha > 0$, we equip the 
spaces $\hat \CT_\alpha = \CT_\alpha \oplus \CT_\alpha$ with $\delta$-dependent 
norms by setting
\begin{equ}[e:defNorm]
\|(\tau,\bar \tau)\|_\alpha = \|\tau + \bar\tau\|_\alpha + \delta \, \|\tau - \bar\tau\|_\alpha\;.
\end{equ} 
In the special case $\alpha = 0$, we simply set 
\begin{equ}
\|(\tau,\bar\tau)\|_0 = \|\tau + \bar \tau\|_0\;,
\end{equ} 
which is independent of the representative of our equivalence class. We also have natural
injection maps $\iota, \bar \iota\colon \CT \to \hat \CT$ given by
$\iota \tau = (\tau,0)$ and $\bar\iota \tau = (0,\tau)$. While these are not isometries,
one has $\max\{\|\iota\|,\|\bar \iota\|\} \le 2$, uniformly over $\delta \in (0,1]$. 
Note that 
every element in $\hat \CT$ has a canonical representative for which the two components
proportional to $\one$ are equal, so that \eref{e:defNorm} holds for every component.
Given 
an element $\hat \Gamma = (\Gamma, \bar \Gamma) \in \hat \CG$ and $\hat \tau \in \hat \CT$ with
canonical representative $\hat \tau = (\tau,\bar \tau)$, we then have for any
exponent $\alpha$ the bound
\begin{equs}
\|\hat \Gamma \hat \tau\|_\alpha &= \|\Gamma \tau + \bar \Gamma \bar \tau\|_\alpha + \delta \|\Gamma \tau - \bar \Gamma \bar \tau\|_\alpha \\
&\le {1\over 2} \bigl(\|(\Gamma + \bar \Gamma)(\tau +  \bar \tau)\|_\alpha + \|(\Gamma - \bar \Gamma)(\tau -  \bar \tau)\|_\alpha \bigr) \\
&\quad + {\delta\over 2} \bigl(\|(\Gamma + \bar \Gamma)(\tau -  \bar \tau)\|_\alpha + \|(\Gamma - \bar \Gamma)(\tau +  \bar \tau)\|_\alpha \bigr)\;, \label{e:boundGammaHat}
\end{equs}
uniformly over the choice of $\delta$. In particular, if we set $\hat \Gamma_{xy} = (\Gamma_{xy}, \bar \Gamma_{xy})$, then it follows from \eref{e:boundGammaHat} and \eref{e:defNorm} that we have
\begin{equ}[e:normHat]
\$\hat \Gamma\$_\gamma \le {\delta + \delta^{-1}\over 2} \$\Gamma - \bar \Gamma\$_\gamma + \$\Gamma\$_\gamma + \$\bar \Gamma\$_\gamma \;.
\end{equ}
The idea will be to consider situations where
$\Gamma$ and $\bar \Gamma$ are very close to each other, so that even a choice $\delta \ll 1$
leads to order one bounds on the norm of $\hat \Gamma$. 

Consider now two modelled distributions $f$ and $\bar f$, where $f \in \CD^{\gamma,\eta}(\Gamma)$
and $\bar f \in \CD^{\gamma,\eta}(\bar \Gamma)$, for two models $(\Pi,\Gamma)$ and $(\bar \Pi,\bar \Gamma)$. It is very natural to lift these two models to a single model on $\hat \CT$ by
setting $\hat \Gamma_{xy} = (\Gamma_{xy}, \bar \Gamma_{xy})$ and $\hat \Pi_x (\tau,\bar \tau) = \Pi_x \tau + \bar \Pi_x \bar \tau$. With these notations at hand, it follows immediately
from the definitions and from the fact that 
\begin{equ}
\hat \Gamma_{xy} \iota = \iota \Gamma_{xy}\;,\qquad
\hat \Gamma_{xy} \bar \iota = \bar \iota \bar \Gamma_{xy}\;,
\end{equ}
that both $\iota f$ and $\bar \iota \bar f$ belong to $\CD^{\gamma,\eta}(\hat \Gamma)$ with
norms that are bounded by at most twice their original norms, provided that $\delta \in (0,1]$.

More precisely, it follows immediately from \eref{e:defNorm}, combined with the definitions
of the (semi-)norms $\|\cdot\|_{\gamma,\eta}$, that one has the two-sided bound
\begin{equ}
\|f;\bar f\|_{\gamma,\eta} \le \|\iota f - \bar \iota \bar f\|_{\gamma,\eta}
\le \|f;\bar f\|_{\gamma,\eta} + \delta \bigl(\|f\|_{\gamma,\eta} + \|\bar f\|_{\gamma,\eta}\bigr)\;,
\end{equ}
uniformly over $\delta \in (0,1]$. We are now at the stage where we can
use \cite[Prop.~6.13]{Regularity}, so that
\begin{equs}
\|\hat G(f); \hat G(\bar f)\|_{\gamma,\eta} &\le \|\iota \hat G(f) - \bar \iota \hat G(\bar f)\|_{\gamma,\eta} = \|\hat G(\iota f) - \hat G(\bar \iota \bar f)\|_{\gamma,\eta} 
\lesssim \|\iota f - \bar \iota \bar f\|_{\gamma,\eta} \\
&\le \| f ; \bar f\|_{\gamma,\eta} + \delta \bigl(\|f\|_{\gamma,\eta} + \|\bar f\|_{\gamma,\eta}\bigr)\;.
\end{equs}
This bound is uniform over all $\delta \in (0,1]$, all pairs of models $(\Pi,\Gamma)$ and
$(\bar \Pi,\bar \Gamma)$ such that the induced model $(\hat \Pi,\hat \Gamma)$ has ``norm''
bounded by a fixed constant, and all functions $f \in \CD^{\gamma,\eta}(\Gamma)$
and $\bar f \in \CD^{\gamma,\eta}(\bar \Gamma)$ with corresponding norms bounded by a 
fixed constant. At this stage it looks like one could take $\delta$ arbitrarily small.
The choice of $\delta$ is however limited by  \eref{e:normHat}, which suggests  
that an optimal choice is given by
\begin{equ}
\delta = \$\Gamma - \bar \Gamma\$_\gamma \;,
\end{equ}
With this particular choice of $\delta$, it follows from \eref{e:normHat} that for every $R>0$
there exists $C$ such that $\$\hat \Gamma\$_\gamma \le C$ for any two models $\Gamma$ and $\bar \Gamma$
such that $\$\Gamma\$_\gamma + \$\bar \Gamma\$_\gamma \le R$.
The claim now follows at once.
\end{proof}

\begin{proof}[of Theorem~\ref{theo:gen}]
This is now an almost immediate corollary of \cite[Thm~7.8]{Regularity}, noting that
Proposition~\ref{prop:diffNonlin} guarantees that the nonlinearity is ``strongly locally Lipschitz''
in the terminology of that article. 

The only thing that needs to be verified is that one has indeed $\TT_\gamma P u_0 \in \CD^{\gamma,0}$.
For this, we first note that the heat kernel $P$ satisfies for every $k,\ell \in \N$ the bound
\begin{equ}
\int_\R |\d_x^k \d_t^\ell P(t,x)|\,dx \lesssim t^{-\ell - {k \over 2}}\;,
\end{equ}
uniformly over $(t,x) \in [0,T]\times S^1$ for every fixed final time $T$.
It immediately follows from this bound that the first quantity in \eref{e:normU} is bounded, for every fixed $\gamma > 0$,
by a multiple of $\|u_0\|_{L^\infty}$. The second quantity is then bounded as a consequence of this
by using the remainder formula of \cite[Thm~A.1]{Regularity}.
\end{proof}

\begin{remark}
Note that, a consequence of the second property of $\CP$, any solution $U$ to
\eref{e:FP} satisfies
\begin{equ}[e:propFP]
U(z) - \CI\bigl(\hat H(U(z)) + \hat G(U(z))\Xi\bigr) \in \bar \CT\;,
\end{equ}
for all points $z = (t,x)$ with $t \in (0,T)$.
This fact will be very important in the sequel when we study the effect of our
renormalisation procedure on solutions.
\end{remark}

\begin{remark}\label{rem:regularity}
When applied to our situation, the exponent $\chi$ appearing in the statement of 
Proposition~\ref{prop:diffNonlin} is equal to $\chi = {1\over 2} - \kappa$.
As a consequence, provided that we restrict ourselves to $\gamma < 2 - 4\kappa$, 
we only need $F \in \CC^2$ and $G \in \CC^5$ for the conclusion of Theorem~\ref{theo:gen} 
to hold.
\end{remark}

\section{Renormalisation procedure and main result}
\label{sec:renorm}

At this stage, we note that while Theorem~\ref{theo:gen} allows us to identify solutions
to \eref{e:SPDEapprox} with $C_\eps = 0$ with solutions to the abstract fixed point problem
\eref{e:FP} for a suitable model, we announced a convergence result where we simultaneously
let $C_\eps \to \infty$ and introduce additional correction terms of order $1$. 
In particular, this (correctly) suggests that one has no hope to prove
that the sequence of models $\Psi(\xi_\eps)$ converges to a limit as $\eps \to 0$.

In order to obtain our main result, the strategy is to build a sequence $M_\eps\Psi(\xi_\eps)$
of \textit{renormalised} models, where $M_\eps$ denotes a suitable 
continuous map on the space of admissible models, such that the following properties hold.
First, we show that the sequence $M_\eps\Psi(\xi_\eps)$ converges to a limiting model in $\MM$.
Second, we show that if $U$ is a (local) solution to \eref{e:FP} for the model $M_\eps\Psi(\xi_\eps)$,
then the function $u$ in the decomposition \eref{e:decompU} is the classical solution to the 
PDE \eref{e:SPDEapprox}, but with $H$ replaced by $\bar H$ as in \eref{e:defbarH}. 
Finally, we show that solutions to \eref{e:FP} depend continuously on the model,
thus implying that solutions to \eref{e:SPDEapprox} converge to a limit,
and we identify this limit as the It\^o solution to \eref{e:SPDE}.

In order to implement this strategy, we first explain how the one-parameter family of
transformations $M_\eps$ is built. This requires a better understanding of the algebraic properties
of our regularity structure. In order to simplify notations, we first introduce a graphical shorthand
notation for the elements in $\CW$.

\subsection{Shorthand notation}

From now on, we will highlight the canonical basis vectors of $\CT$ (i.e.\ the elements of $\CW$) by drawing
them in \symbol{blue}, so that it is easy to distinguish them from either real-valued coefficients or
elements of $\CT_+$. Note that while the formal structure of $\CT_+$ is quite similar to
that of $\CT$, its role in the theory is very different. Elements of $\CT$ are there to index the
different components of the underlying model $\Pi_z$, while elements of $\CT_+$ index the matrix
elements of the linear maps $F_z$.
While the notation introduced in Section~\ref{sec:regStruc} is convenient 
for making general statements about $\CT$ or $\CG$, it is not very efficient when
talking about any one specific formal expression, since it soon becomes rather lengthy.

We therefore introduce the following alternative graphical notation. 
Instead of $\sXi$, we
just draw a circle, i.e.\ we have $\sXi = \<Xi>$. Each occurrence of the abstract integration
map $\CI$ is then denoted by a downward facing 
straight line and expressions are multiplied by 
simply joining the trees representing them by their roots. We also use the shorthand
$\symbol{\Xi X_1} = \<XiX>$. For example, we have
\begin{equ}
\symbol{\CI(\Xi)} = \<IXi>\;,\quad
\symbol{\Xi \CI(\Xi)} = \<Xi2>\;,\quad \symbol{\Xi\ \I(X_1\Xi)} = \<Xi2X>\;,\quad
\symbol{\Xi\I^2(\Xi)} = \<Xi22>\;, \quad \ldots
\end{equ} 
While this graphical notation does not allow to describe \textit{every} formal expression in $\CW$
(we have no notation for $\symbol{X_0\Xi}$ for example), it will be sufficient for our needs.
In order to describe elements in $\CT_+$, we will 
also use $\J'$ for $\J_{(0,1)}$, $\J''$ for $\J_{(0,2)}$ and $\dot{\J}$ for $\J_{(1,0)}$.
In view of Remark~\ref{rem:reconstr}, an important role will be played by elements in $\CW$ of 
negative homogeneity, so we list all of them here:
\begin{equ}[e:symbols]
\begin{tabular}{lll}\toprule
Homogeneity & Symbol(s)\\
\midrule
$-{3\over 2} - \kappa$ &  \<Xi> \\
$-1 - 2\kappa$ &  \<Xi2> \\
$-{1\over 2} - 3\kappa$ &  \<Xi3>\,, \<Xi3b> \\
$-{1\over 2} - \kappa$ &  \<XiX> \\
$- 4\kappa$ &  \<Xi4>\,, \<Xi4b>\,, \<Xi4c>\,, \<Xi4e> \\
$- 2\kappa$ &  \<Xi2X>\,, \<XXi2> \\
$0$ &  $\1$ \\
\bottomrule
\end{tabular}
\end{equ}
Note that in principle this list may get longer if we take $\kappa$ too large. One can see
by simple inspection that for sufficiently small $\kappa$, the elements of smallest positive homogeneity 
have homogeneity ${1\over 2}-5\kappa$. Therefore, as long as we assume
that $\kappa < 1/10$, the list \eref{e:symbols} does not change, so we make
this a standing assumption. 

\subsection{General renormalisation group}

It was shown in \cite[Sec.~8.3]{Regularity} that one
can build a natural family of continuous transformations of $\MM$ in the following way.
First, we write
\begin{equs}
\CW_0 &= \Big\{\<Xi>, \<Xi2>, \<Xi3>, \<Xi3b>, \<XiX>, \<Xi4>, \<Xi4c>, \<Xi4e>,\<Xi4b>,\<Xi2X>,\<XXi2>,\1,\<IXi^2>,\<IXi2>,\<Xi22>, \<IXi>, \symbol{X}_1\Big\}\;,\\
\CW_\star &= \Big\{\<Xi>, \<Xi2>, \<Xi3>,\<Xi3b>, \<XiX>,\<IXi>\Big\}\;,
\end{equs}
we denote by $\CT_0 \subset \CT$ the linear span of $\CW_0$, and by $\CT_0^+ \subset \CT_+$ the 
free algebra generated as in \eref{e:genT+} by $X$ and $\CW_0^+ \eqdef \{\CJ_k(\tau)\,:\, \tau \in \CW_\star\;,\; |k| < |\tau|+2\}$. The set $\CW_0$ consists of all symbols 
in $\CW$ of negative homogeneity, as well as those symbols generated from them by the
renormalisation procedure described in Section~\ref{sec:opL} below.
The set $\CT_0^+$ generated from $\CW_\star$ consists of the smallest collection of symbols
required to describe the action of $\CG$ on $\CT_0$.
Consider then an arbitrary linear map $M \colon \CT_0 \to \CT_0$ which is such that
\begin{equs}[2][e:basicM]
M\1&=\1\;, \quad &\quad M(\symbol{X^k}\tau)&=\symbol{X^k} M\tau\;,\\
M\sXi&=\sXi\;,\quad&
M(\I(\tau))&=\I(M\tau)\;,
\end{equs}
where the last identity is assumed to hold for every $\tau \in \CT_0$ such that $\CI(\tau) \in \CT_0$.
It was then shown in \cite[Prop.~8.36]{Regularity} that there exist \textit{unique} maps
$\hat M \colon \CT_0^+ \to \CT_0^+$ and $\DeltaM\colon \CT_0 \to \CT_0 \otimes \CT_0^+$ satisfying the identities
\begin{equs}[e:propMM]
\hat{M}\J_k(\tau)&=\M(\J_k\otimes I)\DeltaM\tau\;,\\
\hat M  (\tau \bar \tau)&= (\hat M \tau)(\hat M\bar  \tau)\;,\\
(I\otimes\M)(\Delta\otimes I)\DeltaM \tau&=(M\otimes\hat{M})\Delta \tau\;,
\end{equs}
where $\M:\CT_0^+\otimes \CT_0^+\to \CT_0^+$ denotes the multiplication map, and such that $\hat M$
leaves $X^k$ invariant. Note that the first identity in \eqref{e:propMM} should be checked for all 
$\tau\in\CW_\ast$, the second for all $\tau, \bar \tau\in\CT_0^+$, and the third for all $\tau\in \CW_0$.
With these notations at hand, we have the following definition.

\begin{definition}\label{def:renorm}
The \textit{renormalisation group} $\RR$ associated to our regularity structure is given by the
set of all linear maps $M$ as above 
such that furthermore, for every $\tau \in \CW_0$, one has
\begin{equ}
\DeltaM \tau = \tau \otimes \one + \sum \tau^{(1)}\otimes \tau^{(2)} \;.
\end{equ}
where each of the terms $\tau^{(1)}$ appearing in these sums 
satisfies $|\tau^{(1)}| > |\tau|$.
\end{definition}

Given any $M \in \RR$ and given an admissible model $(\Pi, F)$, we can define a ``renormalised
model'' $(\hat \Pi, \hat F)$ by setting
\begin{equ}[e:renorm]
\hat \Pi_x \tau = (\Pi_x \otimes f_x) \DeltaM \tau\;,\quad
\hat f_x = f_x  \hat M\;.
\end{equ}
Note that in principle $(\hat \Pi, \hat F)$ is only defined on the smaller regularity
structure $\CT_0$. However, it follows from \cite[Prop.~4.11]{Regularity} and \cite[Thm.~5.14]{Regularity}
that any admissible model
on  $\CT_0$ extends uniquely and continuously to a model on all of $\CT$. We will implicitly use this
extension in the sequel.
One then has the following:

\begin{theorem}
The map $(M,\Pi, F) \mapsto (\hat \Pi, \hat F)$ is continuous from $\RR \times \MM$ to $\MM$.
\end{theorem}

\begin{proof}
The fact that $(\hat \Pi, \hat F)$ satisfies the first bound in \eref{e:bounds} as well as
\eref{e:algebraic} and \eref{e:admissible} was shown in \cite[Thm.~8.44]{Regularity}. 
The fact that the second bound in \eref{e:bounds} also holds does in turn follow automatically
from \cite[Thm.~5.14]{Regularity}, using the fact that $\hat \Pi$ is again admissible. 
The continuity with respect to $M$ follows in the same way.
\end{proof}

\subsection{Renormalization map}\label{sec:opL}

We will not give a full characterisation of $\RR$, but we will instead describe a three-dimensional 
subgroup which is sufficient for our needs.
We write $M = \exp(-c L - c^{(1)} L^{(1)} - c^{(2)} L^{(2)})$, where $c, c^{(1)}, c^{(2)} \in \R$ and $L$, $L^{(1)}$, $L^{(2)}$ 
are linear maps on $\CT_0$. The map $L^{(1)}$ is simply given by
$L^{(1)} \<Xi4> = \1$, as well as $L^{(1)} \tau = 0$ for every $\tau \in \CW_0 \setminus \{\<Xi4>\}$.
Similarly, $L^{(2)}$ is given by
$L^{(2)} \<Xi4e> = \1$, and $L^{(2)} \tau = 0$ otherwise.

The map $L$ on the other hand is more complicated to describe. First, one has 
$L\colon \<Xi2> \mapsto \1$. Furthermore, if $\tau$ is a more complicated expression, 
then $L$ iterates over all occurrences of $\<Xi2>$ as a ``subsymbol'' of $\tau$ 
and ``erases'' it in the graphical notation.

More precisely, one has the identities
\begin{equs}[3][e:defL]
L\<Xi2>&=\1\;, \quad&\quad
L\<Xi3>&=\<IXi>\;, \quad&\quad
L\<Xi3b>&=2\,\<IXi>\;, \\
L\<Xi4b>&=3\,\<IXi^2>\;, \quad&\quad
L \<Xi4e> &= \<IXi2> + \<IXi^2>\;, \quad&\quad
L\,\<Xi4>&=\<IXi2> + \<Xi22>\;,\\
L \<Xi4c> &= \<IXi^2> + 2\, \<Xi22>\;, \quad&\quad
L\<Xi2X>&= \X_1\;,  \quad&\quad L\<XXi2>&= \X_1\;.
\end{equs}
(Recall that $\CI(\1) = 0$, so there is no term $\<XiI>$ appearing in $L\<Xi3>$ and
similarly for the other terms.)
We furthermore have
\begin{equ}
L\1 = L\,\<Xi> = L\, \<XiX> = L\<Xi22> = L\<IXi^2> = L \<IXi2> = L\<IXi> = 0\;.
\end{equ}
In particular, this shows immediately that $L^2 \tau = 0$ 
for all $\tau \in \CW_0$ and $L$ and the $L^{(i)}$ all commute, so that
one simply has $M = I - cL - c^{(1)} L^{(1)} - c^{(2)} L^{(2)}$. 
The main result of this section is the following.

\begin{proposition}
With $M$ defined as above, one has $M \in \RR$.
\end{proposition}

\begin{proof}
Since $\RR$ is a group and the operators $L$, $L^{(1)}$, $L^{(2)}$ commute, 
it suffices to verify this separately
for $\exp(-cL)$ and $\exp(-c L^{(i)})$. We note that if we set $M = \exp(-c L^{(i)})$
for $i \in \{1,2\}$, then $M$ satisfies the identity $\Delta M \tau = (M \otimes I)\Delta \tau$.
As a consequence, it is easy to verify that \eref{e:propMM} holds, provided that we set $\DeltaM \tau = M\tau \otimes \one$, and $\hat M \tau = \tau$.
It follows from the ``upper triangular'' structure of $M$ that $\DeltaM$ then satisfy the properties of 
Definition~\ref{def:renorm}.

Verifying these properties for $M = \exp(-c L)$ is less straightforward.
It follows from the recursive definition of $\Delta$ that one has the identities
\begin{equs}
\Delta\<XiX> &=\<XiX>\otimes\one+\<Xi>\otimes X_1,\\
\Delta\<IXi> &=\<IXi>\otimes\one+\1\otimes\J(\<Xi>),\\
\Delta\<Xi2> &=\<Xi2>\otimes\one+\<Xi>\otimes\J(\<Xi>),\\
\Delta\<IXi2> &=\<IXi2>\otimes\one+\<IXi>\otimes\J(\<Xi>)+\1\otimes\J( \<Xi2> ) \\
\Delta\<IXi^2> &=\<IXi^2>\otimes\one+ 2\ \<IXi>\otimes\J(\<Xi>) + \1\otimes\J(\<Xi>)^2,\\
\Delta\<Xi3b> &=\<Xi3b>\otimes\one+ 2\ \<Xi2>\otimes\J(\<Xi>) + \<Xi>\otimes\J(\<Xi>)^2,\\
\Delta\<Xi3>&=\<Xi3>\otimes\one+\<Xi2>\otimes\J(\<Xi>)+\<Xi>\otimes\J(\<Xi2>),\\
\Delta\<Xi4> &=\<Xi4> \otimes\one+\<Xi3>\otimes\J(\<Xi>)+\<Xi2>\otimes\J(\<Xi2>)\\&\quad
                           +\<Xi>\otimes\J(\<Xi3>)+\<XiX>\otimes\J'(\<Xi3>)+\<Xi>\otimes X_1\J'(\<Xi3>),\\
\Delta\<Xi4b> &= \<Xi4b>\otimes\one + 3\ \<Xi3b> \otimes\J(\<Xi>) + 3\ \<Xi2> \otimes\J(\<Xi>)^2 + \<Xi>\otimes\J(\<Xi>)^3 \\
\Delta\<Xi4c> &= \<Xi4c> \otimes \one + 2\ \<Xi3>\otimes \J(\<Xi>) + \<Xi2>\otimes \J(\<Xi>)^2 + \<Xi> \otimes \J(\<Xi3b>)\label{e:Delta}\\
                            &\quad  + \<XiX> \otimes \CJ'(\<Xi3b>) + \<Xi> \otimes X_1 \CJ'(\<Xi3b>) \\
\Delta\<Xi4e> &=\<Xi4e>\otimes\one+ \<Xi3>\otimes\J(\<Xi>)+ \<Xi3b>\otimes\J(\<Xi>)+ \<Xi2>\otimes\J(\<Xi>)^2 \\
&\quad + \<Xi2>\otimes\J(\<Xi2>)+\<Xi>\otimes\J(\<Xi>)\J(\<Xi2>)\\
\Delta \<Xi22> &= \<Xi22>\otimes\one+\<Xi>\otimes\J(\<IXi>)+\<XiX>\otimes\J'(\<IXi>)+\<Xi>\otimes X_1\J'(\<IXi>)\\
             &\quad+\frac{1}{2}\ \symbol{X_1^2\Xi}\otimes\J''(\<IXi>)+\frac{1}{2}\ \<Xi>\otimes X_1^2\J''(\<IXi>)+\<XiX>\otimes X_1\J''(\<IXi>)\\
             &\quad+ \symbol{X_0\Xi}\otimes\dot{\J}(\<IXi>)+\<Xi>\otimes X_0\dot{\J}(\<IXi>),\\
\Delta\<Xi2X>&=\<Xi2X>\otimes\one+\<Xi2>\otimes X_1+\<Xi>\otimes\J(\<XiX>)
                                                +\<XiX>\otimes\J'(\<XiX>)+\<Xi>\otimes X_1\J'(\<XiX>)\;,\\
\Delta\<XXi2>&= \<XXi2>\otimes\one+\<XiX>\otimes\J(\<Xi>) + \<Xi2>\otimes X_1+\<Xi>\otimes X_1\J(\<Xi>)\;.
\end{equs}

For all above symbols, except for $\<Xi4>$ and $\<Xi4c>$, it turns out that we have
\begin{equ}[e:deltaM]
\DeltaM\tau=M\tau\otimes\one\;.
\end{equ}
Moreover, for these two exceptional symbols, one has
\begin{equs}
\DeltaM\<Xi4> &=M \<Xi4>\otimes\one+ \frac{c}{2}\symbol{X_1^2\Xi}\otimes\J''(\<IXi>) 
+c\symbol{X_0\Xi}\otimes\dot{\J}(\<IXi>) \;,\\
\DeltaM\<Xi4c> &=M \<Xi4c>\otimes\one+ c\symbol{X_1^2\Xi}\otimes\J''(\<IXi>)
+2c\symbol{X_0\Xi}\otimes\dot{\J}(\<IXi>) \;.
\end{equs} 
Finally $\hat{M}:\CT_+\to \CT_+$ is such that $\hat{M}(\sigma\overline{\sigma})
=\hat{M}(\sigma)\hat{M}(\overline{\sigma})$ and for all $\tau\in\CW_\ast$,
\begin{equs}
\hat{M}\J_k(\tau)=\J_k(M\tau)\;. 
\end{equs} 

In order to prove that this is the case, it suffices to check that $\DeltaM$
and $\hat M$ defined in this way do indeed satisfy \eref{e:propMM}.
The first two identities are essentially obvious consequences of the definition of $\hat M$. Except for the two cases $\tau = \<Xi4>$ and $\<Xi4c>$,
 the third one follows from the identity
$$\Delta M\tau=(M\otimes\hat{M})\Delta\tau,$$ or equivalently 
$$\Delta(M\tau-\tau)=\big((M\otimes\hat{M})-I\big)\Delta\tau.$$
It is not hard to check that identity for all $\tau\in\CW_0\backslash\{\<Xi4>,\<Xi4c>\}$. Let us give the details of the computation for the case $\tau=\<Xi4e>$.
We need to check that
$$\Delta(M-I)\<Xi4e>=\big((M\otimes\hat{M})-I\big)\Delta\<Xi4e>.$$
It is plain that
$$M\<Xi4e>=\<Xi4e> - c\; \<IXi2> - c\; \<IXi^2>,$$
so that
\begin{equs}
\Delta(M-I)\<Xi4e>&=-c\; \Delta\<IXi2> - c\; \Delta\<IXi^2>\\
&=-c\; \<IXi2>\otimes\one-c\;\<IXi>\otimes\J(\<Xi>)-c\;\1\otimes\J( \<Xi2> )\\&\quad
-c\;\<IXi^2>\otimes\one-2c\; \<IXi>\otimes\J(\<Xi>) -c\; \1\otimes\J(\<Xi>)^2\; .
\end{equs}
On the other hand, using the expression for $\Delta \<X14e>$ given in \eref{e:Delta}, we obtain
\begin{equs}
\big((M\otimes\hat{M})-I\big)\Delta \<Xi4e>&=
-c\;\<IXi2>\otimes\one-c\;\<IXi^2>\otimes\one-3c\;\<IXi>\otimes\J(\<Xi>)\\
&\quad -c\;\1\otimes\J(\<Xi>)^2-c\;\1\otimes\J(\<Xi2>)\;,
\end{equs} 
thus establishing the required identity.

The two cases $\tau = \<Xi4>$ and $\tau = \<Xi4c>$ are similar, so we only give the detailed computations for the case $\tau = \<Xi4>$. We need to check that
$$(I\otimes\M)(\Delta\otimes I)\DeltaM\<Xi4>=(M\otimes\hat M)\Delta\<Xi4>.$$
We first note that
$$\DeltaM\<Xi4>=\<Xi4>\otimes\one-c\; \<IXi2>\otimes\one-c\; \<Xi22>\otimes\one+ \frac{c}{2}\symbol{X_1^2\Xi}\otimes\J''(\<IXi>) 
+c\symbol{X_0\Xi}\otimes\dot{\J}(\<IXi>) \;,$$ 
from which we deduce  that the left hand side reads
\begin{equs}
(I\otimes&\M)(\Delta\otimes I)\DeltaM\<Xi4>\\
&=\<Xi4>\otimes\one+\<Xi3>\otimes\J(\<Xi>)
+ \<Xi2>\otimes\J( \<Xi2>) +\<Xi>\otimes\J(\<Xi3>)+\<XiX>\otimes\J'(\<Xi3>)
+\<Xi>\otimes X_1\J'(\<Xi3>)\\&\quad
-c\Big(\<IXi2>\otimes\one+\<IXi>\otimes\J(\<Xi>)+\1\otimes\J(\<Xi2>)+\<Xi22>\otimes\one
+\<Xi>\otimes\J(\<IXi>)+\<XiX>\otimes\J'(\<IXi>)\\
&\qquad+\<Xi>\otimes X_1\J'(\<IXi>)+\frac{1}{2}\symbol{X_1^2\Xi}\otimes\J''(\<IXi>)+
\symbol{X_1\Xi}\otimes X_1\J''(\<IXi>)\\
&\quad\quad+\frac{1}{2}\symbol{\Xi}\otimes X_1^2\J''(\<IXi>)+\symbol{X_0\Xi}\otimes\dot\J(\<IXi>)+\symbol{\Xi}\otimes X_0\dot\J(\<IXi>)\Big)\\
&\quad+c\Big(\frac{1}{2}\symbol{X_1^2\Xi}\otimes\J''(\<IXi>)+\symbol{X_1\Xi}\otimes X_1\J''(\<IXi>)+
\frac{1}{2}\symbol{\Xi}\otimes X_1^2\J''(\<IXi>)
\\&\qquad + \symbol{X_0\Xi}\otimes\dot\J(\<IXi>)+\symbol{\Xi}\otimes X_0\dot\J(\<IXi>)\Big)\\
&=\<Xi4>\otimes\one+\<Xi3>\otimes\J(\<Xi>)
+ \<Xi2>\otimes\J( \<Xi2>) +\<Xi>\otimes\J(\<Xi3>)+\<XiX>\otimes\J'(\<Xi3>)
+\<Xi>\otimes X_1\J'(\<Xi3>)\\
&\quad -c\Big(\<IXi2>\otimes\one+\<IXi>\otimes\J(\<Xi>)+\1\otimes\J(\<Xi2>)+\<Xi22>\otimes\one
+\<Xi>\otimes\J(\<IXi>)\\
&\qquad +\<XiX>\otimes\J'(\<IXi>) + \<Xi>\otimes X_1\J'(\<IXi>)\Big)
\;,
\end{equs}
while the right hand side reads
\begin{equs}
(M&\otimes\hat M)\Big(\<Xi4> \otimes\one+\<Xi3>\otimes\J(\<Xi>)+\<Xi2>\otimes\J(\<Xi2>) \\
     &\qquad +\<Xi>\otimes\J(\<Xi3>)+\<XiX>\otimes\J'(\<Xi3>)+\<Xi>\otimes X_1\J'(\<Xi3>)\Big)\\
  &=\<Xi4> \otimes\one+\<Xi3>\otimes\J(\<Xi>)+\<Xi2>\otimes\J(\<Xi2>)
   +\<Xi>\otimes\J(\<Xi3>)+\<XiX>\otimes\J'(\<Xi3>)+\<Xi>\otimes X_1\J'(\<Xi3>)\\
&\quad-c\Big(\<IXi2>\otimes\one+ \<Xi22>\otimes\one + \<IXi>\otimes\J(\<Xi>)+\1\otimes\J(\<Xi2>)                      
+\<Xi>\otimes\J(\<IXi>)\\
&\qquad +\<XiX>\otimes\J'(\<IXi>)+\<Xi>\otimes X_1\J'(\<IXi>)\Big)\;.
\end{equs}
The required identity is established. By inspection, we then see that 
$\DeltaM$ is indeed of the form required by Definition~\ref{def:renorm}, thus concluding the proof.
\end{proof}

\subsection{Renormalised solutions}

We now have all the tools required to prove the following result.

\begin{proposition}\label{prop:identify}
Let $H \in \CC^2$, $G \in \CC^5$, 
let $M_\eps = \exp \bigl(- C_\eps L - c^{(1)} L^{(1)} - c^{(2)} L^{(2)}\bigr)$ 
be as in Section~\ref{sec:opL}, let 
$u_0 \in \CC(S^1)$, let $\xi_\eps$ be a smooth function, and denote by $U$ the 
local solution to \eref{e:FP} with model $M_\eps \Psi(\xi_\eps)$ given by Theorem~\ref{theo:gen}.
Then, the function $u$ as in \eref{e:decompU} is the classical solution to the
PDE \eref{e:SPDEapprox}, with $H$ replaced by $\bar H$ as in \eref{e:defbarH}.

If furthermore the PDE \eref{e:SPDEapprox} has classical solutions that remain bounded
up to time $T > 0$, then
the fixed point problem \eref{e:FP} can also be solved uniquely over the same interval $[0,T]$.
\end{proposition}

\begin{proof}
Let $U(t,x)$ denote the local solution to \eref{e:FP} with model $(\hat \Pi^{(\eps)}, \hat F^{(\eps)}) \eqdef M_\eps \Psi(\xi_\eps)$.
Applying the reconstruction operator $\CR_\eps$ associated to this model
to both sides of \eref{e:FP}, we obtain the identity
\begin{equ}[e:general]
u = P\star \bigl(\bigl(H(u) + \CR_\eps(\hat G(U)\sXi)\bigr)\one_{t > 0}\bigr) + Pu_0\;,
\end{equ}
where we used the fact that $\CR_\eps\bigl(\hat H(U)\bigr) = H(u)$, as well as the defining properties
of the operator $\CP$.
While it is also the case that $\CR_\eps \sXi = \xi_\eps$, it is \textit{not} the 
case in general that $\CR_\eps\bigl(\hat G(U)\sXi\bigr) = G(u)\xi_\eps$.

It then follows from \eref{e:propFP} and \eref{e:FHat} that if we consider it as an element of $\CD^{\gamma,\eta}$ with $\gamma$
greater than, but sufficiently close to, ${3\over 2}$, then $U$ is of the form
\begin{equs}
U&=u\,\1+G(u)\,\<IXi>+ G'(u)G(u)\,\<IXi2>  + u'\,\symbol{X_1} \label{e:formU}\\
&\quad+ G'(u)^2 G(u)\,\<IXi3>   + {1\over 2} G''(u) G^2(u) \<Xi4d> + G'(u)u'\,\<IXiX>\;,
\end{equs}
for some functions $u$ and $u'$. In particular, as an element of $\CD^{\gamma'}$ for
$\gamma' > 0$ sufficiently close to $0$, we have the identity
\begin{equs}
\hat G(U)\sXi &= G(u)\,\<Xi>+ G'(u)G(u)\,\<Xi2> + G'(u)^2 G(u)\,\<Xi3> + G'(u)u'\,\<XiX> \\
&\quad  + {1\over 2} G''(u) G^2(u)\, \<Xi3b> + {1\over 6} G'''(u) G^3(u) \,\<Xi4b>
+ G'(u)^3 G(u)\, \<Xi4> \label{e:RHS}\\
&\quad + {1\over 2} G''(u) G'(u) G^2(u)\, \<Xi4c> + G''(u) G'(u) G^2(u)\,\<Xi4e>\\
&\quad + G'(u)^2 u'\,\<Xi2X> + G''(u) G(u) u'\,\<XXi2>\;.
\end{equs}
At this stage, we note that if we write $(\Pi^{(\eps)}, \Gamma^{(\eps)}) = \Psi(\xi_\eps)$,
then we have the identity
\begin{equ}
\bigl(\hat \Pi^{(\eps)}_{z} \tau\bigr)(z) = \bigl(\Pi^{(\eps)}_{z} M_\eps \tau\bigr)(z)\;,
\end{equ}
with $M_\eps$ as in the statement.
Combining this with \eref{e:defL}, we see that this expression is non-zero
for the symbols $\1$, $\<Xi>$, $\<Xi2>$, $\<Xi4>$ and $\<Xi4e>$, where one has
\begin{equs}[0]
\bigl(\hat \Pi^{(\eps)}_{z} \1\bigr)(z) = 1\;,\quad
\bigl(\hat \Pi^{(\eps)}_{z} \<Xi>\bigr)(z) = \xi_\eps(z)\;,\quad
\bigl(\hat \Pi^{(\eps)}_{z} \<Xi2>\bigr)(z) = -C_\eps\;,\\
\bigl(\hat \Pi^{(\eps)}_{z} \<Xi4>\bigr)(z) = -c^{(1)}\;,\quad
\bigl(\hat \Pi^{(\eps)}_{z} \<Xi4e>\bigr)(z) = -c^{(2)}\;.
\end{equs}
The third identity follows from the fact that $M_\eps \<Xi2> = \<Xi2> - C_\eps \1$ and
\begin{equ}
\Pi^{(\eps)}_{z} \<Xi2> = \Pi^{(\eps)}_{z} \<IXi> \,\cdot\, \Pi^{(\eps)}_{z} \<Xi>\;,
\end{equ}
by the definition \eref{e:canonical} of the canonical lift, noting that 
the first factor vanishes at the point $z$ because $|\<IXi>| > 0$. The last two identities 
hold for similar reasons.

Combining this with \eref{e:RHS} and the expression \eref{e:defR} for the
reconstruction operator, it follows that one has the identity
\begin{equs}
\CR_{\eps}\bigl(\hat G(U)\sXi\bigr)(z) &= G(u(z))\xi_\eps(z) - C_\eps G'(u(z))G(u(z)) \\
&\qquad - c^{(1)} G'(u(z))^3 G(u(z)) - c^{(2)} G''(u(z))G'(u(z))G(u(z))^2 \;.
\end{equs}
The first claim now follows by combining this with \eref{e:general}.

Regarding the possibility to solve \eref{e:FP} up to the classical blow-up time of the
corresponding PDE, this was shown in \cite[Prop.~7.11]{Regularity}.
\end{proof}

\subsection{Main results}

Given a cylindrical Wiener process $W$, we define $\xi_\eps$ as in \eref{e:defxieps},
we set $C_\eps = \eps^{-1} c_\rho$ with $c_\rho$ given in \eref{e:crhofirst}, and we
define $c_\rho^{(1)}$ and $c_\rho^{(2)}$ as in Section~\ref{sec:constants}.
We also define $(\Pi^\eps, F^\eps) = \Psi(\xi_\eps)$, the canonical lift of $\xi_\eps$
to the regularity structure $\CT$. As before, we write $M_\eps = \exp(-C_\eps L - c^{(1)}_\rho L^{(1)}- c^{(2)}_\rho L^{(2)})$
and we denote by $(\hat \Pi^\eps, \hat F^\eps)$ the renormalised model obtained
from $(\Pi^\eps, F^\eps)$ by the action of $M_\eps$ given in \eref{e:renorm}.
With this notation, our main convergence result at the level of models
is the following.

\begin{theorem}\label{theo:convModel}
Let $(\hat \Pi^\eps,\hat F^\eps)$ be the renormalised model described above
with $\xi_\eps$ as in \eref{e:defxieps},
$C_\eps = \eps^{-1}c_\rho$, $c^{(1)} = c^{(1)}_\rho$, and $c^{(2)} = c^{(2)}_\rho$ as defined
in Section~\ref{sec:constants}. Then, there exists a
random model $(\hat \Pi, \hat F)$ and a constant $C$ such that
\begin{equ}[e:wantedBoundModel]
\E \|\hat \Pi^\eps;\hat \Pi\| \le C \eps^{\kappa/2}\;,
\end{equ}
for every underlying compact space-time domain. 

Furthermore, for every $\tau \in \CU$ and every $(t,x)$, the process 
$s \mapsto \bigl(\hat \Pi_{(t,x)}\tau\bigr)(s,\cdot)$ is $\CF_s$-adapted for $s > t$
and, for every smooth test function $\phi$ supported in the future $\{(s,y)\,:\, s > t\}$,
 one has the identity
\begin{equ}[e:Itointegral]
\bigl(\hat \Pi_{(t,x)}\sXi\tau\bigr)(\phi) = \int_t^\infty \scal{\bigl(\hat \Pi_{(t,x)} \tau\bigr)(s,\cdot) \phi(s,\cdot), dW(s)}\;,
\end{equ}
where the integral on the right is the It\^o integral. We call $(\hat \Pi, \hat F)$ the It\^o model.
\end{theorem}

\begin{remark}
In this statement, we did again denote by $\CF$ the filtration generated by the 
underlying cylindrical Wiener process $W$. Note also that \eqref{e:wantedProp}
does \textit{not} hold in general if $\phi$ is not
supported in the future. In fact, the statement may not even make any sense in this case
since in general $\bigl(\Pi_{(t,x)} \tau\bigr)(s,\cdot)$ is not adapted to $\CF$ for
$s < t$. One may wonder if in this case \eqref{e:wantedProp} still holds, but with the integral
on the right interpreted as a Skorokhod integral in situations where the integrand
is anticipative. Again, this is \textit{not} the case in general.
\end{remark}

\begin{proof}
The proof is the content of Section~\ref{sec:convergence} below.
Instead of \eqref{e:wantedBoundModel}, we will however only show the seemingly weaker statement that
\begin{equ}[e:wantedBound]
\E \bigl|\bigl(\hat \Pi_0^\eps \tau - \hat \Pi_0^{\bar \eps} \tau\bigr)(\phi_\lambda)\bigr|^2 
\le C \eps^\kappa \lambda^{2|\tau|+\kappa}\;,\qquad 
\E \bigl|\bigl(\hat \Pi_0^\eps \tau\bigr)(\phi_\lambda)\bigr|^2 
\le C \lambda^{2|\tau|+\kappa}\;,
\end{equ}
uniformly over $\lambda \in (0,1]$, over smooth test functions $\phi$ supported in the ball of radius
$1$ and with $\CC^2$ norm bounded by $1$, over all $\tau \in \CW$ with $|\tau| < 0$, and over
all $0 < \bar \eps \le \eps \le 1$. Here, we wrote $\phi_\lambda$ as a shorthand for the function
\begin{equ}
\phi_\lambda(t,x) = \lambda^{-3} \phi(\lambda^{-2}t, \lambda^{-1} x)\;.
\end{equ}

Taking this bound for granted, 
the existence of a random model $(\hat \Pi, \hat F)$ satisfying the 
required bound \eref{e:wantedBoundModel} is then an immediate consequence of 
\cite[Thm~10.7]{Regularity}.
\end{proof}

We now have all the tools in place to formulate the main convergence result of this article.

\begin{theorem}\label{theo:mainFull}
Consider the setting of Theorem~\ref{theo:convModel}. Fix $\gamma \in ({3\over 2} + \kappa,\zeta]$
and let $H \in \CC^{\chi_1}$ and $G \in \CC^{\chi_2}$ with 
\begin{equ}[e:assChi]
\chi_1 > \Bigl(1 + {2\gamma-4 \over 1-2\kappa}\Bigr) \vee 2\;,\qquad \chi_2 > 
{4 \vee 2\gamma \over 1-2\kappa} \vee 2\;.
\end{equ}
Assume that the derivative of both $H$ and $G$ is uniformly bounded and let $u_0 \in \CC(S^1)$ and $T> 0$. 
For any $\eps > 0$, denote by $U^\eps$ the maximal (up to time $T$)
solution to the fixed point problem in $\CD^{\gamma,0}$
constructed in Theorem~\ref{theo:gen} with respect to the
model $(\hat \Pi^\eps, \hat F^\eps)$. Denote by $U$ the same solution, but with respect to the It\^o model
$(\hat \Pi, \hat F)$.

Then, the maximal existence time for $U$ is almost surely equal to $T$ and $\CR U$ coincides
almost surely with the It\^o solution to \eref{e:SPDE}.
Recalling the distance $\|U;U^\eps\|_{\gamma,\eta}$ stated just after \eref{e:distUUbar},
one then has for every $\theta < \kappa / 2$ the estimate
\begin{equ}
\lim_{\eps \to 0} \P \bigl(\| U; U^\eps\|_{\gamma,0} > \eps^{\theta}\bigr) = 0\;.
\end{equ}
Here, the distance $\|\cdot;\cdot\|_{\gamma,0}$ is taken over the domain
$[0,T] \times S^1$. In particular, it implies that with high probability, the maximal
existence time for $U^\eps$ is at least $T$.
\end{theorem}

\begin{proof}
The proof of the theorem is essentially just a collection of the results of this article.
For the It\^o model 
$(\hat \Pi, \hat F)$, solutions are shown in Corollary~\ref{cor:ItoSol} below 
to coincide with the solutions to \eref{e:SPDE}.
Since these are known to be global almost surely \cite{DPZ}, it follows from
\cite[Prop.~7.11]{Regularity} that the solutions to \eref{e:FP} are also global.

The convergence in probability of $U^\eps$ to $U$
is then an immediate consequence of Theorem~\ref{theo:convModel}, combined with 
the local Lipschitz continuity of the solution map given in Theorem~\ref{theo:gen}.
In particular, it immediately follows that the existence time for $U^\eps$ is at least $T$ with
a probability converging to $1$ as $\eps \to 0$.

The assumptions \eref{e:assChi} on the regularity of $H$ and $G$ are precisely what is 
needed in Proposition~\ref{prop:diffNonlin} to ensure that $\hat H$ is locally Lipschitz continuous
from $\CD^{\gamma}$ into $\CD^{\gamma - 2 + \kappa}$, and that $\hat G$
is locally Lipschitz continuous  from $\CD^{\gamma}$ into $\CD^{\gamma - {1\over 2} + \kappa}$.
(Using the fact that the element of lowest non-zero homogeneity appearing in the description of $U$
is $\<IXi>$ with homogeneity ${1\over 2} - \kappa$.)
\end{proof}

It is now very easy to prove the various results stated in the introduction.

\begin{proof}[of Theorem~\ref{theo:main}]
This is now an immediate corollary of Theorem~\ref{theo:mainFull}, noting that if we
write $u = \CR U$ and $u_\eps = \CR U^\eps$ then, for $\alpha \le {1\over 2} -\kappa$ and $t>0$
both the $\CC^{\alpha,\alpha/2}$ norm over $(t,T]\times S^1$ and the supremum norm
over $[0,T] \times S^1$ are controlled by $\| U; U^\eps\|_{\gamma,0}$.
Furthermore, Proposition~\ref{prop:identify} identifies $u^\eps$ with the classical solution to
\eref{e:SPDEapprox} with modified drift $\bar H$, while 
 Corollary~\ref{cor:ItoSol} identifies $u$ with the It\^o solution to \eref{e:SPDE}.
\end{proof}

\begin{proof}[of Corollary~\ref{cor:flow}]
This follows immediately from the local Lipschitz continuity of the solution map
of Theorem~\ref{theo:gen}, combined with the identification of solutions given by 
Corollary~\ref{cor:ItoSol}.
%
%
\end{proof}

\begin{proof}[of corollaries~\ref{cor:regular} and \ref{cor:coincide}]
These corollaries are both a consequence of the form of the solution. 
Recall that the solution $U$ to the fixed point problem will necessarily be of the form
\eref{e:formU}. Furthermore, it follows from our definition of the It\^o model 
that the solution $v$ to the linearised equation is given by
\begin{equ}
v(\bar z) - v(z) = \bigl(\Pi_z \<IXi>\bigr)(\bar z) + R_v(z,\bar z)\;,
\end{equ}
where the remainder $R_v$ satisfies the bound
\begin{equ}
R_v(z,\bar z) \lesssim {|x-\bar x| + \sqrt{|t-\bar t|} \over (|t| \wedge |\bar t|)^{1\over 4}}\;,
\end{equ}
uniformly over $z, \bar z \in [0,T] \times S^1$ for any fixed $T>0$.
(Here, we used again the shorthands $z = (t,x)$ and $\bar z = (\bar t, \bar x)$.)
In particular, we have
\begin{equ}
G(u(z))\bigl(v(\bar z) - v(z)\bigr) = G(u(z)) \bigl(\Pi_z \<IXi>\bigr)(\bar z) + \bar R_v(z,\bar z)\;,
\end{equ}
with $\bar R_v$ satisfying
$R_v(z,\bar z) \lesssim |x-\bar x| + \sqrt{|t-\bar t|}$,
uniformly away of the line $\{t = 0\}$.

On the other hand, it follows from \eref{e:formU} and the definition of $\CD^{\gamma,0}$ that
\begin{equ}
u(\bar z) - u(z) = G(u(z)) \bigl(\Pi_z \<IXi>\bigr)(\bar z) + R_u(z,\bar z)\;,
\end{equ} 
for a remainder $R_u$ satisfying the bound $R_u(z,\bar z) \lesssim (|x-\bar x| + \sqrt{|t-\bar t|})^{1-2\kappa}$. This is thanks to the definition of a model, combined with the fact that the term
of lowest homogeneity different from $\1$ and $\<IXi>$ appearing in the description of $U$ 
is $\<IXi2>$, which has homogeneity $1-2\kappa$.
Comparing both expressions, Corollary~\ref{cor:regular} follows. 

To prove Corollary~\ref{cor:coincide}, we first note that, as a consequence of \eref{e:formU},
if $z$ is such that $u(z) = \bar u(z)$, then the coefficients multiplying $\<IXi>$ and $\<IXi2>$ in the
description of $U$ and $\bar U$ necessarily also coincide at the point $z$. 
As a consequence, $\bar u - u$ is differentiable at $z$ (in the spatial direction) 
and, after subtracting the corresponding linear term,
the remainder is given by
\begin{equ}
G'(u(z)) \bigl(u'(z) - \bar u'(z)\bigr)  \bigl(\Pi_z \<IXiX>\bigr)(\bar z)  + R\;,
\end{equ}
where $R$ is bounded by a multiple of $|\bar t - t| + |\bar x - x|^2$. (Again, locally
uniformly away from $\{t=0\}$.) Since $\<IXiX>$ is of homogeneity ${3\over 2} - \kappa$, the 
first claim follows.

In the context of the last statement of Corollary~\ref{cor:coincide}, the point $z$ furthermore has
the property that $u'(z) = \bar u'(z)$. Indeed, if this were not the case then, by the first part, 
$u - \bar u$ would have a non-zero spatial derivative at $z$, which would contradict our assumption on
the behaviour of the solutions near $z$. We then combine this with the fact that, to order $2$, the 
expansion for $U$ is the same as that for $\hat G(U)\Xi$ given in \eref{e:RHS} 
(except for the terms multiplying the ``Taylor polynomials''). In particular, as a consequence of
the identities $u(z) = \bar u(z)$ and $u'(z) = \bar u'(z)$, all of these terms agree at the point $z$.
Furthermore, the terms of lowest
homogeneity larger than $2$ have homogeneity ${5\over 2} - 5\kappa$, from which the bound
given in the statement follows at once. The claim concerning the signs of constants $D_i$ is an 
immediate consequence of the assumption that 
$u(s,y) \ge \bar u(s,y)$ in 
$\{(s,y)\,:\, |x-y| \le \delta \;\&\; s \in (t-\delta,t]\}$.
\end{proof}

\section{Construction of the It\^o model}
\label{sec:convergence}

The aim of this section is to obtain the bound \eref{e:wantedBound} as well as the
identity \eqref{e:Itointegral}. 

\subsection{Estimate of the first term}

In order to motivate the technique of proof, we first
show in detail as an example how one shows convergence 
of $\bigl(\hat \Pi_0^\eps \<Xi2>\bigr)(\phi_\lambda)$ to a limit.

\subsubsection{Graphical notation}

Since the random variables $\bigl(\hat \Pi_0^\eps \tau\bigr)(\phi)$ belong to Wiener chaoses
of finite order, they can be described in terms of their Wiener chaos decomposition.
Similarly to what was already done in \cite{KPZ,Regularity}, 
we use a graphical notation to describe these random variables.
A random variable belonging to the $k$th homogeneous Wiener chaos can
be described by a kernel in $L^2(\R^2)^{\otimes k}$, i.e.\ the
space of square-integrable functions in $k$ space-time variables via the correspondence 
$f \mapsto I_k(f)$ given in \cite[p.~8]{Nualart}.
Such a kernel will always be constructed from elementary kernels by 
multiplication and integration.

By translation invariance, we will only ever need to consider random
variables of the type $\bigl(\Pi_0^\eps \tau\bigr)(\phi_\lambda)$,
where $\phi_\lambda$ denotes a test function that is localised around the
origin. Nodes in our graph will represent variables in $\R^2$, with one
special node \tikz[baseline=-3] \node [root] {}; representing the origin.
The nodes \tikz[baseline=-3] \node [var] {}; represent the arguments of our kernel,
so that a random variable in the $k$th (homogeneous) Wiener chaos is represented by
a graph with exactly $k$ such nodes. The remaining nodes, which we draw as \tikz[baseline=-3] \node [dot] {};, represent dummy variables that are to be integrated out.

Each line then represents a kernel, with 
\tikz[baseline=-0.1cm] \draw[kernel] (0,0) to (1,0);
representing the kernel $K$, 
\tikz[baseline=-0.1cm] \draw[rho] (0,0) to (1,0);
representing the kernel $\rho_\eps$, and
\tikz[baseline=-0.1cm] \draw[testfcn] (1,0) to (0,0);
representing a generic test function $\phi_\lambda$ rescaled to 
scale $\lambda$.

Whenever we draw a barred arrow 
\tikz[baseline=-0.1cm] \draw[kernel1] (0,0) to (1,0);
this represents a factor $K(t-s,y-x) - K(-s,-x)$, where
$(s,x)$ and $(t,y)$ are the coordinates of the starting and end
point respectively. Finally, a double barred arrow 
\tikz[baseline=-0.1cm] \draw[kernel2] (0,0) to (1,0);
represents a factor 
$K(t-s,y-x) - K(-s,-x) - y\, K'(-s,-x)$. 

With these graphical notations at hand, it follows for example from the
recursive definition of $\Pi_0^{(\eps)} \<Xi2>$ combined
with the contraction formula for the Wiener chaos decomposition of a product
(see \cite[Prop.~1.1.2]{Nualart}), that one
has the identity
\begin{equ}[e:PiXiTwo]
\bigl(\Pi_0^{(\eps)} \<Xi2>\bigr)(\phi_\lambda) = \;
\begin{tikzpicture}[scale=0.35,baseline=0.3cm]
	\node at (0,-1)  [root] (root) {};
	\node at (-2,1)  [dot] (left) {};
	\node at (-2,3)  [dot] (left1) {};
	\node at (0,1) [var] (variable1) {};
	\node at (0,3) [var] (variable2) {};
	
	\draw[testfcn] (left) to  (root);
	
	\draw[kernel1] (left1) to (left);
	\draw[rho] (variable2) to (left1); 
	\draw[rho] (variable1) to (left); 
\end{tikzpicture}\;
+ \;
\begin{tikzpicture}[scale=0.35,baseline=0.3cm]
	\node at (0,-1)  [root] (root) {};
	\node at (-2,1)  [dot] (left) {};
	\node at (-2,3)  [dot] (left1) {};
	\node at (0,2) [dot] (variable1) {};
	\node at (0,2) [dot] (variable2) {};
	
	\draw[testfcn] (left) to (root);
	
	\draw[kernel1] (left1) to (left);
	\draw[rho] (variable2) to (left1); 
	\draw[rho] (variable1) to (left); 
\end{tikzpicture}\;.
%
%
%
%
\end{equ}
In other words, it consists of the sum of one term belonging to the second homogeneous
Wiener chaos and one term belonging to the zeroth chaos. (So this is just a real number.)
The first term is described by a kernel $(z,\bar z) \mapsto \CW^{(2;\eps)}(z,\bar z)$ given
by
\begin{equ}
\CW^{(2;\eps)}(z,\bar z) = \int \rho_\eps(z-z_1)\rho_\eps(\bar z-z_2) \bigl(K(z_2 - z_1) - K(-z_1)\bigr) \,\phi_\lambda(z_2)\,dz_1\,dz_2\;,
\end{equ}
while the constant term is given by $\int \CW^{(2;\eps)}(z,z)\,dz$.

The problem with this is that the second term diverges as $\eps \to 0$,
so we would like our renormalisation procedure to cancel this term out.
This is why our renormalisation map $M$ was of the form $M\<Xi2> = \<Xi2> - C_\eps \1$
and it motivates the definition
\begin{equ}
C_\eps = \int K(t,x) \rho^{\star 2}_\eps(t,x)\,dt\,dx = \;
\begin{tikzpicture}[scale=0.35,baseline=0.6cm]
	\node at (-2,1)  [root] (left) {};
	\node at (-2,3)  [dot] (left1) {};
	\node at (0,2) [dot] (variable) {};
	
	\draw[kernel] (left1) to (left);
	\draw[rho] (variable) to (left1); 
	\draw[rho] (variable) to (left); 
\end{tikzpicture}\;.
\end{equ}
For $\eps$ small enough, it follows from the scaling properties of
the heat kernel that this constant is indeed equal to 
$\eps^{-1}c_\rho$ with $c_\rho$ as defined in the introduction and already mentioned several times.
Since one then has $\bigl(\hat \Pi_0^{(\eps)} \<Xi2>\bigr)(\phi_\lambda) = 
\bigl(\Pi_0^{(\eps)} \<Xi2>\bigr)(\phi_\lambda) - C_\eps \int \phi_\lambda$
as a consequence of \eref{e:defL}, the renormalised model is given by
\begin{equ}[e:decompPiXiTwo]
\bigl(\hat \Pi_0^{(\eps)} \<Xi2>\bigr)(\phi_\lambda) = \;
\begin{tikzpicture}[scale=0.35,baseline=0.3cm]
	\node at (0,-1)  [root] (root) {};
	\node at (-2,1)  [dot] (left) {};
	\node at (-2,3)  [dot] (left1) {};
	\node at (0,1) [var] (variable1) {};
	\node at (0,3) [var] (variable2) {};
	
	\draw[testfcn] (left) to  (root);
	
	\draw[kernel1] (left1) to (left);
	\draw[rho] (variable2) to (left1); 
	\draw[rho] (variable1) to (left); 
\end{tikzpicture}\;
 - \;
\begin{tikzpicture}[scale=0.35,baseline=0.3cm]
	\node at (0,-1)  [root] (root) {};
	\node at (-1,1)  [dot] (left) {};
	\node at (0,3)  [dot] (top) {};
	\node at (1,1) [dot] (right) {};
	
	\draw[testfcn] (left) to  (root);
	
	\draw[kernel] (right) to (root);
	\draw[rho] (top) to (right); 
	\draw[rho] (top) to (left); 
\end{tikzpicture}\;.
\end{equ}
The reason why the second term is present in this expression is that
the second term in \eref{e:PiXiTwo} really consists of two terms since,
as mentioned above, the arrow represents a factor $K(t-s,y-x) - K(-s,-x)$.
The renormalisation term only cancels the first of these two terms.

In order to obtain the second bound in \eref{e:wantedBound}, we would now like to show that
\begin{equ}
\E \bigl|\bigl(\hat \Pi_0^{(\eps)} \<Xi2>\bigr)(\phi_\lambda)\bigr|^2 \lesssim \lambda^{2|\<Xi2s>| + \kappa} = \lambda^{-2-3\kappa}\;,
\end{equ}
where we used the fact that $|\<Xi2>| = -1 -2\kappa$.
For this, we use the fact that if $X$ is a random variable belonging to the $k$th homogeneous
Wiener chaos described by a kernel $\CW$ with $k$ variables, then one has the bound $\E X^2 \le k! \|\CW\|^2$, where
$\|\cdot\|$ denotes the $L^2$-norm. The reason why we do not have equality is
that in general equality holds only when $\CW$ is symmetrised under permutations of its
$k$ arguments, which is not the case here. 
Using furthermore the orthogonality of Wiener chaoses, we therefore deduce
from \eref{e:decompPiXiTwo} the bound
\begin{equ}[e:boundsXi2]
\E \bigl|\bigl(\hat \Pi_0^{(\eps)} \<Xi2>\bigr)(\phi_\lambda)\bigr|^2 \le 2\;
\begin{tikzpicture}[scale=0.35,baseline=0.3cm]
	\node at (0,-1)  [root] (root) {};
	\node at (-1.5,1)  [dot] (left) {};
	\node at (-1.5,3)  [dot] (left1) {};
	\node at (1.5,1) [dot] (variable1) {};
	\node at (1.5,3) [dot] (variable2) {};
	
	\draw[testfcn] (left) to  (root);
	\draw[testfcn] (variable1) to  (root);
	
	\draw[kernel1] (left1) to (left);
	\draw[kernel1] (variable2) to (variable1);
	\draw[rho] (variable2) to node[dot] {} (left1); 
	\draw[rho] (variable1) to node[dot] {} (left); 
\end{tikzpicture}\; + \;
\left(
\begin{tikzpicture}[scale=0.35,baseline=0.3cm]
	\node at (0,-1)  [root] (root) {};
	\node at (-1,1)  [dot] (left) {};
	\node at (0,3)  [dot] (top) {};
	\node at (1,1) [dot] (right) {};
	
	\draw[testfcn] (left) to  (root);
	
	\draw[kernel] (right) to (root);
	\draw[rho] (top) to (right); 
	\draw[rho] (top) to (left); 
\end{tikzpicture}
\right)^2\;,
\end{equ}
where we used the same graphical notations as before. Note that this is simply a real number (depending
of course on the scale $\lambda$ at which the test function is localised),
as can be seen from the absence of free variables \tikz[baseline=-3] \node [var] {};.

The second term in \eref{e:boundsXi2} can be bounded with relative ease. Indeed, it suffices to
note that it is equal to 
\begin{equ}
\biggl(\int \phi_\lambda(z) \bigl(\rho_\eps^{\star 2} * K\bigr)(z)\,dz\biggr)^2\;.
\end{equ}
This in turn is bounded by a multiple of $\lambda^{-2}$ (uniformly in $\eps$) as a consequence of 
\cite[Lemma~10.17]{Regularity}, which is precisely the desired result.

To bound the first term in \eref{e:boundsXi2}, one realises that it is not really necessary to keep
track of the precise form of the kernels represented by each edge. The only relevant information
is their singular behaviour near the origin. There are however two additional pieces of information
that we need to keep track of. First, there is the fact that the two arrows really correspond
to a difference between two kernels where one of them is evaluated at the root.
Then, there is the fact that, although the dashed lines correspond to a kernel that has
homogeneity $-3$ (if one wants a bound uniform in $\eps$, then the best one can do is to
bound the function $\rho_\eps(z)$ by $|z|^{-3}$ in parabolic space-time), we know for a fact that
it is integrable even though such a homogeneity could in principle lead to a logarithmic divergence.

\subsubsection{Bounds on labelled graphs}
One might then rewrite the bound \eref{e:boundsXi2} as follows:
\begin{equ}[e:IXi2bnd]
\E \bigl|\bigl(\hat \Pi_0^{(\eps)} \<Xi2>\bigr)(\phi_\lambda)\bigr|^2 \lesssim \;
\begin{tikzpicture}[scale=0.4,baseline=0.3cm]
	\node at (0,-1)  [root] (root) {};
	\node at (-1.5,1)  [dot] (left) {};
	\node at (-1.5,3)  [dot] (left1) {};
	\node at (1.5,1) [dot] (variable1) {};
	\node at (1.5,3) [dot] (variable2) {};
	
	\draw[dist] (left) to (root);
	\draw[dist] (variable1) to (root);
	
	\draw[->,generic] (left1) to  node[labl,pos=0.45] {\tiny 1,1} (left);
	\draw[->,generic] (variable2) to node[labl,pos=0.45] {\tiny 1,1} (variable1);
	\draw[generic] (variable2) to node[labl] {\tiny 3,-1} (left1); 
	\draw[generic] (variable1) to  node[labl] {\tiny 3,-1} (left); 
\end{tikzpicture}
\; + \lambda^{-2}\;.
\end{equ}
Here, each vertex $v$ represents an integration variable $x_v$ as before, except for the 
larger root which represents the origin. Edges $e = (v,\bar v)$ are oriented and decorated 
with a label $(m_e,r_e) \in \R \times \Z$. The orientation of an edge matters only if $r_e > 0$.

If $r_e = 0$, then the corresponding edge represents a factor 
$\hat J_e(x_v, x_{\bar v}) = J_e(x_{\bar v} - x_v)$,
where $J_e$ is a smooth compactly supported function with a singularity of order 
$m$ at the origin. See \cite{KPZJeremy} for a precise definition of what we mean by a singularity
of order $m$. It suffices to know at this stage that $\rho_\eps$ and $\rho_\eps^{\star 2}$
have a singularity of order $3$ and $K$ has a singularity of order $1$.

If $r_e > 0$, then the corresponding edge
represents a factor 
\begin{equ}
\hat J_e(x_v, x_{\bar v}) = J_e(x_{\bar v} - x_v) - \sum_{|k|_\s < r_e} {x_{\bar v}^k \over k!} D^k J_e(-x_v)\;.
\end{equ}
Here we see why the orientation matters in this case. Previously, changing the orientation 
of the edge yields exactly the same factor, provided that we simultaneously change the function
$J_e$ into the function $x \mapsto J_e(-x)$. When $r>0$, this is no longer the case.

The description of the edges with $r_e < 0$ is slightly more delicate. As before, we give 
ourselves a kernel $J_e$ with a singularity of order $m$ at the origin. This time however,
$J_e$ is not necessarily integrable, so we build from it the distribution
\begin{equ}
\bigl(\Ren J_e\bigr)(\phi) = \int J_e(x) \Bigl(\phi(x) - \sum_{|k|_\s < |r_e|} {x^k \over k!} D^k\phi(0)\Bigr)\,dx + \sum_{|k|_\s < |r_e|} {I_e^{(k)} \over k!} D^k\phi(0)\;,
\end{equ}
where the $I_e^{(k)}$ are some finite numbers that we also need to specify for such an edge.
Provided that $m_e + r_e < 3$ (here $3$ is the scaling dimension of parabolic space-time), 
this yields a well-posed distribution thanks to the
fact that $J_e$ is integrated against a function that vanishes to sufficiently high order
at the origin. In our particular case, we only need to subtract the value of the test function 
$\phi$ at the origin. Furthermore, since the kernel $\rho_\eps^{(2)}$ integrates
to $1$, we actually have that, in the particular case described
by \eqref{e:IXi2bnd}, 
$\Ren J_e = J_e$ if we choose $I_0 = 1$. With such a ``renormalised
kernel'' at hand, we then set somewhat informally 
$\hat J_e(x_v, x_{\bar v}) = \bigl(\Ren J_e\bigr)(x_{\bar v} - x_v)$ in cases where $r < 0$.

There are furthermore two distinguished
edges (represented in boldface) that necessarily connect to the origin, so they are always
of the type $e = (v,0)$, and that represent a factor 
$\hat J_e(x_v, x_0) = \phi_\lambda(x_v-x_0)$. 
One should think of these edges as being decorated with the label $(0,0)$ but since
this is always the case we do not draw these labels.

In order to state our bounds, we will always denote such a graph by $(\CV,\CE)$,
where $\CE$ is a set of directed edges for the vertex set $\CV$. The distinguished
``origin'' is denoted by $0 \in \CV$ and we use the notation $\CV_0 = \CV \setminus \{0\}$.
We furthermore denote by $v_{\star,1}$ and $v_{\star,2}$ the two vertices that are
connected to the origin by the distinguished edges and we set
\begin{equ}
\CV_\star = \{0,v_{\star,1},v_{\star,2}\}\;.
\end{equ}
With all of these notations at hand, a labelled graph as above, together with
the corresponding collection of kernels $J_e$ and constants $I_e^{(k)}$
determines a number
\begin{equ}[e:bigsum]
\CI_\lambda(J) \eqdef \int_{(\R^2)^{\CV_0}} \prod_{e \in \CE} \hat J_e(x_{e_-},x_{e_+})\,dx\;,
\end{equ}
where we also implicitly set $x_0 = 0$.

\begin{remark}
At this stage, the careful reader may wonder what \eqref{e:bigsum} actually means:
some of the factors appearing there are distributions, so that this is not at all 
clear a priori. 
To clarify this, given any homogeneity $m \in \R$
and some $d > 0$, we define the following (semi)norm on the space of 
compactly supported functions that are smooth everywhere, except at the origin:
\begin{equ}
\|J\|_{m,d} = \sup_{|k|_\s < d} \sup_{0 < |x|_\s \le 1} |x|_\s^{m+|k|_\s} |D^k J(x)|\;.
\end{equ}
Recall again that here $x$ denotes a space-time point and $|x|_\s$ denotes
its parabolic norm. 
If we then replace each of the $J_e$ by a smooth
function $J_e^{(n)}$ such that $\|J_e - J_e^{(n)}\|_{m_e,d} \le 1/n$
for every $e$ and define $\hat J_e^{(n)}$ from $J_e^{(n)}$ as above, then
$\CI_\lambda(J^{(n)})$ is well-defined. 
The quantity $\CI_\lambda(J)$ is then defined as the limit of this quantity 
as $n \to 0$, provided that this limit exists and is independent of the 
approximating sequence for $d$ sufficiently large. If the limit doesn't exist 
or depends on the approximating sequence, then we simply set $\CI_\lambda(J) = \infty$.
\end{remark}

There is a natural homogeneity associated to $\CI_\lambda$ as follows. To each
integration variable, we associate a homogeneity $3$, which is the scaling dimension
of parabolic space-time. To each factor $\hat J_e$ corresponding to a kernel of
singularity $m_e$, we associate a homogeneity $-m_e$, except for the factors
$\phi_\lambda$ which have homogeneity $-3$, and not $0$. 
In other words, the total homogeneity of this expression is given by
\begin{equ}
\alpha = 3 |\CV_0| - 6 - \sum_{e \in \CE} m_e = 3 |\CV \setminus\CV_\star| - \sum_{e \in \CE} m_e\;.
\end{equ}
It is natural to guess that one then has $|\CI_\lambda| \sim \lambda^\alpha$
for small values of $\lambda$. This is \textit{not} the case in general! For example,
it might happen that the integral in \eref{e:bigsum} does not even converge.
Alternatively, it might happen that it converges, but the resulting expression
has the ``wrong'' homogeneity.

In order to formulate the additional assumptions we will place on our labelled graph, 
we define, for any $\bar \CV \subset \CV$, the following subsets of $\CE$:
\begin{equs}
\CE^\uparrow(\bar \CV) &= \{e  \in \CE\,:\, e\cap \bar \CV = e_- \;\&\; r_e > 0\}\;,\\
\CE^\downarrow(\bar \CV) &= \{e  \in \CE\,:\, e\cap \bar \CV = e_+ \;\&\; r_e > 0\}\;,\\
\CE_0(\bar \CV) &= \{e  \in \CE\,:\, e\cap \bar \CV = e\}\;,\\
\CE(\bar \CV) &= \{e  \in \CE\,:\, e\cap \bar \CV \neq \emptyset\}\;.
\end{equs}
Here, we use the notation $e = (e_-,e_+)$ for a directed edge.
Consider now the following assumption, where we use the shorthands 
$r_e^+ = (r_e \vee 0)$ and $r_e^- = -(r_e \wedge 0)$.

\begin{assumption}\label{ass:graph}
The labelled graph $(\CV,\CE)$ satisfies the following properties.
\begin{itemize}\itemsep0em
\item[1.] For every edge $e \in \CE$, one has $m_e - r_e^- < 3$.
\item[2.] For every subset
$\bar \CV \subset \CV_0$ of cardinality at least $3$, one has 
\begin{equ}[e:assEdges]
\sum_{e \in \CE_0(\bar \CV)} m_e < 3(|\bar \CV| - 1)\;.
\end{equ}
\item[3.] For every subset
$\bar \CV \subset \CV$ containing $0$ and of cardinality at least $2$, one has 
\begin{equ}[e:assEdges1]
\sum_{e \in \CE_0(\bar \CV)} m_e + \sum_{e \in \CE^\uparrow(\bar \CV)}(m_e + r_e - 1) - \sum_{e \in \CE^\downarrow(\bar \CV)} r_e < 3(|\bar \CV| - 1)\;.
\end{equ}
\item[4.] For every non-empty subset $\bar \CV \subset \CV\setminus\CV_\star$,
one has the bounds
\begin{equ}[e:assEdges2]
\sum_{e \in \CE(\bar \CV)\setminus \CE^\downarrow(\bar \CV)} m_e  + \sum_{e \in \CE^\uparrow(\bar \CV)} r_e- \sum_{e \in \CE^\downarrow(\bar \CV)} (r_e-1) > 3|\bar \CV| \;.
\end{equ}
\end{itemize}
\end{assumption}
It turns out that Assumption~\ref{ass:graph} is sufficient to 
guarantee that the quantity
$\CI_\lambda$ does indeed have the correct scaling behaviour for small values of $\lambda$.
This is the content of the following theorem, the proof of which can be found in \cite{KPZJeremy}.

\begin{theorem}\label{theo:ultimate}
Provided that Assumption~\ref{ass:graph} holds, there exists $d> 0$ and a constant 
$C$ depending only on the number of  vertices in $\CV$ 
and on the values of the constants $I_e$, such that
\begin{equ}
|\CI_\lambda(J)| \le C \lambda^\alpha \prod_{e\in \CE} \|J_e\|_{m_e,d}\;,\qquad \lambda \in (0,1]\;,
\end{equ}
where $\alpha = 3 |\CV\setminus\CV_\star| - \sum_{e \in \CE} m_e$.
\end{theorem}

\subsection{Construction of labelled graphs}
\label{sec:labGraphs}

In our case, we will always consider the situation where the labelled graph 
$(\CV,\CE)$ is built from a ``half graph'' in the following way:
\begin{equ}
\E\;\left|\;
\begin{tikzpicture}[baseline=1.2cm]
	  \node[cloud, fill=gray!20, cloud puffs=16, cloud puff arc= 100,
	        minimum width=1.5cm, minimum height=2.2cm, aspect=1] at (-1,1.7) {$\cdots$};
	
	  \node at (0,0)  [root] (zero) {};
	  \node at (-1,1)  [dot] (left) {};
	  \node at (-0.6,1.2)  [dot] (l1) {};
	  \node at (-0.6,2.2)  [dot] (l2) {};
	  \node at (0.3,1.2)  [var] (v1) {};
	  \node at (0.3,2.2)  [var] (v2) {};
	\draw[testfcn] (left) to (zero);
	
	\draw[rho] (l1) to (v1);
	\draw[rho] (l2) to (v2);
\end{tikzpicture}\;
\right|^2
\;\le\quad
\begin{tikzpicture}[baseline=1.2cm]
	  \node[cloud, fill=gray!20, cloud puffs=16, cloud puff arc= 100,
	        minimum width=1.5cm, minimum height=2.2cm, aspect=1] at (-1,1.7) {$\cdots$};
	  \node[cloud, fill=gray!20, cloud puffs=16, cloud puff arc= 100,
	        minimum width=1.5cm, minimum height=2.2cm, aspect=1] at (1,1.7) {$\cdots$};
	
	  \node at (0,0)  [root] (zero) {};
	  \node at (-1,1)  [dot] (left) {};
	  \node at (1,1)  [dot] (right) {};
	  \node at (-0.6,1.2)  [dot] (l1) {};
	  \node at (-0.6,2.2)  [dot] (l2) {};
	  \node at (0.6,1.2)  [dot] (r1) {};
	  \node at (0.6,2.2)  [dot] (r2) {};
	\draw[dist] (left) to (zero);
	\draw[dist] (right) to (zero);
	
	\draw (l1) to  node[labl] {\tiny 3+,-1}  (r1);
	\draw (l2) to  node[labl] {\tiny 3+,-1}  (r2);
\end{tikzpicture}
\;,
\end{equ}
where one furthermore performs the substitutions
\begin{equ}[e:subs]
\tikz[baseline=-0.6cm] \draw[kernel] (0,0) -- (0,-1); \to \tikz[baseline=-0.6cm] \draw (0,0) to node[labl] {\tiny 1,0} (0,-1); \;,\qquad
\tikz[baseline=-0.6cm] \draw[kernel1] (0,0) -- (0,-1); \to \tikz[baseline=-0.6cm] \draw (0,0) to node[labl] {\tiny 1,1} (0,-1); \;,\qquad
\tikz[baseline=-0.6cm] \draw[kernel2] (0,0) -- (0,-1); \to \tikz[baseline=-0.6cm] \draw (0,0) to node[labl] {\tiny 1,2} (0,-1); \;,\qquad
\tikz[baseline=-0.6cm]{\node at (0,-0.5) [dot] (d) {}; \draw[rho] (0,0) -- (d) -- (0,-1);} \to \tikz[baseline=-0.6cm] \draw (0,0) to node[labl] {\tiny 3+,-1} (0,-1); \;,\qquad
\tikz[baseline=-0.6cm] \draw[kprime] (0,0) -- (0,-1); \to \tikz[baseline=-0.6cm] \draw (0,0) to node[labl] {\tiny 2,0} (0,-1); \;,\qquad
\tikz[baseline=-0.6cm] \draw[kernelBig] (0,0) -- (0,-1); \to \tikz[baseline=-0.6cm] \draw (0,0) to node[labl] {\tiny 4+,-2} (0,-1); \;,
\end{equ}
and the bold green edges denote the distinguished edges $(v_{\star,i},0)$. 
Here, \tikz[baseline=-0.1cm] \draw[kprime] (0,0) -- (1,0); denotes the spatial derivative
of the kernel $K$ and  \tikz[baseline=-0.1cm] \draw[kernelBig] (0,0) -- (1,0);
denotes the kernel $\Ren Q_\eps$ appearing in \eqref{e:defRQeps} below.
The rationale for this is the fact that, for any $d > 0$, one has
\begin{equ}
\|K\|_{1,d} + \|K'\|_{2,d}  < \infty\;, 
\quad \sup_{\eps \in (0,1]} \|\rho^{(2)}_\eps\|_{3,d} < \infty \;,
\quad \sup_{\eps \in (0,1]} \eps^{-\kappa} \|\rho^{(2)}_\eps\|_{3+\kappa,d}
 < \infty\;, \label{e:boundsKernels}
\end{equ}
for every $\kappa \in (0,1)$.

%
%

\subsubsection{Estimating the first term}

Let us now return to our proof of convergence for $\hat \Pi_0^{(\eps)} \<Xi2>$.
It is straightforward to verify by inspection that the graph appearing in \eref{e:IXi2bnd}
does indeed satisfy the assumptions of Theorem~\ref{theo:ultimate}, so that one has a 
bound of the type $\E \bigl|\bigl(\hat \Pi_0^{(\eps)} \<Xi2>\bigr)(\phi_\lambda)\bigr|^2 \lesssim \lambda^{-2}$, uniformly over $\eps, \lambda \in (0,1]$. 

We now claim that once a bound of this type has been obtained, one automatically obtains
convergence to a limiting random variable $\hat \Pi_0 \tau$. Indeed, simply define 
$\hat \Pi_0 \tau$ in the same way as $\hat \Pi_0^{(\eps)}\tau$, but with all edges 
representing $\rho_\eps$ removed. In our case, this yields the identity
\begin{equ}[e:limitXi2]
\bigl(\hat \Pi_0 \<Xi2>\bigr)(\phi_\lambda) \eqdef \;
\begin{tikzpicture}[scale=0.35,baseline=0.3cm]
	\node at (0,-1)  [root] (root) {};
	\node at (-2,1)  [var] (left) {};
	\node at (0,3)  [var] (left1) {};
	
	\draw[testfcn] (left) to  (root);	
	\draw[kernel1] (left1) to (left);
\end{tikzpicture}\;
 - \;
\begin{tikzpicture}[scale=0.35,baseline=0.3cm]
	\node at (0,-1)  [root] (root) {};
	\node at (0,3)  [dot] (top) {};
	
	\draw[testfcn,bend right = 60] (top) to  (root);
	
	\draw[kernel,bend left = 60] (top) to (root);
\end{tikzpicture}\;.
\end{equ}
(As a matter of fact, this definition only yields a random variable 
$\bigl(\hat \Pi_0 \<Xi2>\bigr)(\phi_\lambda)$
for every test function $\phi_\lambda$. The fact that there exists a model-valued random
variable $\hat\Pi_0$ such that $\bigl(\hat \Pi_0 \tau\bigr)(\phi_\lambda)$ agrees with this
almost surely will follow from \cite[Thm~10.7]{Regularity}, once all the relevant 
bounds have been obtained.)
We then note that $\bigl(\hat \Pi_0^{(\eps)} \<Xi2> - \hat \Pi_0 \<Xi2>\bigr)(\phi_\lambda)$
can be decomposed as a sum of terms, each of them looking like \eref{e:decompPiXiTwo}, but with 
some of the edges representing $\rho_\eps$ now representing $\delta$ and exactly one of these edges
representing $\rho_\eps - \delta$. As a consequence of the last bound in \eref{e:boundsKernels},
we then obtain immediately the bound
\begin{equ}[e:IXi2bnddiff]
\eps^{-\kappa} \E \bigl|\bigl(\hat \Pi_0^{(\eps)} \<Xi2> - \hat \Pi_0 \<Xi2>\bigr)(\phi_\lambda)\bigr|^2 \lesssim \;
\begin{tikzpicture}[scale=0.4,baseline=0.3cm]
	\node at (0,-1)  [root] (root) {};
	\node at (-1.5,1)  [dot] (left) {};
	\node at (-1.5,3)  [dot] (left1) {};
	\node at (1.5,1) [dot] (variable1) {};
	\node at (1.5,3) [dot] (variable2) {};
	
	\draw[dist] (left) to (root);
	\draw[dist] (variable1) to (root);
	
	\draw[->,generic] (left1) to  node[labl,pos=0.45] {\tiny 1,1} (left);
	\draw[->,generic] (variable2) to node[labl,pos=0.45] {\tiny 1,1} (variable1);
	\draw[generic] (variable2) to node[labl] {\tiny 3+,-1} (left1); 
	\draw[generic] (variable1) to  node[labl] {\tiny 3+,-1} (left); 
\end{tikzpicture}
\; + \;
\left(
\begin{tikzpicture}[scale=0.4,baseline=0.2cm]
	\node at (0,-1)  [root] (root) {};
	\node at (-1.5,2)  [dot] (left) {};
	\node at (1.5,2) [dot] (variable1) {};
	
	\draw[dist] (left) to (root);
	\draw[generic] (variable1) to node[labl] {\tiny 1,0} (root);
	
	\draw[generic] (variable1) to  node[labl] {\tiny 3+,-1} (left); 
\end{tikzpicture}
\right)^2 \lesssim \lambda^{-2-2\kappa}\;,
\end{equ}
where we wrote $3+$ as a shorthand for $3+\kappa$.
Noting that $2|\<Xi2>| = -2-4\kappa$, this does indeed imply 
the bound required in \eref{e:wantedBound} for the particular case $\tau = \<Xi2>$.
The remainder of this section is devoted to the proof of this bound for the 
remaining symbols $\tau$ with $|\tau| < 0$, see the list \eref{e:symbols}.

\begin{remark}
Note that in general one does \textit{not} have the identity 
\begin{equ}[e:wantedWick]
\bigl(\hat \Pi_0 \<Xi2>\bigr)(\phi_\lambda) = \bigl(\hat \Pi_0 \<IXi> \diamond \hat \Pi_0 \<Xi>\bigr)(\phi_\lambda)\;,
\end{equ}
with $\diamond$ the Wick product in the sense of white noise analysis. 
The discrepancy between
the two expressions is the second term in \eref{e:limitXi2}.
If however the support of the test function is located in the future,
then \eqref{e:wantedWick} and therefore \eqref{e:Itointegral} 
(with $\tau = \<IXi>$) \textit{does} hold thanks to the fact that the 
second term in \eref{e:limitXi2} 
vanishes in that case. This is an immediate consequence of the 
non-anticipative nature of the kernel $K$.
\end{remark}

\subsection{Proof of Theorem~\ref{theo:convModel}}

We now proceed to give the proof of Theorem~\ref{theo:convModel}
for any symbol $\tau$ with $|\tau| < 0$. 
The proof is always essentially the same, so we only give the main steps.

\subsubsection[Terms Xi2X]{Convergence for the symbols \texorpdfstring{$\<Xi2>$}{Xi2}\,, \texorpdfstring{$\<Xi2X>$}{Xi2X}\,,
and \texorpdfstring{$\<XXi2>$}{XXi2}}

In the preceding subsection, we have shown in detail that 
there exists a random distribution $\hat \Pi_0 \<Xi2>$ such that
\begin{equ}
\E \bigl|\bigl(\hat \Pi_0 \<Xi2>\bigr)(\phi^\lambda)\bigr|^2
\lesssim \lambda^{2|\<Xi2s>|+\kappa}\;,\qquad 
\E \bigl|\bigl(\hat \Pi_0 \<Xi2> - \hat \Pi_0^{(\eps)} \<Xi2>\bigr)(\phi^\lambda)\bigr|^2
\lesssim \eps^\kappa\lambda^{2|\<Xi2s>|+\kappa}\;.
\end{equ}
We have furthermore shown that the identity \eqref{e:wantedWick} holds, which
indeed yields \eqref{e:Itointegral} for $\tau = \<IXi>$.
Since $\bigl(\hat \Pi_0^{(\eps)}\<XXi2>\bigr)(\phi) = \bigl(\hat \Pi_0^{(\eps)}\<Xi2>\bigr)(\tilde\phi)$ with $\tilde \phi(t,x) = x\phi(t,x)$, the required properties and bounds for
$\<XXi2>$ follow immediately from those for $\<Xi2>$.

At this stage, we introduce the kernel $Q_\eps$ given by 
\begin{equ}
Q_\eps(z) = K(z)\,\rho^{(2)}_\eps(z)\;,
\end{equ}
and we use the notation \tikz[baseline=-0.1cm] \draw[kernelBig] (0,0) to (1,0); for the renormalised kernel 
\begin{equ}[e:defRQeps]
\Ren Q_\eps(z) = K(z)\,\rho^{(2)}_\eps(z) - C_\eps \delta_0(z)\;.
\end{equ}
Note that $C_\eps$ is precisely the integral of $Q_\eps$ and $Q_\eps$ is an
even function of the spatial variable, so that $\Ren Q_\eps$ annihilates every
polynomial of parabolic degree strictly less than $2$.
Regarding $\<Xi2X>$, we use these notations to obtain the identity
\begin{equ}[e:decompPiXiTwo2]
\bigl(\hat \Pi_0^{(\eps)} \<Xi2X>\bigr)(\phi_\lambda) = \;
\begin{tikzpicture}[scale=0.35,baseline=0.3cm]
	\node at (0,-1)  [root] (root) {};
	\node at (-1,1)  [dot] (left) {};
	\node at (1,1)  [dot] (left1) {};
	\node at (-1,3) [var] (variable1) {};
	\node at (1,3) [var] (variable2) {};
	
	\draw[testfcn] (left) to  (root);
	
	\draw[kernel2] (left1) to (left);
	\draw[rho] (variable2) to (left1); 
	\draw[rho] (variable1) to (left); 
	\draw[multx] (left1) to (root); 
\end{tikzpicture}
\; + \;
\begin{tikzpicture}[scale=0.35,baseline=0.3cm]
	\node at (0,-1)  [root] (root) {};
	\node at (-1,1)  [dot] (left) {};
	\node at (1,1) [dot] (right) {};
	
	\draw[testfcn] (left) to  (root);
	\draw[kernelBig] (left) to (right);
	\draw[multx] (right) to (root);
\end{tikzpicture}
\; - \;
\begin{tikzpicture}[scale=0.35,baseline=0.3cm]
	\node at (0,-1)  [root] (root) {};
	\node at (-1,1)  [dot] (left) {};
	\node at (0,3)  [dot] (top) {};
	\node at (1,1) [dot] (right) {};
	
	\draw[testfcn] (left) to  (root);
	
	\draw[kernelx] (right) to (root);
	\draw[rho] (top) to (right); 
	\draw[rho] (top) to (left); 
\end{tikzpicture}
\; - \;
\begin{tikzpicture}[scale=0.35,baseline=0.3cm]
	\node at (0,-1)  [root] (root) {};
	\node at (-1,1)  [dot] (left) {};
	\node at (0,3)  [dot] (top) {};
	\node at (1,1) [dot] (right) {};
	
	\draw[testfcnx] (left) to  (root);
	
	\draw[kprimex] (right) to (root);
	\draw[rho] (top) to (right); 
	\draw[rho] (top) to (left); 
\end{tikzpicture}\;.
\end{equ}
Here, we used notation \tikz[baseline=-0.1cm] \draw[multx] (0,0) to (1,0); 
for the kernel $(t,x) \mapsto x$ (which is of homogeneity $+1$),
\tikz[baseline=-0.1cm] \draw[kernelx] (0,0) to (1,0); for the kernel
$(t,x) \mapsto xK(t,x)$, and similarly for \tikz[baseline=-0.1cm] \draw[kprimex] (0,0) to (1,0);. 
Note that while the function $(t,x) \mapsto x$ is of course not of compact support,
we can replace it by a compactly supported function independent of $\eps$ without
changing the values of these integrals. We are therefore back in the context of
Theorem~\ref{theo:ultimate} and it is indeed possible to verify that each of these
terms satisfies Assumption~\ref{ass:graph}.

\begin{remark}
Here and below, Assumption~\ref{ass:graph}
can be verified ``by hand'',
but this soon becomes rather tedious. The interested reader will find
a small computer program at the URL \href{http://www.hairer.org/paper/Trees.zip}{http://www.hairer.org/paper/Trees.zip} which verifies that
Assumption~\ref{ass:graph} does indeed hold for all the graphs
for which we make such a claim in this article. 
\end{remark}

Since one has 
\begin{equ}[e:zigzag]
\|\Ren Q_\eps\|_{-4-\kappa} \lesssim \eps^\kappa\;,
\end{equ}
the term including this kernel vanishes in the limit. It follows that one
has the identity
\begin{equ}
\bigl(\hat \Pi_0 \<Xi2X>\bigr)(\phi_\lambda) = \;
\begin{tikzpicture}[scale=0.35,baseline=0.3cm]
	\node at (0,-1)  [root] (root) {};
	\node at (-2,1)  [var] (left) {};
	\node at (0,3)  [var] (left1) {};
	
	\draw[testfcn] (left) to  (root);	
	\draw[kernel2] (left1) to (left);
	\draw[multx] (left1) to (root);
\end{tikzpicture}
\; - \;
\begin{tikzpicture}[scale=0.35,baseline=0.3cm]
	\node at (0,-1)  [root] (root) {};
	\node at (0,3)  [dot] (top) {};
	
	\draw[testfcn,bend right = 60] (top) to  (root);
	
	\draw[kernelx,bend left = 60] (top) to (root);
\end{tikzpicture}
\; - \;
\begin{tikzpicture}[scale=0.35,baseline=0.3cm]
	\node at (0,-1)  [root] (root) {};
	\node at (0,3)  [dot] (top) {};
	
	\draw[testfcnx,bend right = 60] (top) to  (root);
	
	\draw[kprimex,bend left = 60] (top) to (root);
\end{tikzpicture}\;.
\end{equ}
As before, the last two terms vanish if the test function is supported in the future.
The first term on the other hand is easily seen to be equal to 
$\bigl(\hat \Pi_0 \<IXiX>\bigr)(\phi_\lambda)\diamond \bigl(\hat \Pi_0 \<Xi>\bigr)(\phi_\lambda)$,
which yields \eqref{e:Itointegral} in this case.

\subsubsection[Term Xi3]{Convergence for the symbol \texorpdfstring{$\<Xi3>$}{Xi3}}

Regarding the symbol $\<Xi3>$, we combine \eref{e:admissible3} with the
expression \eref{e:Delta} for $\Delta \<Xi3>$ and the expression \eqref{e:deltaM}
for $\DeltaM$ to obtain similarly to before
\begin{equ}
\bigl(\hat \Pi_0^{(\eps)} \<Xi3>\bigr)(\phi_\lambda) = \;
\begin{tikzpicture}[scale=0.35,baseline=0.5cm]
	\node at (0,-1)  [root] (root) {};
	\node at (-2,1)  [dot] (left) {};
	\node at (-2,3)  [dot] (left1) {};
	\node at (-2,5)  [dot] (left2) {};
	\node at (0,1) [var] (variable1) {};
	\node at (0,3) [var] (variable2) {};
	\node at (0,5) [var] (variable3) {};
	
	\draw[testfcn] (left) to  (root);
	
	\draw[kernel1] (left1) to (left);
	\draw[kernel1] (left2) to (left1);
	\draw[rho] (variable3) to (left2); 
	\draw[rho] (variable2) to (left1); 
	\draw[rho] (variable1) to (left); 
\end{tikzpicture}
\; + \;
\left(
\begin{tikzpicture}[scale=0.35,baseline=0.5cm]
	\node at (0,-1)  [root] (root) {};
	\node at (-2,1)  [dot] (left) {};
	\node at (-2,3)  [dot] (left1) {};
	\node at (-2,5)  [dot] (left2) {};
	\node at (0,2) [dot] (variable2) {};
	\node at (0,5) [var] (variable3) {};
	
	\draw[testfcn] (left) to  (root);
	
	\draw[kernel1] (left1) to (left);
	\draw[kernel1] (left2) to (left1);
	\draw[rho] (variable3) to (left2); 
	\draw[rho] (variable2) to (left1); 
	\draw[rho] (variable2) to (left); 
\end{tikzpicture}
\; - \;
\begin{tikzpicture}[scale=0.35,baseline=0.5cm]
	\node at (0,-1)  [root] (root) {};
	\node at (-2,1)  [dot] (left) {};
	\node at (-1,3)  [dot] (left1) {};
	\node at (-2,5)  [dot] (left2) {};
	\node at (0,2) [dot] (variable2) {};
	\node at (0,5) [var] (variable3) {};
	
	\draw[testfcn] (left) to  (root);
	
	\draw[kernel] (left1) to (left);
	\draw[kernel1] (left2) to (left);
	\draw[rho] (variable3) to (left2); 
	\draw[rho] (variable2) to (left1); 
	\draw[rho] (variable2) to (left); 
\end{tikzpicture}
\right)
\; + \;
\begin{tikzpicture}[scale=0.35,baseline=0.5cm]
	\node at (0,-1)  [root] (root) {};
	\node at (-2,1)  [dot] (left) {};
	\node at (-2,3)  [dot] (jnct) {};
	\node at (-0.75,3)  [dot] (left1) {};
	\node at (-2,5)  [dot] (left2) {};
	\node at (0.5,3) [var] (variable2) {};
	
	\draw[testfcn] (left) to  (root);
	
	\draw[kernel1] (left1) to (left);
	\draw[kernel1] (left2) to (left1);
	\draw[rho] (jnct) to (left2); 
	\draw[rho] (variable2) to (left1); 
	\draw[rho] (jnct) to (left); 
\end{tikzpicture}
\; - \;
\begin{tikzpicture}[scale=0.35,baseline=0.3cm]
	\node at (0,-1)  [root] (root) {};
	\node at (-1,1)  [dot] (left) {};
	\node at (-0.5,2)  [dot] (topl) {};
	\node at (1,1) [dot] (right) {};
	\node at (2,3)  [var] (variable) {};
	\node at (0,3)  [dot] (topr) {};
	
	\draw[testfcn] (right) to  (root);
	
	\draw[kernel] (left) to (root);
	\draw[rho] (topr) to (topl); 
	\draw[rho] (topl) to (left); 
	\draw[kernel1] (topr) to (right); 
	\draw[rho] (right) to (variable); 
\end{tikzpicture}\;.
\end{equ}
 With the above notation for $\Ren Q_\eps$, one then has
\begin{equ}
\bigl(\hat \Pi_0^{(\eps)} \<Xi3>\bigr)(\phi_\lambda) = \;
\begin{tikzpicture}[scale=0.35,baseline=0.5cm]
	\node at (0,-1)  [root] (root) {};
	\node at (-2,1)  [dot] (left) {};
	\node at (-2,3)  [dot] (left1) {};
	\node at (-2,5)  [dot] (left2) {};
	\node at (0,1) [var] (variable1) {};
	\node at (0,3) [var] (variable2) {};
	\node at (0,5) [var] (variable3) {};
	
	\draw[testfcn] (left) to  (root);
	
	\draw[kernel1] (left1) to (left);
	\draw[kernel1] (left2) to (left1);
	\draw[rho] (variable3) to (left2); 
	\draw[rho] (variable2) to (left1); 
	\draw[rho] (variable1) to (left); 
\end{tikzpicture}
\; + \;
\begin{tikzpicture}[scale=0.35,baseline=0.5cm]
	\node at (0,-1)  [root] (root) {};
	\node at (-2,1)  [dot] (left) {};
	\node at (-2,3)  [dot] (left1) {};
	\node at (-2,5)  [dot] (left2) {};
	\node at (0,5) [var] (variable3) {};
	
	\draw[testfcn] (left) to  (root);
	
	\draw[kernelBig] (left1) to (left);
	\draw[kernel1] (left2) to (left1);
	\draw[rho] (variable3) to (left2); 
\end{tikzpicture}
\; - \;
\begin{tikzpicture}[scale=0.35,baseline=0.5cm]
	\node at (0,-1)  [root] (root) {};
	\node at (-2,1)  [dot] (left) {};
	\node at (-2,3)  [dot] (left1) {};
	\node at (-2,5)  [dot] (left2) {};
	\node at (0,3) [dot] (variable2) {};
	\node at (0,5) [var] (variable3) {};
	
	\draw[testfcn] (left) to  (root);
	
	\draw[kernel1] (left2) to (variable2);
	\draw[rho] (variable3) to (left2); 
	\draw[rho] (left) to (left1); 
	\draw[rho] (variable2) to (left1); 
	\draw[kernel] (variable2) to (root); 
\end{tikzpicture}
\; + \;
\begin{tikzpicture}[scale=0.35,baseline=0.5cm]
	\node at (0,-1)  [root] (root) {};
	\node at (-2,1)  [dot] (left) {};
	\node at (-2,3)  [dot] (jnct) {};
	\node at (-0.75,3)  [dot] (left1) {};
	\node at (-2,5)  [dot] (left2) {};
	\node at (0.5,3) [var] (variable2) {};
	
	\draw[testfcn] (left) to  (root);
	
	\draw[kernel1] (left1) to (left);
	\draw[kernel1] (left2) to (left1);
	\draw[rho] (jnct) to (left2); 
	\draw[rho] (variable2) to (left1); 
	\draw[rho] (jnct) to (left); 
\end{tikzpicture}
\; - \;
\begin{tikzpicture}[scale=0.35,baseline=0.3cm]
	\node at (0,-1)  [root] (root) {};
	\node at (-1,1)  [dot] (left) {};
	\node at (-0.5,2)  [dot] (topl) {};
	\node at (1,1) [dot] (right) {};
	\node at (2,3)  [var] (variable) {};
	\node at (0,3)  [dot] (topr) {};
	
	\draw[testfcn] (right) to  (root);
	
	\draw[kernel] (left) to (root);
	\draw[rho] (topr) to (topl); 
	\draw[rho] (topl) to (left); 
	\draw[kernel1] (topr) to (right); 
	\draw[rho] (right) to (variable); 
\end{tikzpicture}\;.
\end{equ}
As before, one can verify that each of these terms separately satisfies 
Assumption~\ref{ass:graph} (after associating to them a labelled graph via the 
procedure outlined in Section~\ref{sec:labGraphs}), 
so that they satisfy the bounds \eqref{e:wantedBound}.
Since furthermore $\Ren Q_\eps \to 0$ and $\rho_\eps \to \delta$ in the 
distributional sense as $\eps \to 0$, one obtains in the limit 
\begin{equ}[e:exprXi3]
\bigl(\hat \Pi_0 \<Xi3>\bigr)(\phi_\lambda) \eqdef \;
\begin{tikzpicture}[scale=0.35,baseline=0.25cm]
	\node at (0,-1)  [root] (root) {};
	\node at (-2,0.4)  [var] (left) {};
	\node at (0,1.6)  [var] (left1) {};
	\node at (-2,3)  [var] (left2) {};
	
	\draw[testfcn] (left) to  (root);
	
	\draw[kernel1] (left1) to (left);
	\draw[kernel1] (left2) to (left1);
\end{tikzpicture}
\; - \;
\begin{tikzpicture}[scale=0.35,baseline=0.25cm]
	\node at (0,-1)  [root] (root) {};
	\node at (0,1)  [dot] (top) {};
	\node at (0,3) [var] (variable) {};
	
	\draw[testfcn,bend left = 60] (top) to  (root);
	\draw[kernel,bend right = 60] (top) to (root);
	\draw[kernel1] (variable) to (top); 
\end{tikzpicture}
\; + \;
\begin{tikzpicture}[scale=0.35,baseline=0.25cm]
	\node at (0,-1)  [root] (root) {};
	\node at (0,1)  [dot] (top) {};
	\node at (0,3) [var] (variable) {};
	
	\draw[testfcn] (top) to  (root);
	\draw[kernel1,bend left = 60] (top) to (variable);
	\draw[kernel1,bend left = 60] (variable) to (top); 
\end{tikzpicture}
\; - \;
\begin{tikzpicture}[scale=0.35,baseline=0.25cm]
	\node at (0,-1)  [root] (root) {};
	\node at (-1,1)  [dot] (left) {};
	\node at (0,3) [var] (right) {};
	
	\draw[testfcn,bend left = 60] (right) to  (root);
	
	\draw[kernel] (left) to (root);
	\draw[kernel1] (left) to (right); 
\end{tikzpicture}\;.
\end{equ}
One might think that the penultimate term in this expression vanishes since the kernel
$K$ is non-anticipative. This would indeed be the case if this term were equal to
\begin{tikzpicture}[scale=0.35,baseline=-0.1cm]
	\node at (0,0)  [root] (root) {};
	\node at (2,0)  [dot] (top) {};
	\node at (4,0) [var] (variable) {};
	
	\draw[testfcn] (top) to  (root);
	\draw[kernel,bend left = 60] (top) to (variable);
	\draw[kernel,bend left = 60] (variable) to (top); 
\end{tikzpicture}. In our case however, it does not vanish in general, unless
the test function is supported in the future.

It remains to show \eqref{e:Itointegral}, namely that if the test function $\phi$ 
has support located in the future, then
\begin{equ}
\bigl(\hat \Pi_0 \<Xi3>\bigr)(\phi) = \bigl(\hat \Pi_0 \<IXi2> \diamond \hat \Pi_0 \<Xi>\bigr)(\phi)\;.
\end{equ}
For this, we note that, as a consequence of \eqref{e:admissible}, 
$\bigl(\hat \Pi_0 \<IXi2>\bigr)(\phi)$ is obtained from $\bigl(\hat \Pi_0 \<Xi2>\bigr)(\phi)$
by simply replacing 
\tikz[baseline=-0.1cm] \draw[testfcn] (0,0) to (1,0);
by 
\tikz[baseline=-0.1cm]{\node at (0,0)  [dot] (root) {}; \draw[testfcn] (root) to (1,0);\draw[kernel1] (-1,0) to (root);}. Taking the Wick product with $\hat \Pi_0 \<Xi>$ (which is nothing but the
underlying white noise $\xi$) then has the effect of simply further replacing 
\tikz[baseline=-0.1cm]{\node at (0,0)  [dot] (root) {}; \draw[testfcn] (root) to (1,0);\draw[kernel1] (-1,0) to (root);} by \tikz[baseline=-0.1cm]{\node at (0,0)  [var] (root) {}; \draw[testfcn] (root) to (1,0);\draw[kernel1] (-1,0) to (root);}. In other words, we obtain
\begin{equ}
\bigl(\hat \Pi_0 \<IXi2> \diamond \hat \Pi_0 \<Xi>\bigr)(\phi)
= \begin{tikzpicture}[scale=0.35,baseline=0.25cm]
	\node at (0,-1)  [root] (root) {};
	\node at (-2,0.4)  [var] (left) {};
	\node at (0,1.6)  [var] (left1) {};
	\node at (-2,3)  [var] (left2) {};
	
	\draw[testfcn] (left) to  (root);
	
	\draw[kernel1] (left1) to (left);
	\draw[kernel1] (left2) to (left1);
\end{tikzpicture}
\; - \;
\begin{tikzpicture}[scale=0.35,baseline=0.25cm]
	\node at (0,-1)  [root] (root) {};
	\node at (-1,1)  [dot] (left) {};
	\node at (0,3) [var] (right) {};
	
	\draw[testfcn,bend left = 60] (right) to  (root);
	
	\draw[kernel] (left) to (root);
	\draw[kernel1] (left) to (right); 
\end{tikzpicture}\;,
\end{equ}
so that the difference between 
$\bigl(\hat \Pi_0 \<Xi3>\bigr)(\phi_\lambda)$ and $\bigl(\hat \Pi_0 \<IXi2> \diamond \hat \Pi_0 \<Xi>\bigr)(\phi_\lambda)$ is given by the second and third terms in \eqref{e:exprXi3}.
As before, the second term vanishes
if the test function $\phi_\lambda$ is supported in the future.
The reason why this is also true for the third term is as follows. By definition, 
one has the identity
\begin{equ}[e:explodeloop]
\begin{tikzpicture}[scale=0.35,baseline=0.25cm]
	\node at (0,-1)  [root] (root) {};
	\node at (0,1)  [dot] (top) {};
	\node at (0,3) [var] (variable) {};
	
	\draw[testfcn] (top) to  (root);
	\draw[kernel1,bend left = 60] (top) to (variable);
	\draw[kernel1,bend left = 60] (variable) to (top); 
\end{tikzpicture}
=
\begin{tikzpicture}[scale=0.35,baseline=0.25cm]
	\node at (0,-1)  [root] (root) {};
	\node at (0,1)  [dot] (top) {};
	\node at (0,3) [var] (variable) {};
	
	\draw[testfcn] (top) to  (root);
	\draw[kernel,bend left = 60] (top) to (variable);
	\draw[kernel,bend left = 60] (variable) to (top); 
\end{tikzpicture}
\;-\;
\begin{tikzpicture}[scale=0.35,baseline=0.25cm]
	\node at (0,-1)  [root] (root) {};
	\node at (0,1)  [dot] (top) {};
	\node at (0,3) [var] (variable) {};
	
	\draw[testfcn] (top) to  (root);
	\draw[kernel] (top) to (variable);
	\draw[kernel,bend left = 60] (variable) to (root); 
\end{tikzpicture}
\;-\;
\begin{tikzpicture}[scale=0.35,baseline=0.25cm]
	\node at (0,-1)  [root] (root) {};
	\node at (0,1)  [dot] (top) {};
	\node at (0,3) [var] (variable) {};
	
	\draw[testfcn] (top) to  (root);
	\draw[kernel,bend right = 60] (top) to (root);
	\draw[kernel] (variable) to (top); 
\end{tikzpicture}
\;+\;
\begin{tikzpicture}[scale=0.35,baseline=0.25cm]
	\node at (0,-1)  [root] (root) {};
	\node at (0,1)  [dot] (top) {};
	\node at (0,3) [var] (variable) {};
	
	\draw[testfcn] (top) to  (root);
	\draw[kernel,bend right = 60] (top) to (root);
	\draw[kernel,bend left = 40] (variable) to (root); 
\end{tikzpicture}\;.
\end{equ}
We now see that each of these terms contains a closed loop with all arrows pointing
in the same direction. Since each of these arrows depicts either $K$ or $\phi$, both
of which are supported in the future, this implies that the corresponding integrands
vanish identically.

\subsubsection[Term Xi3b]{Convergence for the symbol \texorpdfstring{$\<Xi3b>$}{Xib3}}

We now turn to $\<Xi3b>$. In this case, we obtain from 
 \eref{e:admissible3}  and \eref{e:Delta} the identity
\begin{equ}
\bigl(\hat \Pi_0^{(\eps)} \<Xi3b>\bigr)(\phi_\lambda) = \;
\begin{tikzpicture}[scale=0.35,baseline=0.4cm]
	\node at (0,-1)  [root] (root) {};
	\node at (0,1)  [dot] (root2) {};
	\node at (-1.5,2)  [dot] (left) {};
	\node at (1.5,2)  [dot] (right) {};
	\node at (0,3) [var] (variable) {};
	\node at (-1.5,4) [var] (variablel) {};
	\node at (1.5,4) [var] (variabler) {};
	
	\draw[testfcn] (root2) to  (root);
	
	\draw[kernel1] (left) to (root2);
	\draw[kernel1] (right) to (root2);
	\draw[rho] (variable) to (root2); 
	\draw[rho] (variablel) to (left); 
	\draw[rho] (variabler) to (right); 
\end{tikzpicture}
\; - 2 \;
\begin{tikzpicture}[scale=0.35,baseline=0.4cm]
	\node at (0,-1)  [root] (root) {};
	\node at (0,1.5)  [dot] (root2) {};
	\node at (0,4)  [var] (left) {};
	\node at (1.5,0.5)  [dot] (right) {};
	\node at (1.5,3) [dot] (variabler) {};
	
	\draw[testfcn] (root2) to  (root);
	
	\draw[kernel1] (left) to (root2);
	\draw[kernel] (right) to (root);
	\draw[rho] (variable) to (root2); 
	\draw[rho] (variabler) to (root2); 
	\draw[rho] (variabler) to (right); 
\end{tikzpicture}
\; + \;
\begin{tikzpicture}[scale=0.35,baseline=0.4cm]
	\node at (0,-1)  [root] (root) {};
	\node at (0,1)  [dot] (root2) {};
	\node at (-1.5,2.5)  [dot] (left) {};
	\node at (1.5,2.5)  [dot] (right) {};
	\node at (0,2.5) [var] (variable) {};
	\node at (0,4) [dot] (top) {};
	
	\draw[testfcn] (root2) to  (root);
	
	\draw[kernel1] (left) to (root2);
	\draw[kernel1] (right) to (root2);
	\draw[rho] (variable) to (root2); 
	\draw[rho] (top) to (left); 
	\draw[rho] (top) to (right); 
\end{tikzpicture}
\;.
\end{equ}
Performing the substitutions \eqref{e:subs}, it is straightforward to verify that 
the labelled graphs arising from these three expressions from the procedure outlined
in Section~\ref{sec:labGraphs}
all satisfy Assumption~\ref{ass:graph},
so that the required bound \eqref{e:wantedBound} holds.
It follows that the limit as $\eps\to 0$ is given by
\begin{equ}[e:Xi3b]
\bigl(\hat \Pi_0 \<Xi3b>\bigr)(\phi_\lambda) = \;
\begin{tikzpicture}[scale=0.35,baseline=0.25cm]
	\node at (0,-1)  [root] (root) {};
	\node at (0,1)  [var] (root2) {};
	\node at (-1,3)  [var] (left) {};
	\node at (1,3)  [var] (right) {};
	
	\draw[testfcn] (root2) to  (root);
	
	\draw[kernel1] (left) to (root2);
	\draw[kernel1] (right) to (root2);
\end{tikzpicture}
\; - 2 \;
\begin{tikzpicture}[scale=0.35,baseline=0.25cm]
	\node at (0,-1)  [root] (root) {};
	\node at (0,1)  [dot] (top) {};
	\node at (0,3) [var] (variable) {};
	
	\draw[testfcn,bend left = 60] (top) to  (root);
	\draw[kernel,bend right = 60] (top) to (root);
	\draw[kernel1] (variable) to (top); 
\end{tikzpicture}
\; + \;
\begin{tikzpicture}[scale=0.35,baseline=0.25cm]
	\node at (0,-1)  [root] (root) {};
	\node at (0,1)  [var] (root2) {};
	\node at (0,3) [dot] (top) {};
	
	\draw[testfcn] (root2) to  (root);
	
	\draw[kernel1,bend left = 60] (top) to (root2);
	\draw[kernel1,bend right = 60] (top) to (root2);
\end{tikzpicture}
\;.
\end{equ}
Similarly, one sees that
\begin{equ}
\bigl(\hat \Pi_0 \<IXi^2>\bigr)(\phi_\lambda) = \;
\begin{tikzpicture}[scale=0.35,baseline=0.25cm]
	\node at (0,-1)  [root] (root) {};
	\node at (0,1)  [dot] (root2) {};
	\node at (-1,3)  [var] (left) {};
	\node at (1,3)  [var] (right) {};
	
	\draw[testfcn] (root2) to  (root);
	
	\draw[kernel1] (left) to (root2);
	\draw[kernel1] (right) to (root2);
\end{tikzpicture}
\; + \;
\begin{tikzpicture}[scale=0.35,baseline=0.25cm]
	\node at (0,-1)  [root] (root) {};
	\node at (0,1)  [dot] (root2) {};
	\node at (0,3) [dot] (top) {};
	
	\draw[testfcn] (root2) to  (root);
	
	\draw[kernel1,bend left = 60] (top) to (root2);
	\draw[kernel1,bend right = 60] (top) to (root2);
\end{tikzpicture}
\;.
\end{equ}
It follows that the difference between 
$\bigl(\hat \Pi_0 \<Xi3b>\bigr)(\phi_\lambda)$ and $\bigl(\hat \Pi_0 \<IXi^2> \diamond \hat \Pi_0 \<Xi>\bigr)(\phi_\lambda)$ is given by the second term in \eqref{e:Xi3b}, which again vanishes
if $\phi_\lambda$ is supported in the future.

\subsubsection[Term Xi4]{Convergence for the symbol \texorpdfstring{$\<Xi4>$}{Xi4}}

The renormalisation of this term involves the two constants $c_\rho^{(1,1)}$
and $c_\rho^{(1,2)}$. We note here that the notations used in \eqref{e:defrhopicture}
are slightly inconsistent from the ones employed here: dotted lines there denote
the convolution of $\rho$ with itself instead of $\rho_\eps$, arrows denote the heat
kernel $P$ instead of $K$, and \tikz[scale=0.35,baseline=-0.1cm]{\node at (0,0) [dot] (2) {};
\node at (2.2,0) [dot] (4) {};
\draw[kernel,bend right=60] (2) to (4);
\draw[rho,bend left=60] (2) to (4);
\node at (1.1,0) {$\Ren$};} denotes the kernel $\Ren Q_1$ (i.e. $\Ren Q_\eps$
with $\eps = 1$). It is however straightforward to verify that this integral is invariant
under rescaling of the variables by a factor $\eps$ so that, with current pictorial
notations, one has
\begin{equ}
c_\rho^{(1,1)} \approx
\begin{tikzpicture}[baseline=10,scale=0.5]
\node at (1,2) [dot] (1) {};
\node at (-1,2) [dot] (2) {};
\node at (1,0) [dot] (3) {};
\node at (1,1) [dot] (3b) {};
\node at (-1,0) [root] (4) {};
\node at (-1,1) [dot] (4b) {};

\draw[kernel] (1) to (2);
\draw[kernel] (2) to (3);
\draw[kernel] (3) to (4);

\draw[rho] (1) to (3b);
\draw[rho] (3b) to (3);
\draw[rho] (2) to (4b);
\draw[rho] (4b) to (4);
\end{tikzpicture}\;, \quad
c_\rho^{(1,2)} \approx
\begin{tikzpicture}[baseline=10,scale=0.5]
\node at (-1,2) [dot] (1) {};
\node at (1,2) [dot] (2) {};
\node at (-1,0) [root] (3) {};
\node at (-1,1) [dot] (3b) {};
\node at (1,0) [dot] (4) {};

\draw[kernel] (1) to (2);
\draw[kernel] (4) to (3);
\draw[kernelBig] (2) to (4);

\draw[rho] (1) to (3b);
\draw[rho] (3b) to (3);
\end{tikzpicture}\;,
\end{equ}
which would actually be identities if the arrows denoted the heat kernel without truncation. 
It is then a consequence of the convergence shown in Section~\ref{sec:convConst} 
that the error $E_\eps$ implicit in the $\approx$ signs appearing in these expressions
converges to $0$ as $\eps \to 0$.

Assume from now on without loss of generality that $\int \phi^\lambda(z)\,dz = 1$.
Combining this with the recursive definition of $\hat \Pi_0^{(\eps)} \<Xi4>$, we then obtain
\begin{equs}  
\bigl(\hat \Pi_0^{(\eps)} \<Xi4>\bigr)(\phi_\lambda) &= \; \label{e:exprXi4}
\begin{tikzpicture}[scale=0.35,baseline=0.5cm]
	\node at (0,-1)  [root] (root) {};
	\node at (-2,1)  [dot] (left) {};
	\node at (-2,3)  [dot] (left1) {};
	\node at (-2,5)  [dot] (left2) {};
	\node at (-2,7)  [dot] (left3) {};
	\node at (0,1) [var] (variable1) {};
	\node at (0,3) [var] (variable2) {};
	\node at (0,5) [var] (variable3) {};
	\node at (0,7) [var] (variable4) {};
	
	\draw[testfcn] (left) to  (root);
	
	\draw[kernel2] (left1) to (left);
	\draw[kernel1] (left2) to (left1);
	\draw[kernel1] (left3) to (left2);
	\draw[rho] (variable4) to (left3);
	\draw[rho] (variable3) to (left2); 
	\draw[rho] (variable2) to (left1); 
	\draw[rho] (variable1) to (left); 
\end{tikzpicture}
\;-\;
\begin{tikzpicture}[scale=0.35,baseline=0.5cm]
	\node at (0,-1)  [root] (root) {};
	\node at (-2,1)  [dot] (left) {};
	\node at (-2,3)  [dot] (left1) {};
	\node at (-2,5)  [dot] (left2) {};
	\node at (0,5)  [dot] (right2) {};
	\node at (-2,7)  [dot] (left3) {};
	\node at (-4,1) [var] (variable) {};
	\node at (0,7) [var] (variable3) {};

	\draw[testfcn] (left) to  (root);
	
	\draw[kernel2] (left1) to (left);
	\draw[rho] (left2) to (left1);
	\draw[kernel1] (left3) to (right2);
	\draw[kernel1] (left3) to (right2);
	\draw[kernel] (right2) to (root);
	\draw[rho] (left2) to (right2);
	\draw[rho] (variable3) to (left3); 
	\draw[rho] (variable) to (left); 
\end{tikzpicture}
\;+\;
\begin{tikzpicture}[scale=0.35,baseline=0.5cm]
	\node at (0,-1)  [root] (root) {};
	\node at (-2,1)  [dot] (left) {};
	\node at (-2,3)  [dot] (left1) {};
	\node at (-2,5)  [dot] (left2) {};
	\node at (-2,7)  [dot] (left3) {};
	\node at (0,1) [var] (variable1) {};
	\node at (-4,5) [dot] (exleft) {};
	\node at (0,5) [var] (variable3) {};

	\draw[testfcn] (left) to  (root);
	
	\draw[kernel2] (left1) to (left);
	\draw[kernel1] (left2) to (left1);
	\draw[kernel1] (left3) to (left2);
	\draw[rho] (variable3) to (left2); 
	\draw[rho] (exleft) to (left1); 
	\draw[rho] (exleft) to (left3);
	\draw[rho] (variable1) to (left); 
\end{tikzpicture}
\;-\;
\begin{tikzpicture}[scale=0.35,baseline=0.5cm]
	\node at (0,-1)  [root] (root) {};
	\node at (-2,1)  [dot] (left) {};
	\node at (-2,3)  [dot] (left1) {};
	\node at (-2,5)  [dot] (left2) {};
	\node at (0,3) [dot] (right1) {};
	\node at (0,5) [dot] (right2) {};
	\node at (-4,1) [var] (variable) {};
	\node at (-4,3) [var] (variable1) {};
	
	\draw[testfcn] (left) to  (root);
	
	\draw[kernel2] (left1) to (left);
	\draw[kernel1] (left2) to (left1);
	\draw[rho] (right2) to (left2);
	\draw[rho] (right2) to (right1); 
	\draw[rho] (variable1) to (left1); 
	\draw[rho] (variable) to (left); 
	\draw[kernel] (right1) to (root);
\end{tikzpicture}
\\&\;+\;
\begin{tikzpicture}[scale=0.35,baseline=0.5cm]
	\node at (0,-1)  [root] (root) {};
	\node at (-2,1)  [dot] (left) {};
	\node at (-2,3)  [dot] (left1) {};
	\node at (-2,5)  [dot] (left2) {};
	\node at (-2,7)  [dot] (left3) {};
	\node at (0,1) [var] (variable) {};
	\node at (0,7) [var] (variable3) {};
	
	\draw[testfcn] (left) to  (root);
	
	\draw[kernel2] (left1) to (left);
	\draw[kernelBig] (left2) to (left1);
	\draw[kernel1] (left3) to (left2);
	\draw[rho] (variable3) to (left3); 
	\draw[rho] (variable) to (left); 
\end{tikzpicture}
\;+\;
\begin{tikzpicture}[scale=0.35,baseline=0.5cm]
	\node at (0,-1)  [root] (root) {};
	\node at (-2,1)  [dot] (left) {};
	\node at (-2,3)  [dot] (left1) {};
	\node at (-2,5)  [dot] (left2) {};
	\node at (-2,7)  [dot] (left3) {};
	\node at (0,5) [var] (variable2) {};
	\node at (0,7) [var] (variable3) {};
	
	\draw[testfcn] (left) to  (root);
	
	\draw[kernelBig] (left1) to (left);
	\draw[kernel1] (left2) to (left1);
	\draw[kernel1] (left3) to (left2);
	\draw[rho] (variable3) to (left3); 
	\draw[rho] (variable2) to (left2); 
	\end{tikzpicture}
\;-\;
\begin{tikzpicture}[scale=0.35,baseline=0.5cm]
	\node at (0,-1)  [root] (root) {};
	\node at (-2,1)  [dot] (left) {};
	\node at (-2,3)  [dot] (left1) {};
	\node at (-2,5)  [dot] (left2) {};
	\node at (-2,7)  [dot] (left3) {};
	\node at (0,3) [dot] (right1) {};
	\node at (0,5) [var] (variable2) {};
	\node at (0,7) [var] (variable3) {};
	
	\draw[testfcn] (left) to  (root);
	
	\draw[rho] (left1) to (left);
	\draw[kernel1] (left2) to (right1);
	\draw[kernel1] (left3) to (left2);
	\draw[kernel] (right1) to (root);
	\draw[rho] (right1) to (left1);
	\draw[rho] (variable3) to (left3); 
	\draw[rho] (variable2) to (left2); 
	\end{tikzpicture}
\;-\;
\begin{tikzpicture}[scale=0.35,baseline=0.5cm]
	\node at (0,-1)  [root] (root) {};
	\node at (-2,1)  [dot] (left) {};
	\node at (-2,3)  [dot] (left1) {};
	\node at (-2,5)  [dot] (left2) {};
	\node at (-2,7)  [dot] (left3) {};
	\node at (0,3) [dot] (right1) {};
	\node at (0,5) [var] (variable2) {};
	\node at (0,7) [var] (variable3) {};
	
	\draw[testfcnx] (left) to  (root);
	
	\draw[rho] (left1) to (left);
	\draw[kernel1] (left2) to (right1);
	\draw[kernel1] (left3) to (left2);
	\draw[kprime] (right1) to (root);
	\draw[rho] (right1) to (left1);
	\draw[rho] (variable3) to (left3); 
	\draw[rho] (variable2) to (left2); 
	\end{tikzpicture}
\;+\;
\begin{tikzpicture}[scale=0.35,baseline=0.5cm]
	\node at (0,-1)  [root] (root) {};
	\node at (-2,1)  [dot] (left) {};
	\node at (-2,3)  [dot] (left1) {};
	\node at (-2,5)  [dot] (left2) {};
	\node at (-2,7)  [dot] (left3) {};
	\node at (0,3) [var] (variable2) {};
	\node at (-3,3) [dot] (exleft) {};
	\node at (0,7) [var] (variable4) {};

	\draw[testfcn] (left) to  (root);
	
	\draw[kernel2] (left1) to (left);
	\draw[kernel1] (left2) to (left1);
	\draw[kernel1] (left3) to (left2);
	\draw[rho] (variable4) to (left3); 
	\draw[rho] (exleft) to (left); 
	\draw[rho] (exleft) to (left2);
	\draw[rho] (variable2) to (left1); 
\end{tikzpicture}
\;+\;
\begin{tikzpicture}[scale=0.35,baseline=0.5cm]
	\node at (0,-1)  [root] (root) {};
	\node at (-2,1)  [dot] (left) {};
	\node at (-2,3)  [dot] (left1) {};
	\node at (-2,5)  [dot] (left2) {};
	\node at (-2,7)  [dot] (left3) {};
	\node at (0,3) [var] (variable2) {};
	\node at (-3,4) [dot] (exleft) {};
	\node at (0,5) [var] (variable3) {};

	\draw[testfcn] (left) to  (root);
	
	\draw[kernel2] (left1) to (left);
	\draw[kernel1] (left2) to (left1);
	\draw[kernel1] (left3) to (left2);
	\draw[rho] (variable3) to (left2); 
	\draw[rho] (exleft) to (left); 
	\draw[rho] (exleft) to (left3);
	\draw[rho] (variable2) to (left1); 
\end{tikzpicture}
\\&
\;+\;
\begin{tikzpicture}[scale=0.35,baseline=0.5cm]
	\node at (0,-1)  [root] (root) {};
	\node at (-2,1)  [dot] (left) {};
	\node at (-2,3)  [dot] (left1) {};
	\node at (-2,5)  [dot] (left2) {};
	\node at (-2,7)  [dot] (left3) {};
	\node at (0,5) [dot] (variable2) {};
	\node at (0,7) [dot] (variable3) {};

	\draw[testfcn] (left) to  (root);
	
	\draw[rho] (left1) to (left);
	\draw[rho] (left2) to (left1);
	\draw[kernel1] (left3) to (left2);
	\draw[rho] (variable3) to (left3); 
	\draw[rho] (variable2) to (variable3); 
	\draw[kernel] (left2) to (root);
	\draw[kernel] (variable2) to (root); 
\end{tikzpicture}
\;+\;
\begin{tikzpicture}[scale=0.35,baseline=0.5cm]
	\node at (0,-1)  [root] (root) {};
	\node at (-2,1)  [dot] (left) {};
	\node at (-2,3)  [dot] (left1) {};
	\node at (-2,5)  [dot] (left2) {};
	\node at (-2,7)  [dot] (left3) {};
	\node at (0,5) [dot] (variable2) {};
	\node at (0,7) [dot] (variable3) {};

	\draw[testfcnx] (left) to  (root);
	
	\draw[rho] (left1) to (left);
	\draw[rho] (left2) to (left1);
	\draw[kernel1] (left3) to (left2);
	\draw[rho] (variable3) to (left3); 
	\draw[rho] (variable2) to (variable3); 
	\draw[kprime] (left2) to (root);
	\draw[kernel] (variable2) to (root); 
\end{tikzpicture}
\;-\;
\begin{tikzpicture}[scale=0.35,baseline=0.5cm]
	\node at (0,-1)  [root] (root) {};
	\node at (-2,1)  [dot] (left) {};
	\node at (-2,3)  [dot] (left1) {};
	\node at (-2,5)  [dot] (left2) {};
	\node at (0,3) [dot] (right2) {};
	\node at (0,5) [dot] (right3) {};

	\draw[testfcn] (left) to  (root);
	
	\draw[kernelBig] (left1) to (left);
	\draw[kernel1] (left2) to (left1);
	\draw[rho] (right3) to (left2); 
	\draw[rho] (right2) to (right3); 
	\draw[kernel] (right2) to (root);
\end{tikzpicture}
\;-\;
\begin{tikzpicture}[scale=0.35,baseline=0.5cm]
	\node at (0,-1)  [root] (root) {};
	\node at (-2,1)  [dot] (left) {};
	\node at (-2,3)  [dot] (left1) {};
	\node at (-2,5)  [dot] (left2) {};
	\node at (-2,7)  [dot] (left3) {};
	\node at (-4,4) [dot] (exleft) {};
	\node at (0,5) [dot] (right) {};
	
	\draw[testfcn] (left) to  (root);
	
	\draw[kernel2] (left1) to (left);
	\draw[kernel1] (left3) to (right);
	\draw[kernel] (right) to (root);
	\draw[rho] (left2) to (left1);
	\draw[rho] (left2) to (right);
	\draw[rho] (exleft) to (left); 
	\draw[rho] (exleft) to (left3);
\end{tikzpicture}
\\&
\;+\;
\begin{tikzpicture}[scale=0.35,baseline=0.5cm]
	\node at (0,-1)  [root] (root) {};
	\node at (-2,1)  [dot] (left) {};
	\node at (-2,3)  [dot] (left1) {};
	\node at (0,5)  [dot] (left2) {};
	\node at (0,7)  [dot] (left3) {};
	\node at (0,3) [dot] (right) {};
	\node at (-2,5) [dot] (exleft) {};
	
	\draw[testfcn] (left) to  (root);
	
	\draw[kernel2] (left1) to (left);
	\draw[kernel1] (left2) to (left1);
	\draw[kernel1] (left3) to (left2);
	\draw[rho] (right) to (left2); 
	\draw[rho] (exleft) to (left1); 
	\draw[rho] (exleft) to (left3);
	\draw[rho] (right) to (left); 
\end{tikzpicture}
\;-\;
\begin{tikzpicture}[scale=0.35,baseline=0.5cm]
	\node at (0,-1)  [root] (root) {};
	\node at (-2,1)  [dot] (left) {};
	\node at (-2,3)  [dot] (left1) {};
	\node at (0,5)  [dot] (left2) {};
	\node at (0,7)  [dot] (left3) {};
	\node at (0,3) [dot] (right) {};
	\node at (-2,5) [dot] (exleft) {};
	
	\draw[testfcn] (left) to  (root);
	
	\draw[kernel] (left1) to (left);
	\draw[kernel] (left2) to (left1);
	\draw[kernel] (left3) to (left2);
	\draw[rho] (right) to (left2); 
	\draw[rho] (exleft) to (left1); 
	\draw[rho] (exleft) to (left3);
	\draw[rho] (right) to (left); 
\end{tikzpicture}
\;+\;
\begin{tikzpicture}[scale=0.35,baseline=0.5cm]
	\node at (0,-1)  [root] (root) {};
	\node at (-2,1)  [dot] (left) {};
	\node at (0,3)  [dot] (left1) {};
	\node at (0,5)  [dot] (left2) {};
	\node at (-2,7)  [dot] (left3) {};
	\node at (-2,4) [dot] (exleft) {};
	
	\draw[testfcn] (left) to  (root);
	
	\draw[kernel2] (left1) to (left);
	\draw[kernelBig] (left2) to (left1);
	\draw[kernel1] (left3) to (left2);
	\draw[rho] (exleft) to (left); 
	\draw[rho] (exleft) to (left3);
\end{tikzpicture}
\;-\;
\begin{tikzpicture}[scale=0.35,baseline=0.5cm]
	\node at (0,-1)  [root] (root) {};
	\node at (-2,1)  [dot] (left) {};
	\node at (0,3)  [dot] (left1) {};
	\node at (0,5)  [dot] (left2) {};
	\node at (-2,7)  [dot] (left3) {};
	\node at (-2,4) [dot] (exleft) {};
	
	\draw[testfcn] (left) to  (root);
	
	\draw[kernel] (left1) to (left);
	\draw[kernelBig] (left2) to (left1);
	\draw[kernel] (left3) to (left2);
	\draw[rho] (exleft) to (left); 
	\draw[rho] (exleft) to (left3);
\end{tikzpicture} 
\;- E_\eps\;,
\end{equs}
where we used the shorthand  \tikz[baseline=-0.1cm] \draw[kprime] (0,0) -- (1,0);
for the kernel $K' = \d_x K$ and  \tikz[baseline=-0.1cm] \draw[testfcnx] (0,0) -- (1,0);
for the test function $(t,x) \mapsto x\phi^\lambda(t,x)$.
Note that if we set $\tilde \phi(t,x) = x \phi(t,x)$, then $\tilde \phi$ is again an 
admissible test function and one has $x\phi^\lambda(t,x) = \lambda \tilde \phi^\lambda(t,x)$.
As a consequence, when applying Theorem~\ref{theo:ultimate} 
to a graph with test function $\tilde \phi$,
one gains an additional power of $\lambda$. This however is exactly compensated by the fact 
that in this case one instance of the kernel $K$ is replaced by $K'$, thus lowering
the total homogeneity of the graph by one.

It is a lengthy but straightforward task to verify that each of the terms appearing on the
first three lines verify Assumption~\ref{ass:graph} when performing the 
``doubling'' procedure of Section~\ref{sec:labGraphs} and the
substitutions \eqref{e:subs} to turn them into labelled graphs,
so that the required bounds hold for them. In order to obtain analogous bounds for the terms
on the last line, one needs to exploit the fact that they create cancellations. More precisely,
we rewrite these terms as
\begin{equ}[e:termsFirstCancel]
\begin{tikzpicture}[scale=0.35,baseline=0.5cm]
	\node at (0,-1)  [root] (root) {};
	\node at (-2,1)  [dot] (left) {};
	\node at (-2,3)  [dot] (left1) {};
	\node at (0,5)  [dot] (left2) {};
	\node at (0,7)  [dot] (left3) {};
	\node at (0,3) [dot] (right) {};
	\node at (-2,5) [dot] (exleft) {};
	
	\draw[testfcn] (left) to  (root);
	
	\draw[kernel] (left1) to (left);
	\draw[kernel] (left2) to (left1);
	\draw[kernel] (left3) to (left2);
	\draw[rho] (right) to (left2); 
	\draw[rho] (exleft) to (left1); 
	\draw[rho] (exleft) to (left3);
	\draw[rho] (right) to (left); 
\end{tikzpicture}
\;-\;
\begin{tikzpicture}[scale=0.35,baseline=0.5cm]
	\node at (0,-1)  [root] (root) {};
	\node at (-2,1)  [dot] (left) {};
	\node at (-2,3)  [dot] (left1) {};
	\node at (0,5)  [dot] (left2) {};
	\node at (0,7)  [dot] (left3) {};
	\node at (0,3) [dot] (right) {};
	\node at (-2,5) [dot] (exleft) {};
	
	\draw[testfcn] (left) to  (root);
	
	\draw[kernel2] (left1) to (left);
	\draw[kernel1] (left2) to (left1);
	\draw[kernel1] (left3) to (left2);
	\draw[rho] (right) to (left2); 
	\draw[rho] (exleft) to (left1); 
	\draw[rho] (exleft) to (left3);
	\draw[rho] (right) to (left); 
\end{tikzpicture}
\;=\;
\begin{tikzpicture}[scale=0.35,baseline=0.5cm]
	\node at (-2,-1)  [root] (root) {};
	\node at (0,1)  [dot] (left) {};
	\node at (-2,3)  [dot] (left1) {};
	\node at (0,5)  [dot] (left2) {};
	\node at (0,7)  [dot] (left3) {};
	\node at (0,3) [dot] (right) {};
	\node at (-2,5) [dot] (exleft) {};
	
	\draw[testfcn] (left) to  (root);
	
	\draw[kernel] (left1) to (root);
	\draw[kernel1] (left2) to (left1);
	\draw[kernel1] (left3) to (left2);
	\draw[rho] (right) to (left2); 
	\draw[rho] (exleft) to (left1); 
	\draw[rho] (exleft) to (left3);
	\draw[rho] (right) to (left); 
\end{tikzpicture}
\;+\;
\begin{tikzpicture}[scale=0.35,baseline=0.5cm]
	\node at (-2,-1)  [root] (root) {};
	\node at (0,1)  [dot] (left) {};
	\node at (-2,3)  [dot] (left1) {};
	\node at (0,5)  [dot] (left2) {};
	\node at (0,7)  [dot] (left3) {};
	\node at (0,3) [dot] (right) {};
	\node at (-2,5) [dot] (exleft) {};
	
	\draw[testfcnx] (left) to  (root);
	
	\draw[kprime] (left1) to (root);
	\draw[kernel1] (left2) to (left1);
	\draw[kernel1] (left3) to (left2);
	\draw[rho] (right) to (left2); 
	\draw[rho] (exleft) to (left1); 
	\draw[rho] (exleft) to (left3);
	\draw[rho] (right) to (left); 
\end{tikzpicture}
\;+\;
\begin{tikzpicture}[scale=0.35,baseline=0.5cm]
	\node at (0,-1)  [root] (root) {};
	\node at (-2,1)  [dot] (left) {};
	\node at (-2,3)  [dot] (left1) {};
	\node at (0,5)  [dot] (left2) {};
	\node at (0,7)  [dot] (left3) {};
	\node at (-1,3) [dot] (right) {};
	\node at (-2,5) [dot] (exleft) {};
	
	\draw[testfcn] (left) to  (root);
	
	\draw[kernel] (left1) to (left);
	\draw[kernel] (left2) to (root);
	\draw[kernel1] (left3) to (left2);
	\draw[rho] (right) to (left2); 
	\draw[rho] (exleft) to (left1); 
	\draw[rho] (exleft) to (left3);
	\draw[rho] (right) to (left); 
\end{tikzpicture}
\;+\;
\begin{tikzpicture}[scale=0.35,baseline=0.5cm]
	\node at (0,-1)  [root] (root) {};
	\node at (-2,1)  [dot] (left) {};
	\node at (-2,3)  [dot] (left1) {};
	\node at (0,5)  [dot] (left2) {};
	\node at (0,7)  [dot] (left3) {};
	\node at (0,3) [dot] (right) {};
	\node at (-2,5) [dot] (exleft) {};
	
	\draw[testfcn] (left) to  (root);
	
	\draw[kernel] (left1) to (left);
	\draw[kernel] (left2) to (left1);
	\draw[kernel,bend left=40] (left3) to (root);
	\draw[rho] (right) to (left2); 
	\draw[rho] (exleft) to (left1); 
	\draw[rho] (exleft) to (left3);
	\draw[rho] (right) to (left); 
\end{tikzpicture}\;,
\end{equ}
as well as
\begin{equ}[e:termsSecondCancel]
\begin{tikzpicture}[scale=0.35,baseline=0.5cm]
	\node at (0,-1)  [root] (root) {};
	\node at (-2,1)  [dot] (left) {};
	\node at (0,3)  [dot] (left1) {};
	\node at (0,5)  [dot] (left2) {};
	\node at (-2,7)  [dot] (left3) {};
	\node at (-2,4) [dot] (exleft) {};
	
	\draw[testfcn] (left) to  (root);
	
	\draw[kernel] (left1) to (left);
	\draw[kernelBig] (left2) to (left1);
	\draw[kernel] (left3) to (left2);
	\draw[rho] (exleft) to (left); 
	\draw[rho] (exleft) to (left3);
\end{tikzpicture} 
\;-\;
\begin{tikzpicture}[scale=0.35,baseline=0.5cm]
	\node at (0,-1)  [root] (root) {};
	\node at (-2,1)  [dot] (left) {};
	\node at (0,3)  [dot] (left1) {};
	\node at (0,5)  [dot] (left2) {};
	\node at (-2,7)  [dot] (left3) {};
	\node at (-2,4) [dot] (exleft) {};
	
	\draw[testfcn] (left) to  (root);
	
	\draw[kernel2] (left1) to (left);
	\draw[kernelBig] (left2) to (left1);
	\draw[kernel1] (left3) to (left2);
	\draw[rho] (exleft) to (left); 
	\draw[rho] (exleft) to (left3);
\end{tikzpicture}
=
\begin{tikzpicture}[scale=0.35,baseline=0.5cm]
	\node at (0,-1)  [root] (root) {};
	\node at (-2,1)  [dot] (left) {};
	\node at (0,3)  [dot] (left1) {};
	\node at (0,5)  [dot] (left2) {};
	\node at (-2,7)  [dot] (left3) {};
	\node at (-2,4) [dot] (exleft) {};
	
	\draw[testfcn] (left) to  (root);
	
	\draw[kernel] (left1) to (root);
	\draw[kernelBig] (left2) to (left1);
	\draw[kernel] (left3) to (left2);
	\draw[rho] (exleft) to (left); 
	\draw[rho] (exleft) to (left3);
\end{tikzpicture}
\;+\;
\begin{tikzpicture}[scale=0.35,baseline=0.5cm]
	\node at (0,-1)  [root] (root) {};
	\node at (-2,1)  [dot] (left) {};
	\node at (0,3)  [dot] (left1) {};
	\node at (0,5)  [dot] (left2) {};
	\node at (-2,7)  [dot] (left3) {};
	\node at (-2,4) [dot] (exleft) {};
	
	\draw[testfcnx] (left) to  (root);
	
	\draw[kprime] (left1) to (root);
	\draw[kernelBig] (left2) to (left1);
	\draw[kernel] (left3) to (left2);
	\draw[rho] (exleft) to (left); 
	\draw[rho] (exleft) to (left3);
\end{tikzpicture}
\;.
\end{equ}
It is then possible to verify that each of these terms appearing in the right hand side of
\eqref{e:termsFirstCancel} as well as the first term appearing in the right hand side of
\eqref{e:termsSecondCancel} all separately give rise to graphs
satisfying Assumption~\ref{ass:graph}.
The second term appearing on the right hand side of \eqref{e:termsSecondCancel} fails Condition~3, 
but it can easily be dealt with ``by hand'': it simply consists of a convolution of (renormalized)
kernels of respective homogeneities $-3$, $-4$ and $-2$, tested against the test function $\lambda\tilde \phi^\lambda$.
These can easily be bounded by repeatedly applying \cite[Lem.~3.14-3.16]{Regularity}.

We now verify that one has
\begin{equ}
\bigl(\hat \Pi_0 \<Xi4>\bigr)(\phi) = \bigl(\hat \Pi_0 \<IXi3> \diamond \hat \Pi_0 \<Xi>\bigr)(\phi)\;,
\end{equ}
for $\phi$ with support in $\R_+ \times S^1$.
Similarly to before, the right hand side is obtained from $\bigl(\hat \Pi_0 \<Xi3>\bigr)(\phi)$
by  replacing 
\tikz[baseline=-0.1cm] \draw[testfcn] (0,0) to (1,0);
with 
\tikz[baseline=-0.1cm]{\node at (0,0)  [var] (root) {}; \draw[testfcn] (root) to (1,0);\draw[kernel2] (-1,0) to (root);} in its pictorial representation.
It then follows immediately from \eqref{e:exprXi3} that this yields
precisely the first four terms in \eqref{e:exprXi4}.
The first three terms on the second line of \eqref{e:exprXi4} contain a factor 
\tikz[baseline=-0.1cm] \draw[kernelBig] (0,0) to (1,0); and satisfy the assumptions 
of Theorem~\ref{theo:ultimate}, so they vanish as $\eps \to 0$ for any test function
as a consequence of \eqref{e:zigzag}.
It remains to show that all remaining terms
vanish as $\eps \to 0$ when the test function $\phi$ is supported in the future.
For the remaining terms on the second and third lines, this 
can be verified by a systematic use of the argument following \eqref{e:explodeloop}.
Regarding the two differences appearing on the last line, they can again be treated by
the same arguments. Note however that it is crucial here to make use of the cancellations
appearing there: the individual terms on the last line do in general \textit{not}
converge to $0$!

\subsubsection[Term Xi4e]{Convergence for the symbol \texorpdfstring{$\<Xi4e>$}{Xi4e}}

In this case, we note that similarly to the previous case one can write
\begin{equ}
c_\rho^{(2,1)} + c_\rho^{(2,2)} = 
\begin{tikzpicture}[baseline=10,scale=0.5]
\node at (1,2) [dot] (1) {};
\node at (-1,2) [dot] (2) {};
\node at (1,0) [dot] (3) {};
\node at (1,1) [dot] (3b) {};
\node at (-1,0) [root] (4) {};
\node at (-1,1) [dot] (4b) {};

\draw[kernel] (1) to (2);
\draw[kernel] (2) to (3);
\draw[kernel] (4) to (3);

\draw[rho] (1) to (3b);
\draw[rho] (3b) to (3);
\draw[rho] (2) to (4b);
\draw[rho] (4b) to (4);
\end{tikzpicture}
\;+\;
\begin{tikzpicture}[baseline=10,scale=0.5]
\node at (-1,2) [dot] (1) {};
\node at (1,2) [dot] (2) {};
\node at (-1,0) [root] (3) {};
\node at (-1,1) [dot] (3b) {};
\node at (1,0) [dot] (4) {};

\draw[kernel] (1) to (2);
\draw[kernel] (3) to (4);
\draw[kernelBig] (2) to (4);

\draw[rho] (1) to (3b);
\draw[rho] (3b) to (3);
\end{tikzpicture}
\;+E_\eps\;,
\end{equ}
for some error $E_\eps$ with $\lim_{\eps \to 0}E_\eps = 0$. With this identity at hand, 
we then obtain
\begin{equs}  
\bigl(\hat \Pi_0^{(\eps)} \<Xi4e>\bigr)(\phi_\lambda) &= \; 
\begin{tikzpicture}[scale=0.35,baseline=0.4cm]
	\node at (0,-1)  [root] (root) {};
	\node at (0,1)  [dot] (root2) {};
	\node at (-1.5,2.5)  [dot] (left) {};
	\node at (1.5,2.5)  [dot] (right) {};
	\node at (0,2.5) [var] (variable) {};
	\node at (-1.5,4) [var] (variablel) {};
	\node at (1.5,4) [var] (variabler) {};
	\node at (0,4) [dot] (top) {};
	\node at (0,5.5) [var] (variablet) {};
	
	\draw[testfcn] (root2) to  (root);
	
	\draw[kernel1] (left) to (root2);
	\draw[kernel1] (right) to (root2);
	\draw[kernel1] (top) to (left);
	\draw[rho] (variable) to (root2); 
	\draw[rho] (variablel) to (left); 
	\draw[rho] (variabler) to (right); 
	\draw[rho] (variablet) to (top); 
\end{tikzpicture}
\;+\;
\begin{tikzpicture}[scale=0.35,baseline=0.4cm]
	\node at (0,-1)  [root] (root) {};
	\node at (0,1)  [dot] (root2) {};
	\node at (-1.5,2.5)  [dot] (left) {};
	\node at (1.5,2.5)  [dot] (right) {};
	\node at (0,2.5) [dot] (variable) {};
	\node at (-1.5,4) [var] (variablel) {};
	\node at (1.5,4) [var] (variabler) {};
	\node at (0,4) [dot] (top) {};
	
	\draw[testfcn] (root2) to  (root);
	
	\draw[kernel1] (left) to (root2);
	\draw[kernel1] (right) to (root2);
	\draw[kernel1] (top) to (left);
	\draw[rho] (variable) to (root2); 
	\draw[rho] (variablel) to (left); 
	\draw[rho] (variabler) to (right); 
	\draw[rho] (top) to (variable); 
\end{tikzpicture}
\;+\;
\begin{tikzpicture}[scale=0.35,baseline=0.4cm]
	\node at (0,-1)  [root] (root) {};
	\node at (0,1)  [dot] (root2) {};
	\node at (-1.5,2.5)  [dot] (left) {};
	\node at (1.5,2.5)  [dot] (right) {};
	\node at (0,2.5) [dot] (center) {};
	\node at (-1.5,1) [var] (variable) {};
	\node at (0,4) [dot] (top) {};
	\node at (0,5.5) [var] (variablet) {};
	
	\draw[testfcn] (root2) to  (root);
	
	\draw[kernel1] (left) to (root2);
	\draw[kernel1] (right) to (root2);
	\draw[kernel1] (top) to (left);
	\draw[rho] (variable) to (root2); 
	\draw[rho] (variablet) to (top); 
	\draw[rho] (left) to (center); 
	\draw[rho] (right) to (center); 
\end{tikzpicture}
\;+\;
\begin{tikzpicture}[scale=0.35,baseline=0.4cm]
	\node at (0,-1)  [root] (root) {};
	\node at (0,1)  [dot] (root2) {};
	\node at (-1.5,2.5)  [dot] (left) {};
	\node at (1.5,2.5)  [dot] (right) {};
	\node at (0,2.5) [var] (variable) {};
	\node at (-1.5,4) [var] (variablel) {};
	\node at (0,4) [dot] (top) {};
	\node at (0.75,3.25) [dot] (center) {};
	
	\draw[testfcn] (root2) to  (root);
	
	\draw[kernel1] (left) to (root2);
	\draw[kernel1] (right) to (root2);
	\draw[kernel1] (top) to (left);
	\draw[rho] (variable) to (root2); 
	\draw[rho] (variablel) to (left); 
	\draw[rho] (center) to (right); 
	\draw[rho] (center) to (top); 
\end{tikzpicture}
\;-\;
\begin{tikzpicture}[scale=0.35,baseline=0.4cm]
	\node at (0,-1)  [root] (root) {};
	\node at (0,1)  [dot] (root2) {};
	\node at (-1.5,2.5)  [dot] (left) {};
	\node at (1.5,2.5)  [dot] (right) {};
	\node at (0.75,1.75) [dot] (center) {};
	\node at (-1.5,4) [var] (variablel) {};
	\node at (0,4) [dot] (top) {};
	\node at (0,5.5) [var] (variablet) {};
	
	\draw[testfcn] (root2) to  (root);
	
	\draw[kernel1] (left) to (root2);
	\draw[kernel,bend left=20] (right) to (root);
	\draw[kernel1] (top) to (left);
	\draw[rho] (center) to (root2); 
	\draw[rho] (variablel) to (left); 
	\draw[rho] (center) to (right); 
	\draw[rho] (variablet) to (top); 
\end{tikzpicture}
\\&
\;-\;
\begin{tikzpicture}[scale=0.35,baseline=0.4cm]
	\node at (0,-1)  [root] (root) {};
	\node at (0,1)  [dot] (root2) {};
	\node at (-1.5,2.5)  [dot] (left) {};
	\node at (1.5,2.5)  [dot] (right) {};
	\node at (0,2.5) [var] (variable) {};
	\node at (1.5,4) [var] (variabler) {};
	\node at (-1.5,4.5) [dot] (top) {};
	\node at (-1.5,3.5) [dot] (center) {};
	
	\draw[testfcn] (root2) to  (root);
	
	\draw[kernel1] (left) to (root2);
	\draw[kernel1] (right) to (root2);
	\draw[kernel,bend right = 40] (top) to (root);
	\draw[rho] (variable) to (root2); 
	\draw[rho] (center) to (left); 
	\draw[rho] (variabler) to (right); 
	\draw[rho] (center) to (top); 
\end{tikzpicture}
\;+\;
\begin{tikzpicture}[scale=0.35,baseline=0.4cm]
	\node at (0,-1)  [root] (root) {};
	\node at (0,1)  [dot] (root2) {};
	\node at (-1.5,2.5)  [dot] (left) {};
	\node at (1.5,2.5)  [dot] (right) {};
	\node at (1.5,4) [var] (variabler) {};
	\node at (0,4) [dot] (top) {};
	\node at (0,5.5) [var] (variablet) {};
	
	\draw[testfcn] (root2) to  (root);
	
	\draw[kernelBig] (left) to (root2);
	\draw[kernel1] (right) to (root2);
	\draw[kernel1] (top) to (left);
	\draw[rho] (variablet) to (top); 
	\draw[rho] (variabler) to (right); 
\end{tikzpicture}
\;-\;
\begin{tikzpicture}[scale=0.35,baseline=0.4cm]
	\node at (0,-1)  [root] (root) {};
	\node at (0,1)  [dot] (root2) {};
	\node at (-1.5,2.5)  [dot] (left) {};
	\node at (-0.75,1.75)  [dot] (center) {};
	\node at (1.5,2.5)  [dot] (right) {};
	\node at (1.5,4) [var] (variabler) {};
	\node at (0,4) [dot] (top) {};
	\node at (0,5.5) [var] (variablet) {};
	
	\draw[testfcn] (root2) to  (root);
	
	\draw[rho] (left) to (center);
	\draw[rho] (center) to (root2);
	\draw[kernel1] (right) to (root2);
	\draw[kernel1] (top) to (left);
	\draw[rho] (variablet) to (top); 
	\draw[rho] (variabler) to (right); 
	\draw[kernel,bend right=20] (left) to (root);
\end{tikzpicture}
\;+\;
\begin{tikzpicture}[scale=0.35,baseline=0.4cm]
	\node at (0,-1)  [root] (root) {};
	\node at (0,1)  [dot] (root2) {};
	\node at (0,3)  [dot] (left) {};
	\node at (1.5,2.5)  [dot] (right) {};
	\node at (0.75,1.75) [dot] (center) {};
	\node at (-0.5,4.5)  [dot] (cl) {};
	\node at (-1,3)  [dot] (ll) {};
	
	\draw[testfcn] (root2) to  (root);
	
	\draw[kernel1] (left) to (root2);
	\draw[kernel,bend left=20] (right) to (root);

	\draw[kernel,bend right=20] (ll) to (root);
	\draw[rho] (left) to (cl); 
	\draw[rho] (ll) to (cl); 
	\draw[rho] (center) to (root2); 
	\draw[rho] (center) to (right); 
\end{tikzpicture}
\;-\;
\begin{tikzpicture}[scale=0.35,baseline=0.4cm]
	\node at (0,-1)  [root] (root) {};
	\node at (0,1)  [dot] (root2) {};
	\node at (-1.5,2.5)  [dot] (left) {};
	\node at (1.5,2.5)  [dot] (right) {};
	\node at (0,4) [dot] (top) {};
	\node at (0.75,3.25) [dot] (center) {};
	\node at (-0.75,1.75) [dot] (cl) {};
	
	\draw[testfcn] (root2) to  (root);
	
	\draw[kernel,bend right=20] (left) to (root);
	\draw[kernel1] (right) to (root2);
	\draw[kernel1] (top) to (left);
	\draw[rho] (cl) to (root2); 
	\draw[rho] (cl) to (left); 
	\draw[rho] (center) to (right); 
	\draw[rho] (center) to (top); 
\end{tikzpicture}
\\&
\;+\;
\begin{tikzpicture}[scale=0.35,baseline=0.4cm]
	\node at (0,-1)  [root] (root) {};
	\node at (0,1)  [dot] (root2) {};
	\node at (-1.5,2.5)  [dot] (left) {};
	\node at (1.5,2.5)  [dot] (right) {};
	\node at (0,4) [dot] (top) {};
	\node at (0.75,3.25) [dot] (center) {};
	
	\draw[testfcn] (root2) to  (root);
	
	\draw[kernel1] (right) to (root2);
	\draw[kernel1] (top) to (left);
	\draw[kernelBig] (left) to (root2); 
	\draw[rho] (center) to (right); 
	\draw[rho] (center) to (top); 
\end{tikzpicture}
\;-\;
\begin{tikzpicture}[scale=0.35,baseline=0.4cm]
	\node at (0,-1)  [root] (root) {};
	\node at (0,1)  [dot] (root2) {};
	\node at (-1.5,2.5)  [dot] (left) {};
	\node at (1.5,2.5)  [dot] (right) {};
	\node at (0,4) [dot] (top) {};
	\node at (0.75,3.25) [dot] (center) {};
	
	\draw[testfcn] (root2) to  (root);
	
	\draw[kernel] (right) to (root2);
	\draw[kernel] (top) to (left);
	\draw[kernelBig] (left) to (root2); 
	\draw[rho] (center) to (right); 
	\draw[rho] (center) to (top); 
\end{tikzpicture}
\;+\;
\begin{tikzpicture}[scale=0.35,baseline=0.4cm]
	\node at (0,-1)  [root] (root) {};
	\node at (0,1)  [dot] (root2) {};
	\node at (-1.5,2.5)  [dot] (left) {};
	\node at (1.5,2.5)  [dot] (right) {};
	\node at (0,3) [dot] (top) {};
	\node at (0,4) [dot] (center) {};
	\node at (0,2) [dot] (cc) {};
	
	\draw[testfcn] (root2) to  (root);
	
	\draw[rho] (cc) to (root2); 
	\draw[rho] (cc) to (top); 

	\draw[kernel1] (right) to (root2);
	\draw[kernel1] (top) to (left);
	\draw[kernel1] (left) to (root2); 
	\draw[rho] (center) to (right); 
	\draw[rho] (center) to (left); 
\end{tikzpicture}
\;-\;
\begin{tikzpicture}[scale=0.35,baseline=0.4cm]
	\node at (0,-1)  [root] (root) {};
	\node at (0,1)  [dot] (root2) {};
	\node at (-1.5,2.5)  [dot] (left) {};
	\node at (1.5,2.5)  [dot] (right) {};
	\node at (0,3) [dot] (top) {};
	\node at (0,4) [dot] (center) {};
	\node at (0,2) [dot] (cc) {};
	
	\draw[testfcn] (root2) to  (root);
	
	\draw[rho] (cc) to (root2); 
	\draw[rho] (cc) to (top); 

	\draw[kernel] (right) to (root2);
	\draw[kernel] (top) to (left);
	\draw[kernel] (left) to (root2); 
	\draw[rho] (center) to (right); 
	\draw[rho] (center) to (left); 
\end{tikzpicture}
\;-\; E_\eps\;.\label{e:exprXi4c}
\end{equs}
Again, each term appearing on the first two lines in this expression gives rise
to a labelled graph satisfying 
Assumption~\ref{ass:graph}. The terms appearing on the last line require
us again to make use of cancellations similarly to what we did for the last two terms 
appearing in the expression for $\bigl(\hat \Pi_0^{(\eps)} \<Xi4>\bigr)(\phi_\lambda)$.
This time, we use the identities
\begin{equs}[e:someMessyStuff]
\begin{tikzpicture}[scale=0.35,baseline=0.4cm]
	\node at (0,-1)  [root] (root) {};
	\node at (0,1)  [dot] (root2) {};
	\node at (-1.5,2.5)  [dot] (left) {};
	\node at (1.5,2.5)  [dot] (right) {};
	\node at (0,4) [dot] (top) {};
	\node at (0.75,3.25) [dot] (center) {};
	
	\draw[testfcn] (root2) to  (root);
	
	\draw[kernel1] (right) to (root2);
	\draw[kernel1] (top) to (left);
	\draw[kernelBig] (left) to (root2); 
	\draw[rho] (center) to (right); 
	\draw[rho] (center) to (top); 
\end{tikzpicture}
\;-\;
\begin{tikzpicture}[scale=0.35,baseline=0.4cm]
	\node at (0,-1)  [root] (root) {};
	\node at (0,1)  [dot] (root2) {};
	\node at (-1.5,2.5)  [dot] (left) {};
	\node at (1.5,2.5)  [dot] (right) {};
	\node at (0,4) [dot] (top) {};
	\node at (0.75,3.25) [dot] (center) {};
	
	\draw[testfcn] (root2) to  (root);
	
	\draw[kernel] (right) to (root2);
	\draw[kernel] (top) to (left);
	\draw[kernelBig] (left) to (root2); 
	\draw[rho] (center) to (right); 
	\draw[rho] (center) to (top); 
\end{tikzpicture}
\;&=\;-\;\begin{tikzpicture}[scale=0.35,baseline=0.4cm]
	\node at (0,-1)  [root] (root) {};
	\node at (0,1)  [dot] (root2) {};
	\node at (0,3)  [dot] (left) {};
	\node at (2,1)  [dot] (right) {};
	\node at (2,3) [dot] (top) {};
	\node at (2,2) [dot] (center) {};
	
	\draw[testfcn] (root2) to  (root);
	
	\draw[kernel] (right) to (root);
	\draw[kernel] (top) to (left);
	\draw[kernelBig] (left) to (root2); 
	\draw[rho] (center) to (right); 
	\draw[rho] (center) to (top); 
\end{tikzpicture}\;,\\
\begin{tikzpicture}[scale=0.35,baseline=0.4cm]
	\node at (0,-1)  [root] (root) {};
	\node at (0,1)  [dot] (root2) {};
	\node at (-1.5,2.5)  [dot] (left) {};
	\node at (1.5,2.5)  [dot] (right) {};
	\node at (0,3) [dot] (top) {};
	\node at (0,4) [dot] (center) {};
	\node at (0,2) [dot] (cc) {};
	
	\draw[testfcn] (root2) to  (root);
	
	\draw[rho] (cc) to (root2); 
	\draw[rho] (cc) to (top); 

	\draw[kernel] (right) to (root2);
	\draw[kernel] (top) to (left);
	\draw[kernel] (left) to (root2); 
	\draw[rho] (center) to (right); 
	\draw[rho] (center) to (left); 
\end{tikzpicture}
\;-\;
\begin{tikzpicture}[scale=0.35,baseline=0.4cm]
	\node at (0,-1)  [root] (root) {};
	\node at (0,1)  [dot] (root2) {};
	\node at (-1.5,2.5)  [dot] (left) {};
	\node at (1.5,2.5)  [dot] (right) {};
	\node at (0,3) [dot] (top) {};
	\node at (0,4) [dot] (center) {};
	\node at (0,2) [dot] (cc) {};
	
	\draw[testfcn] (root2) to  (root);
	
	\draw[rho] (cc) to (root2); 
	\draw[rho] (cc) to (top); 

	\draw[kernel1] (right) to (root2);
	\draw[kernel1] (top) to (left);
	\draw[kernel1] (left) to (root2); 
	\draw[rho] (center) to (right); 
	\draw[rho] (center) to (left); 
\end{tikzpicture}
\;&=\;
\begin{tikzpicture}[scale=0.35,baseline=0.4cm]
	\node at (0,-1)  [root] (root) {};
	\node at (0,1)  [dot] (root2) {};
	\node at (-1.5,2.5)  [dot] (left) {};
	\node at (1.5,2.5)  [dot] (right) {};
	\node at (0,3) [dot] (top) {};
	\node at (0,4) [dot] (center) {};
	\node at (0,2) [dot] (cc) {};
	
	\draw[testfcn] (root2) to  (root);
	
	\draw[rho] (cc) to (root2); 
	\draw[rho] (cc) to (top); 

	\draw[kernel1] (right) to (root2);
	\draw[kernel,bend left=60] (top) to (root);
	\draw[kernel1] (left) to (root2); 
	\draw[rho] (center) to (right); 
	\draw[rho] (center) to (left); 
\end{tikzpicture}
\;+\;
\begin{tikzpicture}[scale=0.35,baseline=0.4cm]
	\node at (0,-1)  [root] (root) {};
	\node at (0,1)  [dot] (root2) {};
	\node at (-1.5,2.5)  [dot] (left) {};
	\node at (1.5,2.5)  [dot] (right) {};
	\node at (0,3) [dot] (top) {};
	\node at (0,4) [dot] (center) {};
	\node at (0,2) [dot] (cc) {};
	
	\draw[testfcn] (root2) to  (root);
	
	\draw[rho] (cc) to (root2); 
	\draw[rho] (cc) to (top); 

	\draw[kernel1] (right) to (root2);
	\draw[kernel] (top) to (left);
	\draw[kernel] (left) to (root); 
	\draw[rho] (center) to (right); 
	\draw[rho] (center) to (left); 
\end{tikzpicture}
\;+\;
\begin{tikzpicture}[scale=0.35,baseline=0.4cm]
	\node at (0,-1)  [root] (root) {};
	\node at (0,1)  [dot] (root2) {};
	\node at (-1.5,2.5)  [dot] (left) {};
	\node at (1.5,2.5)  [dot] (right) {};
	\node at (0,3) [dot] (top) {};
	\node at (0,4) [dot] (center) {};
	\node at (0,2) [dot] (cc) {};
	
	\draw[testfcn] (root2) to  (root);
	
	\draw[rho] (cc) to (root2); 
	\draw[rho] (cc) to (top); 

	\draw[kernel] (right) to (root);
	\draw[kernel] (top) to (left);
	\draw[kernel] (left) to (root2); 
	\draw[rho] (center) to (right); 
	\draw[rho] (center) to (left); 
\end{tikzpicture}
\end{equs}
Similarly to before, the term appearing on the first line does not satisfy the assumptions
of Theorem~\ref{theo:ultimate} but can again be dealt with by repeatedly invoking \cite[Lem.~3.14-3.16]{Regularity}.
All the terms appearing on the right hand side of the second line on the other hand
do give rise to labelled graphs satisfying Assumption~\ref{ass:graph}.

If the test function $\phi$ is supported in the future, then the limit of the right
hand side of \eqref{e:exprXi4c} as $\eps \to 0$ is given by
\begin{equs}
\bigl(\hat \Pi_0 \<Xi4e>\bigr)(\phi_\lambda) &= \; 
\begin{tikzpicture}[scale=0.35,baseline=0.4cm]
	\node at (0,-1)  [root] (root) {};
	\node at (0,1)  [var] (root2) {};
	\node at (-1.5,2.5)  [var] (left) {};
	\node at (1.5,2.5)  [var] (right) {};
	\node at (0,4) [var] (top) {};
	
	\draw[testfcn] (root2) to  (root);
	
	\draw[kernel1] (left) to (root2);
	\draw[kernel1] (right) to (root2);
	\draw[kernel1] (top) to (left);
\end{tikzpicture}
\;+\;
\begin{tikzpicture}[scale=0.35,baseline=0.4cm]
	\node at (0,-1)  [root] (root) {};
	\node at (0,1)  [var] (root2) {};
	\node at (0,3) [dot] (top) {};
	\node at (0,5) [var] (ttop) {};
	
	\draw[testfcn] (root2) to  (root);
	
	\draw[kernel1,bend left=60] (top) to (root2);
	\draw[kernel1,bend right=60] (top) to (root2);
	\draw[kernel1] (ttop) to (top);
\end{tikzpicture}
\;+\;
\begin{tikzpicture}[scale=0.35,baseline=0.4cm]
	\node at (0,-1)  [root] (root) {};
	\node at (0,1)  [var] (root2) {};
	\node at (-1,3) [var] (top) {};
	\node at (1,3) [dot] (ttop) {};
	
	\draw[testfcn] (root2) to  (root);
	
	\draw[kernel1] (top) to (root2);
	\draw[kernel1] (ttop) to (root2);
	\draw[kernel1] (ttop) to (top);
\end{tikzpicture}
\;-\;
\begin{tikzpicture}[scale=0.35,baseline=0.4cm]
	\node at (0,-1)  [root] (root) {};
	\node at (0,1)  [var] (root2) {};
	\node at (-1,3) [var] (top) {};
	\node at (1.5,3) [dot] (ttop) {};
	
	\draw[testfcn] (root2) to  (root);
	
	\draw[kernel1] (top) to (root2);
	\draw[kernel,bend left=20] (ttop) to (root);
	\draw[kernel1] (ttop) to (root2);
\end{tikzpicture}\;.
\end{equs}
As before, it is now a straightforward task to verify that this is indeed
equal to $\bigl(\hat \Pi_0 \<Xi4e'> \diamond \hat \Pi_0 \<Xi>\bigr)(\phi_\lambda)$ as required.

\subsubsection[Term Xi4b]{Convergence for the symbol \texorpdfstring{$\<Xi4b>$}{Xi4b}}

This time, one obtains the identity
\begin{equ}  
\bigl(\hat \Pi_0^{(\eps)} \<Xi4b>\bigr)(\phi_\lambda) = \; 
\begin{tikzpicture}[scale=0.35,baseline=0.4cm]
	\node at (0,-1)  [root] (root) {};
	\node at (0,1)  [dot] (root2) {};
	\node at (-1.5,2.5)  [dot] (left) {};
	\node at (1.5,2.5)  [dot] (right) {};
	\node at (0,3) [dot] (middle) {};
	\node at (-1.5,4) [var] (variablel) {};
	\node at (1.5,4) [var] (variabler) {};
	\node at (0,4.5) [var] (variablem) {};
	
	\node at (-1.5,1) [var] (variable) {};

	\draw[testfcn] (root2) to  (root);
	
	\draw[kernel1] (left) to (root2);
	\draw[kernel1] (right) to (root2);
	\draw[kernel1] (middle) to (root2);
	\draw[rho] (variable) to (root2); 
	\draw[rho] (variablel) to (left); 
	\draw[rho] (variabler) to (right); 
	\draw[rho] (variablem) to (middle); 
\end{tikzpicture}
\;+3\;
\begin{tikzpicture}[scale=0.35,baseline=0.4cm]
	\node at (0,-1)  [root] (root) {};
	\node at (0,1)  [dot] (root2) {};
	\node at (-1.5,2.5)  [dot] (left) {};
	\node at (1.5,2.5)  [dot] (right) {};
	\node at (0,3) [dot] (middle) {};
	\node at (-1.5,4) [var] (variablel) {};
	\node at (0.85,4) [dot] (contr) {};
	
	\node at (-1.5,1) [var] (variable) {};

	\draw[testfcn] (root2) to  (root);
	
	\draw[kernel1] (left) to (root2);
	\draw[kernel1] (right) to (root2);
	\draw[kernel1] (middle) to (root2);
	\draw[rho] (variable) to (root2); 
	\draw[rho] (variablel) to (left); 
	\draw[rho] (contr) to (right); 
	\draw[rho] (contr) to (middle); 
\end{tikzpicture}
\;-3\;
\begin{tikzpicture}[scale=0.35,baseline=0.4cm]
	\node at (0,-1)  [root] (root) {};
	\node at (0,1)  [dot] (root2) {};
	\node at (-0.75,2)  [dot] (leftc) {};
	\node at (-1.55,1)  [dot] (left) {};
	\node at (1.5,2.5)  [dot] (right) {};
	\node at (0,3) [dot] (middle) {};
	\node at (1.5,4) [var] (variabler) {};
	\node at (0,4.5) [var] (variablem) {};
	
	\node at (-1.5,1) [dot] (variable) {};

	\draw[testfcn] (root2) to  (root);
	
	\draw[kernel] (left) to (root);
	\draw[kernel1] (right) to (root2);
	\draw[kernel1] (middle) to (root2);
	\draw[rho] (leftc) to (root2); 
	\draw[rho] (leftc) to (left); 
	\draw[rho] (variabler) to (right); 
	\draw[rho] (variablem) to (middle); 
\end{tikzpicture}
\;+3\;
\begin{tikzpicture}[scale=0.35,baseline=0.4cm]
	\node at (0,-1)  [root] (root) {};
	\node at (0,1)  [dot] (root2) {};
	\node at (-0.75,2)  [dot] (leftc) {};
	\node at (-1.55,1)  [dot] (left) {};
	\node at (1.5,2.5)  [dot] (right) {};
	\node at (0,3) [dot] (middle) {};
	\node at (0.85,4) [dot] (contr) {};
	
	\node at (-1.5,1) [dot] (variable) {};

	\draw[testfcn] (root2) to  (root);
	
	\draw[kernel] (left) to (root);
	\draw[kernel1] (right) to (root2);
	\draw[kernel1] (middle) to (root2);
	\draw[rho] (leftc) to (root2); 
	\draw[rho] (leftc) to (left); 
	\draw[rho] (contr) to (right); 
	\draw[rho] (contr) to (middle); 
\end{tikzpicture}
\;.
\end{equ}
This time, each term generates a graph satisfying Assumption~\ref{ass:graph}, so that the
required bounds and convergence hold. When testing against
a test function that is supported in the future, the limit as $\eps \to 0$ of
this expression becomes
\begin{equ}  
\bigl(\hat \Pi_0^{(\eps)} \<Xi4b>\bigr)(\phi) = \; 
\begin{tikzpicture}[scale=0.35,baseline=0.25cm]
	\node at (0,-1)  [root] (root) {};
	\node at (0,1)  [var] (root2) {};
	\node at (-1.5,2.5)  [var] (left) {};
	\node at (1.5,2.5)  [var] (right) {};
	\node at (0,3) [var] (middle) {};
	
	\draw[testfcn] (root2) to  (root);
	
	\draw[kernel1] (left) to (root2);
	\draw[kernel1] (right) to (root2);
	\draw[kernel1] (middle) to (root2);
\end{tikzpicture}
\;+3\;
\begin{tikzpicture}[scale=0.35,baseline=0.25cm]
	\node at (0,-1)  [root] (root) {};
	\node at (0,1)  [var] (root2) {};
	\node at (-2,1)  [var] (left) {};
	\node at (0,3)  [dot] (contr) {};

	\draw[testfcn] (root2) to  (root);
	
	\draw[kernel1] (left) to (root2);
	\draw[kernel1,bend left=60] (contr) to (root2);
	\draw[kernel1,bend right=60] (contr) to (root2);
\end{tikzpicture}\;.
\end{equ}
As before, one can easily verify that this is equal to 
$\bigl(\hat \Pi_0 \<Xi4b'>\diamond \hat \Pi_0 \<Xi>\bigr)(\phi_\lambda)$, so that \eqref{e:Itointegral} is verified.

\subsubsection[Term Xi4c]{Convergence for the symbol \texorpdfstring{$\<Xi4c>$}{Xi4c}}

For this last term, we have
\begin{equs}  
\bigl(\hat \Pi_0^{(\eps)} \<Xi4c>\bigr)(\phi_\lambda) &= \; 
\begin{tikzpicture}[scale=0.35,baseline=0.4cm]
	\node at (0,-1)  [root] (root) {};
	\node at (0,1)  [dot] (root2) {};
	\node at (0,3)  [dot] (middle) {};
	\node at (0,5)  [dot] (top) {};
	\node at (-2,3)  [dot] (left) {};

	\node at (-2,5)  [var] (varl) {};
	\node at (2,5)  [var] (vart) {};
	\node at (2,3)  [var] (varm) {};
	\node at (2,1)  [var] (varr) {};
	\draw[testfcn] (root2) to  (root);
	
	\draw[kernel2] (middle) to (root2);
	\draw[kernel1] (left) to (middle);
	\draw[kernel1] (top) to (middle);
	\draw[rho] (varr) to (root2); 
	\draw[rho] (varl) to (left); 
	\draw[rho] (vart) to (top); 
	\draw[rho] (varm) to (middle); 
\end{tikzpicture}
\;-2\;
\begin{tikzpicture}[scale=0.35,baseline=0.4cm]
	\node at (0,-1)  [root] (root) {};
	\node at (0,1)  [dot] (root2) {};
	\node at (0,3)  [dot] (middle) {};
	\node at (0,5)  [dot] (top) {};
	\node at (-2,1)  [dot] (left) {};
	\node at (-1,2)  [dot] (leftc) {};

	\node at (2,5)  [var] (vart) {};
	\node at (2,1)  [var] (varr) {};
	\draw[testfcn] (root2) to  (root);
	
	\draw[kernel2] (middle) to (root2);
	\draw[kernel] (left) to (root);
	\draw[kernel1] (top) to (middle);
	\draw[rho] (varr) to (root2); 
	\draw[rho] (leftc) to (left); 
	\draw[rho] (vart) to (top); 
	\draw[rho] (leftc) to (middle); 
\end{tikzpicture}
\;+\;
\begin{tikzpicture}[scale=0.35,baseline=0.4cm]
	\node at (0,-1)  [root] (root) {};
	\node at (0,1)  [dot] (root2) {};
	\node at (0,3)  [dot] (middle) {};
	\node at (0,5)  [dot] (top) {};
	\node at (-2,3)  [dot] (left) {};

	\node at (-2,5)  [dot] (centre) {};
	\node at (2,3)  [var] (varm) {};
	\node at (2,1)  [var] (varr) {};
	\draw[testfcn] (root2) to  (root);
	
	\draw[kernel2] (middle) to (root2);
	\draw[kernel1] (left) to (middle);
	\draw[kernel1] (top) to (middle);
	\draw[rho] (varr) to (root2); 
	\draw[rho] (centre) to (left); 
	\draw[rho] (centre) to (top); 
	\draw[rho] (varm) to (middle); 
\end{tikzpicture}
\\&
\;+2\;
\begin{tikzpicture}[scale=0.35,baseline=0.4cm]
	\node at (0,-1)  [root] (root) {};
	\node at (0,1)  [dot] (root2) {};
	\node at (0,3)  [dot] (middle) {};
	\node at (0,5)  [dot] (top) {};
	\node at (-2,3)  [dot] (left) {};

	\node at (-2,1)  [dot] (center) {};
	\node at (2,5)  [var] (vart) {};
	\node at (2,3)  [var] (varm) {};
	\draw[testfcn] (root2) to  (root);
	
	\draw[kernel2] (middle) to (root2);
	\draw[kernel1] (left) to (middle);
	\draw[kernel1] (top) to (middle);
	\draw[rho] (center) to (root2); 
	\draw[rho] (center) to (left); 
	\draw[rho] (vart) to (top); 
	\draw[rho] (varm) to (middle); 
\end{tikzpicture}
\;+\;
\begin{tikzpicture}[scale=0.35,baseline=0.4cm]
	\node at (0,-1)  [root] (root) {};
	\node at (0,1)  [dot] (root2) {};
	\node at (0,3)  [dot] (middle) {};
	\node at (0,5)  [dot] (top) {};
	\node at (-2,3)  [dot] (left) {};

	\node at (-2,5)  [var] (varl) {};
	\node at (2,5)  [var] (vart) {};
	\draw[testfcn] (root2) to  (root);
	
	\draw[kernelBig] (middle) to (root2);
	\draw[kernel1] (left) to (middle);
	\draw[kernel1] (top) to (middle);
	\draw[rho] (varl) to (left); 
	\draw[rho] (vart) to (top); 
\end{tikzpicture}
\;-\;
\begin{tikzpicture}[scale=0.35,baseline=0.4cm]
	\node at (0,-1)  [root] (root) {};
	\node at (0,1)  [dot] (root2) {};
	\node at (0,3)  [dot] (middle) {};
	\node at (0,5)  [dot] (top) {};
	\node at (-2,3)  [dot] (left) {};
	\node at (-1,2)  [dot] (center) {};

	\node at (-2,5)  [var] (varl) {};
	\node at (2,5)  [var] (vart) {};
	\draw[testfcn] (root2) to  (root);
	
	\draw[rho] (center) to (middle); 
	\draw[rho] (center) to (root2); 

	\draw[kernel,bend left=60] (middle) to (root);
	\draw[kernel1] (left) to (middle);
	\draw[kernel1] (top) to (middle);
	\draw[rho] (varl) to (left); 
	\draw[rho] (vart) to (top); 
\end{tikzpicture}
\;-\;
\begin{tikzpicture}[scale=0.35,baseline=0.4cm]
	\node at (0,-1)  [root] (root) {};
	\node at (0,1)  [dot] (root2) {};
	\node at (0,3)  [dot] (middle) {};
	\node at (0,5)  [dot] (top) {};
	\node at (-2,3)  [dot] (left) {};
	\node at (-1,2)  [dot] (center) {};

	\node at (-2,5)  [var] (varl) {};
	\node at (2,5)  [var] (vart) {};
	\draw[testfcnx] (root2) to  (root);
	
	\draw[rho] (center) to (middle); 
	\draw[rho] (center) to (root2); 

	\draw[kprime,bend left=60] (middle) to (root);
	\draw[kernel1] (left) to (middle);
	\draw[kernel1] (top) to (middle);
	\draw[rho] (varl) to (left); 
	\draw[rho] (vart) to (top); 
\end{tikzpicture}
\\&
\;+\;
\begin{tikzpicture}[scale=0.35,baseline=0.4cm]
	\node at (0,-1)  [root] (root) {};
	\node at (0,1)  [dot] (root2) {};
	\node at (0,3)  [dot] (middle) {};
	\node at (0,5)  [dot] (top) {};
	\node at (-2,3)  [dot] (left) {};

	\node at (-2,5)  [dot] (contr) {};
	\draw[testfcn] (root2) to  (root);
	
	\draw[kernelBig] (middle) to (root2);
	\draw[kernel1] (left) to (middle);
	\draw[kernel1] (top) to (middle);
	\draw[rho] (contr) to (left); 
	\draw[rho] (contr) to (top); 
\end{tikzpicture}
\;-\;
\begin{tikzpicture}[scale=0.35,baseline=0.4cm]
	\node at (0,-1)  [root] (root) {};
	\node at (0,1)  [dot] (root2) {};
	\node at (0,3)  [dot] (middle) {};
	\node at (0,5)  [dot] (top) {};
	\node at (-2,3)  [dot] (left) {};
	\node at (-1,2)  [dot] (center) {};

	\node at (-2,5)  [dot] (contr) {};
	\draw[testfcn] (root2) to  (root);
	
	\draw[rho] (center) to (middle); 
	\draw[rho] (center) to (root2); 

	\draw[kernel,bend left=60] (middle) to (root);
	\draw[kernel1] (left) to (middle);
	\draw[kernel1] (top) to (middle);
	\draw[rho] (contr) to (left); 
	\draw[rho] (contr) to (top); 
\end{tikzpicture}
\;-\;
\begin{tikzpicture}[scale=0.35,baseline=0.4cm]
	\node at (0,-1)  [root] (root) {};
	\node at (0,1)  [dot] (root2) {};
	\node at (0,3)  [dot] (middle) {};
	\node at (0,5)  [dot] (top) {};
	\node at (-2,3)  [dot] (left) {};
	\node at (-1,2)  [dot] (center) {};

	\node at (-2,5)  [dot] (contr) {};
	\draw[testfcnx] (root2) to  (root);
	
	\draw[rho] (center) to (middle); 
	\draw[rho] (center) to (root2); 

	\draw[kprime,bend left=60] (middle) to (root);
	\draw[kernel1] (left) to (middle);
	\draw[kernel1] (top) to (middle);
	\draw[rho] (contr) to (left); 
	\draw[rho] (contr) to (top); 
\end{tikzpicture}
\;-2\;
\begin{tikzpicture}[scale=0.35,baseline=0.4cm]
	\node at (0,-1)  [root] (root) {};
	\node at (0,1)  [dot] (root2) {};
	\node at (0,3)  [dot] (middle) {};
	\node at (0,5)  [dot] (top) {};
	\node at (-2,1)  [dot] (left) {};
	\node at (-1,2)  [dot] (leftc) {};

	\node at (2,3)  [dot] (contr) {};
	\draw[testfcn] (root2) to  (root);
	
	\draw[kernel2] (middle) to (root2);
	\draw[kernel] (left) to (root);
	\draw[kernel1] (top) to (middle);
	\draw[rho] (contr) to (root2); 
	\draw[rho] (leftc) to (left); 
	\draw[rho] (contr) to (top); 
	\draw[rho] (leftc) to (middle); 
\end{tikzpicture}\;.
\end{equs}
It is again a lengthy but ultimately straightforward task to verify that
each term yields a graph satisfying Assumption~\ref{ass:graph}, so that the
required bounds and convergence hold, except for the first term appearing on the last line.
This term however can be rewritten as
\begin{equ}
\begin{tikzpicture}[scale=0.35,baseline=0.4cm]
	\node at (0,-1)  [root] (root) {};
	\node at (0,1)  [dot] (root2) {};
	\node at (0,3)  [dot] (middle) {};
	\node at (0,5)  [dot] (top) {};
	\node at (-2,3)  [dot] (left) {};

	\node at (-2,5)  [dot] (contr) {};
	\draw[testfcn] (root2) to  (root);
	
	\draw[kernelBig] (middle) to (root2);
	\draw[kernel1] (left) to (middle);
	\draw[kernel1] (top) to (middle);
	\draw[rho] (contr) to (left); 
	\draw[rho] (contr) to (top); 
\end{tikzpicture}
\;=\;-2\;
\begin{tikzpicture}[scale=0.35,baseline=0.4cm]
	\node at (0,-1)  [root] (root) {};
	\node at (0,1)  [dot] (root2) {};
	\node at (0,3)  [dot] (middle) {};
	\node at (0,5)  [dot] (top) {};
	\node at (-2,3)  [dot] (left) {};

	\node at (-2,5)  [dot] (contr) {};
	\draw[testfcn] (root2) to  (root);
	
	\draw[kernelBig] (middle) to (root2);
	\draw[kernel] (left) to (root);
	\draw[kernel] (top) to (middle);
	\draw[rho] (contr) to (left); 
	\draw[rho] (contr) to (top); 
\end{tikzpicture}\;,
\end{equ}
thus showing that it is identical to the term appearing on the first line of \eqref{e:someMessyStuff}
and has already been dealt with.

An argument very similar to the ones given previously shows that, 
if we test $\hat \Pi_0^{(\eps)} \<Xi4c>$ agains a test function $\phi$ supported in the future 
and take the limit $\eps \to 0$, we do again obtain the identity \eqref{e:Itointegral}.
This finally concludes the proof of Theorem~\ref{theo:convModel}.

\section{Identification of the limit}

We now show that, when using the It\^o model constructed in Theorem~\ref{theo:convModel}, 
the corresponding solutions constructed in Theorem~\ref{theo:gen} coincide with the 
classical It\^o solutions to the nonlinear stochastic 
heat equation. Let us first state the following fact.

\begin{lemma}\label{lem:indicFcn}
Let $\eta \in \CC^\alpha(\R \times S^1)$ for $\alpha \in (-2,0)$ and let 
$[s,t] \subset \R$. Then, there exists a unique distribution
$\one_{[s,t]}\eta \in \CC^\alpha$ such that 
\begin{equ}
\bigl(\one_{[s,t]}\eta\bigr)(\phi) = \eta(\phi)\;,
\end{equ}
for smooth test functions $\phi$ such that $\supp \phi \subset [s,t] \times S^1$,
and such that furthermore $\bigl(\one_{[s,t]}\eta\bigr)(\phi) = 0$ if 
$\supp \phi \cap [s,t] \times S^1 = \emptyset$. The map
$\eta \mapsto \one_{[s,t]}\eta$ is continuous in $\CC^\alpha$.
\end{lemma}

\begin{proof}
In the Euclidean case this is known, see for example 
\cite[Chap.~4.6.3]{MR1419319}. In the more general parabolic case 
considered here, it is an easy corollary of \cite[Prop.~6.9]{Regularity}. (Noting that the
``effective codimension'' of the hyperplane $\{t = 0\}$ is $2$ in the parabolic case
and $1$ in the Euclidean case.)
\end{proof}

With the help of this lemma, we are able to formulate the following result,
which is the main ingredient in identifying our solution with the classical It\^o solution. Here 
we use again the notation $\CT_\CU$ introduced in \eqref{e:Tdirectsum}.

\begin{theorem}\label{theo:ident}
Let $(\hat \Pi,\hat F)$ be the It\^o model built in Theorem~\ref{theo:convModel} 
and let $V$ be a $\D^{\gamma,\eta}_\CU$-valued random variable
for some $\gamma > {3\over 2}+\kappa$ and $\eta \ge 0$, and let $T > 0$ be a bounded stopping time. 
Assume that the stochastic process $V\colon (t,x) \mapsto V(t,x) \in \CT_\CU$ is adapted to
the filtration $\{\CF_s\}_{s \in \R}$ generated by the underlying space-time white noise and that there exists some $p>2$ such
that $\E \|V\|_{\gamma,\eta;\K}^p < \infty$ for the compact set $\K = [0,T] \times S^1$.

Denote by $\CR$ the reconstruction operator associated to $(\hat \Pi,\hat F)$.
Then, for any smooth function $\psi\colon \R_+ \times S^1 \to \R$ with $\supp \psi \subset (0,\infty) \times S^1$, one has
\begin{equ}
\bigl(\one_{[0,T]}\CR (V\sXi)\bigr)(\psi) = \int_0^T \scal{v(t,\cdot)\psi(t,\cdot) ,dW(t)}\;,
\end{equ}
where $v(t,x) = \scal{\one, V(t,x)}$ denotes the component of $V$ in the subspace spanned by $\1$.
Here, $V\sXi$ denotes the modelled distribution given by $(V\sXi)(t,x) = V(t,x)\sXi$.
\end{theorem}

\begin{remark}
Note that the space $\D^{\gamma,\eta}_\CU$ itself is also random in this statement!
The statement is somewhat surprising since it shows that in this situation $\CR_T(V\sXi)$ only depends
on $v$ and not on any of the higher-order coefficients, which is certainly not the case in general.
\end{remark}

\begin{remark}
In practice, we will apply this theorem to the case $V = \hat G(U)$, where $U$ is the solution
to the fixed point equation \eref{e:FP}.
\end{remark}

\begin{proof}[of Theorem~\ref{theo:ident}]
Note first that, as a consequence of the canonical inclusion 
$\CD^{\gamma,\eta} \subset \CD^{\gamma',\eta}$ for $\gamma' < \gamma$ (obtained
by setting to $0$ all components in $\CT_\alpha$ with $\alpha \in [\gamma',\gamma)$), 
we can assume without loss of generality that 
$\gamma < {7\over 4}$.

Every centered square integrable $\CF$-measurable random variable can be 
represented as a stochastic integral against $W$, as is shown in Lemma 1.1 from \cite{CW}. Hence it suffices 
to show that for any smooth and compactly supported test function $\psi$ and for any
 adapted and uniformly Lipschitz continuous process $ \Psi$ with compact support, 
 one has the identity
\begin{equ}[e:wantedIdIto]
\E \Bigl(\bigl(\one_{[0,T]}\CR V\sXi\bigr)(\psi) \int_0^\infty\scal{ \Psi(t),dW(t)} \Bigr) = \E\scal{v \psi,  \Psi}_T\;,
\end{equ}
where $\scal{\cdot,\cdot}_T$ denotes the scalar product in $L^2([0,T]\times S^1)$.

We rely on the fact that, by \cite[Thm~3.23]{Regularity} and Lemma~\ref{lem:indicFcn}, for any modelled distribution $f\in \CD^{\bar \gamma,\bar \eta}_\Xi$ (the space of elements in 
$\CD^{\bar \gamma,\bar \eta}$ with values in $\CT_\Xi$) with $\bar \gamma > 0$, $\bar \eta > -2$, 
and any smooth test function $\psi$ 
supported in a compact set $\K$ not intersecting $\{(t,x)\,:\, t=0\}$, 
one has the bound
\begin{equ}[e:boundConv]
\bigl|\bigl(\one_{[0,T]}\CR f\bigr)(\psi) - \scal{\one_{[0,T]}\CR^n f, \psi}_T\bigr| \lesssim C(\psi) 2^{-\alpha n} \|f\|_{\bar \gamma,\bar \eta;\K} \bigl(\$\hat \Pi\$_{\bar\gamma} + \$\hat \Gamma\$_{\bar\gamma}\bigr) \;,
\end{equ}
where $\alpha$ is any exponent with $\alpha \in (0,\bar \gamma)$, $\$\hat \Gamma\$_{\bar \gamma}$ was defined in \eqref{e:defNormGamma},
\begin{equ}
\$\hat \Pi\$_{\bar \gamma} = \sup_{z,\lambda,\phi} \sup_{|\tau| < \bar \gamma} \lambda^{-|\tau|} |(\Pi_z \tau)(\phi_z^\lambda)|\;,
\end{equ}
(with the supremum over $z$, $\phi$ and $\lambda$ as in \eqref{e:distPi}) 
and where the sequence of functions $\one_{[0,T]}\CR^n f$ is given by
\begin{equ}[e:defRnf]
\bigl(\one_{[0,T]}\CR^n f\bigr)(z) = \one_{[0,T]}(t)\sum_{\bar z \in \Lambda_\s^n(T)} \bigl(\hat \Pi_{\bar z} f(\bar z)\bigr)(\phi^n_{\bar z})\,\phi^n_{\bar z}(z)\;,
\end{equ}
where $t$ is the time coordinate of $z$.
Here, $\Lambda_\s^n(T)$ denotes the dyadic grid on $[0,T]\times S^1$, 
$\phi$ is the scale function of some
sufficiently smooth multiresolution analysis and, writing $z = (t,x)$, $\bar z = (\bar t, \bar x)$, one sets
\begin{equ}
\phi^n_{\bar z}(z) = 2^{3n \over 2} \phi(2^n(x-\bar x))\phi(2^{2n}(t-\bar t))\;.
\end{equ}
(See also \cite[Section~3.3]{Regularity} for more details.)
Our argument relies on the fact that the choice of multiresolution analysis in this construction is arbitrary. In particular, we can ensure that the support of 
$\phi$ is contained in the
interval $[0,K]$ for some $K>0$ and we make such a choice from now on.

Using \eqref{e:defRnf} with $f = V\sXi$ and then applying \eqref{e:Itointegral}, we now have the identity
\begin{equs}
\scal{\one_{[0,T]}\CR^n V\sXi,\psi} &= \sum_{ z \in \Lambda^n_\s}\bigl(\hat \Pi_{ z} \sXi V(\bar z)\bigr)(\phi^n_{ z})\,\scal{\phi^n_{ z},\psi}_T \\
&= \sum_{ z \in \Lambda^n_\s}\scal{\phi^n_{ z},\psi}_T \int_{t}^\infty \scal{\bigl(\hat \Pi_{ z} V( z)\bigr)(s,\cdot) \phi^n_{ z}(s,\cdot), dW(s)}\;,
\end{equs}
where we wrote $\Lambda^n_\s$ for the dyadic grid on all of $\R_+ \times S^1$.
Note now that if either $t \le T - K2^{-2n}$ or $t \ge T$, we have the identity
\begin{equs}
\scal{\phi^n_{ z},\psi}_T &\int_{t}^\infty \scal{\bigl(\hat \Pi_{ z} V( z)\bigr)(s,\cdot) \phi^n_{ z}(s,\cdot), dW(s)} \\ 
&= \scal{\phi^n_{ z},\psi}_T \int_{t}^T \scal{\bigl(\hat \Pi_{ z} V( z)\bigr)(s,\cdot) \phi^n_{ z}(s,\cdot), dW(s)}\;.
\end{equs}
As a consequence, we can write 
\begin{equ}
\scal{\one_{[0,T]}\CR^n V\sXi,\psi} 
= \sum_{ z \in \Lambda^n_\s}\scal{\phi^n_{ z},\psi}_T \int_{t}^T \scal{\bigl(\hat \Pi_{ z} V( z)\bigr)(s,\cdot) \phi^n_{ z}(s,\cdot), dW(s)} + R_n\;,
\end{equ}
where $R_n$ is a sum of the order of $2^n$ terms, each of which is bounded (in squared expectation) 
by a multiple of $2^{-3n/2}$, so that $\E |R_n|^2 \lesssim 2^{-n}$.

Combining this with \eqref{e:boundConv}, we have now shown that there exists $\alpha > 0$ such that
\begin{equ}
\bigl(\one_{[0,T]}\CR V\sXi\bigr)(\psi) = \sum_{ z \in \Lambda^n_\s}\scal{\phi^n_{ z},\psi}_T \int_{t}^T \scal{\bigl(\hat \Pi_{ z} V( z)\bigr)(s,\cdot) \phi^n_{ z}(s,\cdot), dW(s)} + \tilde R_n\;,
\end{equ}
where $\tilde R_n$ satisfies $\E |\tilde  R_n|^2 \lesssim 2^{-\alpha n}$.
Consequently, using It\^o's isometry, 
the left hand side of \eqref{e:wantedIdIto} equals the limit, as $n \to \infty$, of
\begin{equs}
\E \sum_{ z \in \Lambda^n_\s}\scal{\phi^n_{ z},\psi}_T &\int_{t}^T \scal{\bigl(\hat \Pi_{ z} V( z)\bigr)(s,\cdot) \phi^n_{ z}(s,\cdot), \Psi(s,\cdot)}\,ds \\
&=\E \sum_{ z \in \Lambda^n_\s}\scal{\phi^n_{ z},\psi}_T \scal{\bigl(\hat \Pi_{ z} V( z)\bigr)(\cdot) \phi^n_{ z}, \Psi}_T\;. \label{e:lastBound}
\end{equs} 

At this stage we use the fact that, as a consequence of \cite[Prop.~3.28]{Regularity},
there exists an exponent $\alpha > 0$ such that 
\begin{equ}
\bigl|\bigl(\Pi_{z} V(z)\bigr)(\hat z) - v(z)\bigr| \le C |z-\hat z|_\s^{\alpha} \|V\|_{\gamma;\K} 
\$\Gamma\$_{\gamma}\;,
\end{equ}
where $\K$ is some compact set containing $z$ and $\hat z$. 
Inserting this into \eref{e:lastBound} and exploiting the fact that 
$\{\phi^n_{ z}\}_{z \in \Lambda^n_\s}$ is an approximate resolution of the identity,
it is now straightforward to verify that the limit of \eref{e:lastBound}
as $n \to \infty$ equals the right hand side of \eqref{e:wantedIdIto}
as required.
\end{proof}

As a corollary of this result, it is now straightforward to obtain the announced result,
namely that the solution to our abstract fixed point
problem driven by the It\^o model is nothing but the classical weak solution 
to \eref{e:SPDE} interpreted in the It\^o sense.

\begin{corollary}\label{cor:ItoSol}
Let $(\hat \Pi,\hat F)$ be the It\^o model built in Theorem~\ref{theo:convModel}, let $\eta \in (0,{1\over 2}-\kappa)$,
let $M > 0$,
and let $U \in \CD^{\gamma,0}_\CU$ be the solution to the fixed point problem \eref{e:FP}
given by Theorem~\ref{theo:gen} with respect to $(\hat \Pi,\hat F)$, up to the first (stopping)
time $\tau$
where $u = \CR U$ satisfies $\|u(\tau,\cdot)\|_{\CC^\eta} \ge M$. 

Then, the function $u$ on $[0,\tau]$ coincides almost surely with the unique 
weak local solution to \eref{e:SPDE}.
\end{corollary}

\begin{proof}
Applying the reconstruction operator to both sides of \eqref{e:FP} and using the equivalence
between weak and mild solutions for the heat equation with a distributional inhomogeneity
(see the proof of Theorem 3.2 in \cite{Walsh}), we deduce 
that $u$ satisfies the distributional identity
\begin{equ}
\d_t u = \d_x^2 u  + H(u) + \CR \bigl(\hat G(U)\sXi\bigr)\;,
\end{equ}
on $(0,\tau)\times S^1$. Since we know that $u \in \CC^\eta$, each term appearing in this
identity belongs to $\CC^\alpha$ for some $\alpha > -2$, so we can use again Lemma~\ref{lem:indicFcn}
to multiply this identity with $\one_{[0,t]}$ for any stopping time $t \le \tau$.
Testing against a smooth test function $\psi$ of $x$ and applying Theorem~\ref{theo:ident}, 
we thus obtain the identity
\begin{equs}
\scal{u(t,\cdot),\psi} &= \scal{u_0,\psi} + \int_0^t \scal{u(s,\cdot),\d_x^2 \psi}\,ds
+ \int_0^t \scal{H(u(s,\cdot)), \psi}\,ds \\
&\quad + \int_0^t \scal{ G(u(s,\cdot))\psi , dW(s)}\;.
\end{equs}
This is indeed the definition of a weak solution to \eqref{e:SPDE}, the uniqueness of which
was shown in \cite{Walsh}, see Exercise 3.4.
\end{proof}

\bibliographystyle{./Martin}
\bibliography{./refs}

\end{document}